\documentclass[11pt,a4paper]{article}
\usepackage[T1]{fontenc}
\usepackage{lmodern}
\usepackage{authblk}
\usepackage[a4paper,margin=1in]{geometry}
\usepackage{amsthm,amsmath,amsfonts,amssymb,mathtools}
\usepackage{microtype}
\usepackage[numbers,sort&compress]{natbib}
\usepackage{xurl}
\usepackage[colorlinks=true,linkcolor=blue,citecolor=blue,urlcolor=blue]{hyperref}
\hypersetup{pdftitle={Fluctuations of subgraph counts in exponential random graphs},pdfauthor={Hui-Lin Liu, Song-Hao Liu and Dawei Lu},pdfsubject={Probability theory}}
\allowdisplaybreaks[1]
\theoremstyle{plain}
\newtheorem{theorem}{Theorem}[section]
\newtheorem{lemma}[theorem]{Lemma}

\newtheorem{corollary}{Corollary}[section]
\theoremstyle{definition}
\newtheorem{definition}{Definition}[section]
\newtheorem{remark}{Remark}[section]
\newtheorem*{Notation}{Notation}
\newtheorem*{Organization}{Organization}
\numberwithin{equation}{section}
\DeclareMathOperator{\Var}{Var}
\DeclareMathOperator{\Cov}{Cov}
\DeclareMathOperator{\Lip}{Lip}
\DeclareMathOperator{\sgn}{sgn}
\DeclareMathOperator{\Hom}{Hom}
\newcommand{\E}{\mathbb E}
\newcommand{\Pp}{\mathbb P}
\newcommand{\B}{\mathcal B}
\newcommand{\cH}{\mathbf H}
\newcommand{\cI}{\mathcal I}
\newcommand{\ind}{\mathbf1}
\newcommand{\Phibar}{\overline\Phi}
\newcommand{\dKol}{d_{\mathrm{Kol}}}
\newcommand{\dWass}{d_{\mathrm{Wass}}}
\newcommand{\bbeta}{\boldsymbol{\beta}}
\newcommand{\btheta}{\boldsymbol{\theta}}
\newcommand{\bzeta}{\boldsymbol{\zeta}}
\newcommand{\bh}{\boldsymbol{h}}
\newcommand{\bxi}{\boldsymbol{\xi}}
\newcommand{\bA}{\boldsymbol A}
\newcommand{\ba}{\boldsymbol a}
\newcommand{\bu}{\boldsymbol u}
\newcommand{\proofstep}[1]{%
\par\addvspace{0.3\baselineskip}%
\noindent\textit{Step #1.}\quad\ignorespaces}

\title{Fluctuations of subgraph counts in exponential random graphs}
\author{Hui-Lin Liu\thanks{E-mail: \href{mailto:liuhuilin2@mail.dlut.edu.cn}{\texttt{liuhuilin2@mail.dlut.edu.cn}}.}}
\author{Song-Hao Liu\thanks{E-mail: \href{mailto:liusonghao@dlut.edu.cn}{\texttt{liusonghao@dlut.edu.cn}}.}}
\author{Dawei Lu\thanks{E-mail: \href{mailto:ludawei_dlut@163.com}{\texttt{ludawei\_dlut@163.com}}.}}
\affil{School of Mathematical Sciences, Dalian University of Technology, Dalian, China}

\date{}
\begin{document}
\maketitle
\begin{abstract}
In this paper, we study fluctuations of fixed subgraph counts in dense exponential random graphs on $n$ vertices. We compare each normalized subgraph count $V$ with the normalized edge count $W$: their difference $V-W$ is $O(n^{-1/2})$ in $L^2$, whereas $\mathbb{E}[(V-W)f(W)]=O(n^{-1})$ for every $1$-Lipschitz function $f$. We also obtain $O(n^{-1})$ Kolmogorov and Wasserstein bounds, optimal for the edge count, with relative variance error of the same order. These estimates hold uniformly on compact subsets of the ferromagnetic subcritical and absolute-coefficient Dobrushin regions, and under small fixed signed perturbations of the former. Using this uniform control, we derive a local Gaussian likelihood expansion with rank-one leading Fisher information and two-sided moderate deviations with an explicit cubic correction. The proofs combine two applications of the normal Stein equation, covariance estimates from Glauber dynamics and exponential tilting.
\end{abstract}


\section{Introduction}\label{sec:main}

Exponential random graph models (ERGMs) extend the Erd\H{o}s--R\'enyi model by allowing dependence between edges. In the undirected formulation of Frank and Strauss \cite{FS}, this dependence is described by graph statistics such as star and triangle counts. In this setting, we study how the fluctuations of fixed subgraph counts depend on the edge count, both at the leading scale and after its contribution is subtracted. The resulting estimates also apply to normal approximation, relative tail probabilities and local likelihood ratios.

Let $n\ge2$, let $[n]\coloneqq\{1,\ldots,n\}$ be the vertex set, and let $\cI$ be the unordered pairs of distinct vertices in $[n]$. A configuration $y\in\{0,1\}^{\cI}$ encodes a graph $G_y$, with $y_s=1$ when the vertices in $s$ are joined by an edge, and $y_s=0$ otherwise. Fix a family $\cH=(H_1,\ldots,H_k)$ of simple graphs (undirected graphs without self-loops or multiple edges). Let $v_j=|\mathcal V(H_j)|$ denote the number of vertices of $H_j$. The Hamiltonian $T_{\bbeta}$ is defined by
\begingroup
\thinmuskip=1mu\medmuskip=1mu\thickmuskip=2mu
\begin{equation}
T_{\bbeta}(y)=\sum_{j=1}^k\beta_j n^{2-v_j}|\Hom(H_j,G_y)|,\hspace{0.6em}
p_{\bbeta}(y):=\Pp(Y_{\bbeta}=y)=\frac{e^{T_{\bbeta}(y)}}{Z(\bbeta)},\hspace{0.6em}
Z(\bbeta)=\sum_{y\in\{0,1\}^{\cI}}e^{T_{\bbeta}(y)},
\label{eq:model}
\end{equation}
\endgroup
where $|\Hom(H_j,G_y)|$ counts injective, noninduced embeddings and $Z(\bbeta)$ is the partition function normalizing the associated Gibbs distribution. The factors $n^{2-v_j}$ put the subgraph terms on a common $n^2$ scale, matching the number of potential edges. We take $H_1=K_2$, the graph with two vertices joined by one edge. The remaining energy graphs $H_j$, $j\ge2$, have no isolated vertices and at least two edges. We call the model ferromagnetic when $\beta_j>0$ for $j\ge2$, while the edge parameter $\beta_1$ may have either sign.

Previous work on dense ERGMs addresses free energy, dynamics and fluctuations. The large deviation principle of Chatterjee and Varadhan \cite[Theorem 2.3]{CV} leads to the free-energy variational formula of Chatterjee and Diaconis \cite[Theorems 3.1--3.2]{CD}. For the subcritical dynamics, Bhamidi, Bresler and Sly \cite{BBS} proved rapid mixing, and Ganguly and Nam \cite{GN} established concentration and Poincar\'e inequalities. Signed interactions under Dobrushin conditions were treated by Reinert and Ross \cite{RR} and Sambale and Sinulis \cite{SS}, through graph-level approximation and functional inequalities, respectively.

For normal approximation, Fang, Liu, Shao and Zhao \cite{FLSZ} obtained Kolmogorov and Wasserstein bounds for edges and Wasserstein bounds for fixed subgraphs throughout the ferromagnetic subcritical region. Winstein \cite{WQ} improved these rates and treated suitable phase-restricted distributions. Further phase-restricted results include concentration via metastable mixing by Winstein \cite{WMix} and central limit theorems from approximate FKG inequalities by Mukherjee and Winstein \cite{MW}. For specific energy families, Mukherjee and Xu \cite{MX} and Bianchi, Collet and Magnanini \cite{BCM} proved limit theorems. Magnanini and Passuello \cite{MP} proved a triangle central limit theorem for a rounded edge--triangle energy.

These developments motivate the present work in three respects. First, the approximation of a fixed subgraph count by the edge count, used by Sambale and Sinulis \cite{SS}, Fang, Liu, Shao and Zhao \cite{FLSZ}, and Winstein \cite{WQ}, leaves a difference whose size alone does not determine its effect on distributional approximation. We therefore study how this difference interacts with functions of the edge count, alongside sharper normal approximation and variance estimates. Second, quantitative fluctuation results for negative higher-order coefficients remain comparatively limited. Fang, Liu, Shao and Zhao \cite[Remark 3.1(2)]{FLSZ} discuss the role of positivity in their normal approximation argument, while Reinert and Ross \cite{RR} and Sambale and Sinulis \cite{SS} treat signed interactions through graph-level approximation and functional inequalities. This motivates us to study quantitative subgraph fluctuations when the higher-order coefficients are not all positive. Third, relative tail probabilities and local likelihood ratios require control of the model at nearby parameters. In particular, a lower-tail tilt introduces a negative coefficient when the observed subgraph is absent from the energy. Estimates that remain uniform under these parameter changes are therefore needed both for two-sided moderate deviations and for local likelihood expansions.

To specify the two parameter regions considered in this paper, let $e_j=|\mathcal E(H_j)|$ denote the number of edges in $H_j$ and define
\begin{equation}
\Phi_{\bbeta}(u)=\sum_{j=1}^k\beta_j e_j u^{e_j-1},\qquad
\Psi(z)=\frac{e^z}{1+e^z},\qquad
\varphi_{\bbeta}(u)=\Psi(2\Phi_{\bbeta}(u)).
\label{eq:meanfield}
\end{equation}
Here $\Psi$ is the logistic function, and $\varphi_{\bbeta}$ is the mean-field map for the edge density. The equation $u=\varphi_{\bbeta}(u)$ expresses consistency between the assumed homogeneous density and the mean-field conditional edge probability.
\begin{definition}\label{def:parameter-regions}
For $\bbeta=(\beta_1,\ldots,\beta_k)\in\mathbb R^k$, we distinguish the following parameter regions.
\begin{enumerate}
\item[(i)] The \emph{ferromagnetic subcritical region} $\mathcal S_+$ consists of the parameters $\bbeta\in\mathbb R\times(0,\infty)^{k-1}$ for which $u=\varphi_{\bbeta}(u)$ has a unique solution $p_{\bbeta}\in(0,1)$ and $\rho_{\bbeta}:=1-\varphi_{\bbeta}'(p_{\bbeta})>0$.
\item[(ii)] The \emph{absolute-coefficient Dobrushin region} is
\begin{equation}
D(\bbeta)=\frac12\sum_{j=2}^k|\beta_j|e_j(e_j-1),\qquad
\mathcal D=\{\bbeta\in\mathbb R^k:D(\bbeta)<1\}.
\label{eq:dobrushin}
\end{equation}
Higher-order coefficients in $\mathcal D$ may have either sign.
\end{enumerate}
\end{definition}
Throughout, a \emph{unique stable fixed point} means that $u=\varphi_{\bbeta}(u)$ has exactly one solution in $[0,1]$, denoted by $p_{\bbeta}\in(0,1)$, and that $\rho_{\bbeta}=1-\varphi_{\bbeta}'(p_{\bbeta})>0$. On $\mathcal S_+$, the derivative satisfies $0\le\varphi_{\bbeta}'(p_{\bbeta})<1$, so the unique fixed point is locally attracting. This fixed-point criterion defines the usual ferromagnetic high-temperature regime. By \eqref{eq:meanfield} and $\Psi'\le1/4$, we have $\sup_{0\le u\le1}|\varphi_{\bbeta}'(u)|\le D(\bbeta)$. Thus every $\bbeta\in\mathcal D$ has a unique fixed point $p_{\bbeta}$, with $\rho_{\bbeta}=1-\varphi_{\bbeta}'(p_{\bbeta})\ge1-D(\bbeta)>0$. The positive-coefficient part of $\mathcal D$ is therefore contained in $\mathcal S_+$. We use $p_{\bbeta}$ for the scalar fixed point and $p_{\bbeta}(y)$ for the configuration probability in \eqref{eq:model}.

With these parameter regions specified, we now summarize our main results. We first compare each normalized subgraph count $V$ with the normalized edge count $W$, uniformly on compact subsets of the two regions in Definition~\ref{def:parameter-regions}. Their difference $V-W$ is $O(n^{-1/2})$ in $L^2$, whereas $\mathbb{E}[(V-W)f(W)]=O(n^{-1})$ for every $1$-Lipschitz function $f$ (Theorem~\ref{thm:residual}). The second estimate retains information that the $L^2$ bound alone does not provide. Consequently, in every joint subsequential limit of $W$ and $\sqrt n(V-W)$, the limiting rescaled difference has conditional mean zero given the limiting edge fluctuation (Corollary~\ref{cor:conditional-centering}).

For normal approximation, we obtain $O(n^{-1})$ Kolmogorov and Wasserstein bounds, with relative variance error of the same order (Theorem~\ref{thm:normal}). Both distance rates are optimal for the edge count, as shown by the matching lower bounds in \eqref{stein:lattice-lower}. We also show that the unique stable fixed point and these estimates persist under small fixed signed perturbations of compact ferromagnetic subcritical families (Theorem~\ref{thm:signed}). This allows interactions of either sign without requiring the perturbed models to satisfy the Dobrushin condition.

We then use the parameter-uniform estimates to derive a local Gaussian likelihood expansion with rank-one leading Fisher information (Theorem~\ref{thm:likelihood}) and two-sided moderate deviations with an explicit cubic correction (Theorem~\ref{thm:md}). These applications use the fluctuation estimates at nearby parameters, with exponential tilting providing the connection to relative tail probabilities. All counts are centered at their true finite-sample means.

\begin{Notation}
For a random variable $U$, let $\|U\|_q=(\E|U|^q)^{1/q}$. For a function $F$, the supremum in $\|F\|_\infty$ is taken over its full domain. The letters $C,c$ denote positive constants that may change from line to line. They depend only on the fixed compact parameter set, fixed graph families and observed graphs, and any explicitly indicated $L$ or moment order. They do not depend on the particular parameter, $n$, the host edge or the threshold. We write $A\asymp B$ if $cB\le A\le CB$. The notation $O(r)$ denotes a deterministic remainder bounded in absolute value by $Cr$, with no prescribed sign.
\end{Notation}

\begin{Organization}
The rest of the paper is organized as follows. Section~\ref{sec:results} specifies the normalization and states the main results. Section~\ref{sec:stein} proves the comparison of subgraph counts with the edge count (Theorem~\ref{thm:residual}), the normal approximation (Theorem~\ref{thm:normal}) and stability under signed perturbations (Theorem~\ref{thm:signed}), then derives their joint consequences. Section~\ref{sec:md} proves the local Gaussian likelihood expansion (Theorem~\ref{thm:likelihood}) and the two-sided moderate deviation expansion (Theorem~\ref{thm:md}). The lemmas in Appendix~\ref{sec:ferro} provide variance, covariance and moment estimates used to prove the auxiliary lemmas in Appendix~\ref{sec:auxiliary}. These auxiliary lemmas are stated in Sections~\ref{sec:stein} and \ref{sec:md}.
\end{Organization}

\section{Main results}\label{sec:results}

To state the results, we specify the subgraph counts and their normalization. Write
\begin{equation}
\cI=\{\{u,v\}:1\le u<v\le n\},\qquad N=|\cI|=\binom n2.
\label{eq:edges}
\end{equation}
The edge count is $E(y)=\sum_{s\in\cI}y_s$. We sometimes order $\cI$ and index its coordinates by $1,\ldots,N$. For a fixed simple graph $H$, let $v_H=|\mathcal V(H)|$ and $e_H=|\mathcal E(H)|$ denote its numbers of vertices and edges, respectively, so $v_j=v_{H_j}$ and $e_j=e_{H_j}$. Our counting convention is
\begin{equation}
|\Hom(H,G_y)|
=\sum_{\iota:\mathcal V(H)\hookrightarrow[n]}
\prod_{b\in\mathcal E(H)}y_{\iota(b)}.
\label{eq:count}
\end{equation}
The sum runs over injective maps from $\mathcal V(H)$ to $[n]$. For an edge $b=\{u,v\}\in\mathcal E(H)$, we write $\iota(b)=\{\iota(u),\iota(v)\}$ for its image in $\cI$. There is no restriction on the images of nonedges. In particular, $|\Hom(K_2,G_y)|=2E(y)$. The count is zero when no injection exists, and an empty product is one. Each observed graph $H$ is fixed and has at least one edge; it need not occur in the Hamiltonian, and it may be disconnected or have isolated vertices.

For the edge count, vertex-permutation invariance implies that $\widetilde p_{\bbeta}=\E Y_{s,\bbeta}$ does not depend on $s$. Let $E_{\bbeta}=E(Y_{\bbeta})$ and set
\begin{equation}
\mu_{\bbeta}=\E E_{\bbeta}=N\widetilde p_{\bbeta},\quad
v_{\bbeta}=\frac{p_{\bbeta}(1-p_{\bbeta})}{\rho_{\bbeta}},\quad
\sigma_{\bbeta}^2=Nv_{\bbeta},\quad
W_{\bbeta}=\frac{E_{\bbeta}-\mu_{\bbeta}}{\sigma_{\bbeta}}.
\label{eq:edge-scale}
\end{equation}
Here $v_{\bbeta}$ is the variance density, not the vertex count $v_H$. Compared with the Bernoulli variance $p_{\bbeta}(1-p_{\bbeta})$, the factor $1/\rho_{\bbeta}$ accounts for dependence in the leading variance density. For any fixed observed graph $H$, define
\begin{equation}
\begin{aligned}
\mu_{H,\bbeta}&=\E|\Hom(H,G_{Y_{\bbeta}})|,&
\sigma_{H,\bbeta}&=2e_Hp_{\bbeta}^{e_H-1}n^{v_H-2}\sigma_{\bbeta},\\
W_{H,\bbeta}&=\frac{|\Hom(H,G_{Y_{\bbeta}})|-\mu_{H,\bbeta}}{\sigma_{H,\bbeta}},&
\widehat\sigma_{H,\bbeta}^2&=\Var(|\Hom(H,G_{Y_{\bbeta}})|).
\end{aligned}\label{eq:subgraph-scale}
\end{equation}
We choose the scale in \eqref{eq:subgraph-scale} so that the leading edge-count contribution becomes $W_{\bbeta}$, as quantified in Theorem~\ref{thm:residual} below. When the true variance is positive, let $\widehat W_{H,\bbeta}=(|\Hom(H,G_{Y_{\bbeta}})|-\mu_{H,\bbeta})/\widehat\sigma_{H,\bbeta}$. The counts are thus centered at their true finite-sample means, while the unhatted scales are explicit asymptotic scales. In particular, $W_{K_2,\bbeta}=W_{\bbeta}$ exactly. We indicate the underlying parameter through the random-variable subscript, suppressing it on expectations.

We first compare the normalized subgraph count with the normalized edge count on the same probability space. Let
\begin{equation}\label{eq:residual-definition}
D_{H,\bbeta}=W_{H,\bbeta}-W_{\bbeta}.
\end{equation}
The following result controls both the size of this difference and its expected product with a function of the edge count.

\begin{theorem}\label{thm:residual}
Let $\B$ be an arbitrary fixed nonempty compact subset of $\mathcal S_+$ or of $\mathcal D$, and let $H$ be a fixed graph with $e_H\ge1$. For the difference in \eqref{eq:residual-definition}, there is a constant $C=C_{\B,\cH,H}$ such that, for all $n\ge2$ and $\bbeta\in\B$,
\begin{equation}\label{eq:residual-bounds}
\E D_{H,\bbeta}^{\,2}\le\frac Cn,\qquad
\sup_{\Lip(f)\le1}
\left|\E\bigl[D_{H,\bbeta}f(W_{\bbeta})\bigr]\right|\le\frac Cn.
\end{equation}
In particular, we have
\begin{equation}\label{eq:residual-covariance}
\left|\Cov(D_{H,\bbeta},W_{\bbeta})\right|\le\frac Cn.
\end{equation}
For every continuously differentiable $\psi:\mathbb R\to\mathbb R$ with $\Lip(\psi')\le1$, we also have
\begin{equation}\label{eq:residual-smooth}
\left|\E\psi(W_{H,\bbeta})-\E\psi(W_{\bbeta})\right|\le\frac Cn.
\end{equation}
For $n\ge v_H$, all the bounds remain valid with $W_{H,\bbeta}$ and $W_{\bbeta}$ replaced by $\widehat W_{H,\bbeta}$ and $\widehat W_{K_2,\bbeta}$, and with their difference in place of $D_{H,\bbeta}$.
\end{theorem}

To interpret Theorem~\ref{thm:residual}, first take $f(w)=w$ in \eqref{eq:residual-bounds}. This yields the covariance bound \eqref{eq:residual-covariance}. For true-variance normalization, the corresponding covariance can also be expressed directly in terms of the difference, since
\[
\Cov(\widehat D_H,\widehat W_{K_2,\bbeta})
=-\tfrac12\E\widehat D_H^{\,2},\qquad
\widehat D_H=\widehat W_{H,\bbeta}-\widehat W_{K_2,\bbeta}.
\]
The estimate for general Lipschitz $f$ in \eqref{eq:residual-bounds} contains more information than this covariance identity. It also controls nonlinear functions of the normalized edge count and, by Taylor's formula, yields the smooth comparison \eqref{eq:residual-smooth}.

This smooth comparison does not directly cover indicator tests or arbitrary Lipschitz functions. We therefore apply Stein's method directly to the normalized subgraph count to obtain bounds in the two distances defined below.

Let $Z\sim N(0,1)$, with distribution function $\Phi$, density $\phi$ and upper tail $\Phibar=1-\Phi$. We use
\begin{equation}
\dKol(U,Z)=\sup_z|\Pp(U\le z)-\Phi(z)|,\qquad
\dWass(U,Z)=\sup_{\Lip(\psi)\le1}|\E\psi(U)-\E\psi(Z)|.
\label{eq:distances}
\end{equation}
The supremum defining $\dWass$ is over real $1$-Lipschitz functions; we use this distance for variables $U$ with $\E|U|<\infty$.

\begin{theorem}\label{thm:normal}
Let $\B$ be an arbitrary fixed nonempty compact subset of $\mathcal S_+$ or of $\mathcal D$. For an arbitrary fixed graph $H$ with $e_H\ge1$, there is a constant $C=C_{\B,\cH,H}$ such that, for all $n\ge2$ and $\bbeta\in\B$,
\begin{equation}
\dKol(W_{H,\bbeta},Z)+\dWass(W_{H,\bbeta},Z)
+\left|\frac{\widehat\sigma_{H,\bbeta}^2}{\sigma_{H,\bbeta}^2}-1\right|
\le\frac Cn.
\label{eq:main-normal}
\end{equation}
If $n\ge v_H$, then $\widehat\sigma_{H,\bbeta}>0$, and the same distance bounds hold for $\widehat W_{H,\bbeta}$.
\end{theorem}

\begin{remark}
We compare \eqref{eq:main-normal} with the available quantitative central limit theorems. For the edge count, Fang, Liu, Shao and Zhao \cite[Theorem 3.1 and Corollary 3.1]{FLSZ} obtain $O(n^{-1/4})$ bounds in both distances on $\mathcal S_+$ and improve them to $O(n^{-1/2})$ on $\mathcal S_+\cap\mathcal D$. They also prove an $O(n^{-1/4})$ Wasserstein bound for fixed subgraph counts. Winstein \cite[Theorem 1.1 and Corollary 1.2]{WQ} obtains $O(n^{-1/2+\varepsilon})$, for every fixed $\varepsilon>0$, in both distances for edges and in Wasserstein distance for fixed subgraphs. Fang, Liu, Shao and Zhao use the same injective counting convention as we do, whereas Winstein uses ordinary homomorphism counts in the energy. For the corresponding dense subcritical models, we improve the rates to $O(n^{-1})$ and obtain direct Kolmogorov bounds for arbitrary fixed subgraphs, together with relative variance error of the same order. In particular, for triangles, Sambale and Sinulis \cite[Proposition 3.4]{SS} transfer an edge-count Wasserstein estimate with an additional $O(n^{-1/2})$ term. Our proof estimates the triangle count directly and avoids this loss of precision.

For a concrete comparison, consider the original edge--triangle family $(K_2,K_3)$, where $K_3$ is the triangle. By \eqref{eq:main-normal}, we obtain both distance bounds and the relative variance estimate uniformly on compact subsets of $\mathbb R\times(-1/3,0)$ and of $\mathcal S_+$. In the common parameter range, these estimates quantitatively strengthen the edge-count central limit theorem of Bianchi, Collet and Magnanini \cite[Theorem 3.6]{BCM}. Their parameters $\alpha=6\beta_2$ and $b=2\beta_1$ identify the negative interval with $-2<\alpha<0$.
\end{remark}

The independent-edge case already shows that the rate in Theorem~\ref{thm:normal} is optimal over the model class. When $\beta_j=0$ for $2\le j\le k$, \eqref{eq:model} reduces to $G(n,p)$ with $p=\Psi(2\beta_1)$, and $E_{\bbeta}\sim\operatorname{Bin}(N,p)$. For fixed $p\in(0,1)$, the optimal order is $N^{-1/2}\asymp n^{-1}$ in both distances, as discussed by Fang and Liu \cite[Section 1]{FL2025} for Kolmogorov distance and by Goldstein \cite[Theorem 1.1 and (3)]{Goldstein2010} for Wasserstein distance. The lattice argument leading to \eqref{stein:lattice-lower} also proves matching lower bounds for the interacting models.

Theorem~\ref{thm:normal} allows signed interactions under the Dobrushin condition \eqref{eq:dobrushin}. We next show that small signed perturbations of ferromagnetic subcritical models preserve the unique stable fixed point and fluctuation estimates without requiring that condition. This stability also supports the parameter changes in our likelihood and moderate deviation arguments. For a fixed finite family $\mathcal J=(J_1,\ldots,J_m)$ of graphs without isolated vertices and with at least one edge, put
\begin{equation}
\begin{aligned}
T_{\btheta}(y)&=T_{\bbeta}(y)+\sum_{\ell=1}^m h_\ell n^{2-v_{J_\ell}}|\Hom(J_\ell,G_y)|,\qquad
\btheta=(\bbeta,\bh),\\
\B_{\mathcal J}(\eta)&=\left\{(\bbeta,\bh):\bbeta\in\B,\ \sum_{\ell=1}^m e_{J_\ell}^2|h_\ell|\le\eta\right\}.
\end{aligned}
\label{eq:signed-family}
\end{equation}
An added edge term is combined with the original edge coefficient, and repeated templates may likewise be combined. For the enlarged energy, we use \eqref{eq:meanfield}, \eqref{eq:edge-scale} and \eqref{eq:subgraph-scale} with subscript $\btheta$, so that the center, fixed point and scale all correspond to the perturbed model. With this notation, we can state the stability result and the accompanying uniform estimates.

\begin{theorem}\label{thm:signed}
Let $\B\subset\mathcal S_+$ be an arbitrary fixed nonempty compact set. There exists $\eta>0$, depending only on $\B$ and $\cH$, such that, for every fixed $\mathcal J$ in \eqref{eq:signed-family} and every $\btheta\in\B_{\mathcal J}(2\eta)$, the equation $u=\varphi_{\btheta}(u)$ has a unique solution $p_{\btheta}\in(0,1)$, and this solution satisfies $\rho_{\btheta}:=1-\varphi_{\btheta}'(p_{\btheta})>0$. For an arbitrary fixed observed graph $H$ with $e_H\ge1$, there is $C=C_{\B,\cH,\mathcal J,H}$ such that
\begin{equation}
\dKol(W_{H,\btheta},Z)+\dWass(W_{H,\btheta},Z)
+\left|\frac{\widehat\sigma_{H,\btheta}^2}{\sigma_{H,\btheta}^2}-1\right|
\le\frac Cn,
\qquad \btheta\in\B_{\mathcal J}(2\eta),\quad n\ge2.
\label{eq:normal-signed}
\end{equation}
For $n\ge v_H$, the true variance is positive and the two distance bounds also hold for $\widehat W_{H,\btheta}$. Moreover, all conclusions of Theorem~\ref{thm:residual} hold uniformly on $\B_{\mathcal J}(2\eta)$ after replacing $\bbeta$ by the actual parameter $\btheta$.
\end{theorem}

\begin{remark}
The radius $\eta$ in Theorem~\ref{thm:signed} is independent of $n$, so the added coefficients may remain fixed as the graph size grows. To apply Theorem~\ref{thm:md}, we tilt by an observed subgraph, possibly introducing a negative coefficient for the lower tail. Starting from $\B_{\mathcal J}(\eta)$, we reserve the margin up to $2\eta$ for both tilt directions. Thus \eqref{eq:normal-signed} supplies the normal approximation and variance estimates throughout the change of measure.
\end{remark}

We now combine Theorem~\ref{thm:residual} with the normal approximations in Theorems~\ref{thm:normal} and \ref{thm:signed} to describe joint subsequential limits. The first bound in \eqref{eq:residual-bounds} makes the differences tight after multiplication by $\sqrt n$. The second identifies their conditional means given the limiting edge fluctuation.

\begin{corollary}\label{cor:conditional-centering}
Fix graphs $H^{(1)},\ldots,H^{(m)}$, each with at least one edge. Let $\bbeta_n$ be any parameter sequence in a fixed compact set of Theorem~\ref{thm:residual}, or in a fixed enlarged family of Theorem~\ref{thm:signed}, using the actual enlarged parameter in the latter case. The sequence
\[
\bigl(W_{\bbeta_n},\sqrt nD_{H^{(1)},\bbeta_n},\ldots,
\sqrt nD_{H^{(m)},\bbeta_n}\bigr)
\]
is tight. If along a subsequence
\begin{equation}\label{eq:residual-joint-limit}
\bigl(W_{\bbeta_n},\sqrt nD_{H^{(1)},\bbeta_n},\ldots,
\sqrt nD_{H^{(m)},\bbeta_n}\bigr)
\Longrightarrow (Z,U_1,\ldots,U_m),
\end{equation}
then $Z\sim N(0,1)$, each $U_j$ has a finite second moment, and
\begin{equation}\label{eq:residual-conditional-mean}
\E[U_j\mid Z]=0,\qquad 1\le j\le m.
\end{equation}
The same assertions hold under true-variance normalization for all components and for the edge count, taking $n\ge\max_{1\le j\le m}v_{H^{(j)}}$.
\end{corollary}

The conclusion \eqref{eq:residual-conditional-mean} specifies the conditional mean, but does not assert that $U_j$ and $Z$ are independent. The variable $U_j$ may also be degenerate. To identify its distribution, one must study the remaining fluctuations more closely, as in the conditional two-star central limit theorem of Fang, Liu, Su and Wang \cite{FLSW}.

The common edge-count approximation also has a statistical interpretation at the leading fluctuation scale. Since the scaled counts appearing in the Hamiltonian are sufficient statistics, their covariance matrix is the Fisher information matrix. We use this approximation to identify the leading term of the matrix and expand the likelihood ratio for nearby parameters. For the energy family in \eqref{eq:model}, let
\begin{equation}\label{eq:sufficient-statistics}
\bA_n(y)=\bigl(n^{2-v_j}|\Hom(H_j,G_y)|\bigr)_{j=1}^k,
\qquad
\ba_{\bbeta}=\bigl(e_jp_{\bbeta}^{e_j-1}\bigr)_{j=1}^k.
\end{equation}
Here the vectors are written as columns, with $|\cdot|$ denoting the Euclidean norm and $\|\cdot\|_{\mathrm{op}}$ the induced matrix norm. Differentiating the finite sum defining $Z(\bbeta)$ in \eqref{eq:model}, we obtain the score and Fisher information identities
\begin{equation}\label{eq:score-information}
\nabla_{\bbeta}\log p_{\bbeta}(Y_{\bbeta})=\bA_n(Y_{\bbeta})-\E\bA_n(Y_{\bbeta}),\quad
\mathcal I_n(\bbeta)\coloneqq\Cov(\bA_n(Y_{\bbeta}))=\nabla_{\bbeta}^2\log Z(\bbeta).
\end{equation}
For the perturbed models in \eqref{eq:signed-family}, we use the full list of coefficients and sufficient statistics. Repeated templates may be retained as separate coordinates, in which case the finite-sample information matrix can be singular. The rank-one leading term below reflects the common edge-count fluctuation, whether or not the energy family contains repeated graphs.

\begin{theorem}\label{thm:likelihood}
Let $\B$ be an arbitrary fixed nonempty compact subset of $\mathcal S_+$ or of $\mathcal D$. For the fixed energy family, there is a constant $C$ such that
\begin{equation}\label{eq:information-expansion}
\left\|\mathcal I_n(\bbeta)
-4Nv_{\bbeta}\ba_{\bbeta}\ba_{\bbeta}^{\mathsf T}
\right\|_{\mathrm{op}}\le Cn,
\qquad \bbeta\in\B,\quad n\ge2.
\end{equation}
For every fixed $M>0$, there are constants $C_M$ and $n_0$ such that, for $n\ge n_0$, $\bbeta\in\B$ and $|\bu|\le M$,
\begin{equation}\label{eq:likelihood-expansion}
\begin{aligned}
\log\frac{p_{\bbeta+\bu/n}(Y_{\bbeta})}{p_{\bbeta}(Y_{\bbeta})}
=\frac{\bu^{\mathsf T}\{\bA_n(Y_{\bbeta})-\E\bA_n(Y_{\bbeta})\}}n
-v_{\bbeta}(\ba_{\bbeta}^{\mathsf T}\bu)^2
+\mathcal R_{n,\bbeta}(\bu),\\
\sup_{\bbeta\in\B}\sup_{|\bu|\le M}|\mathcal R_{n,\bbeta}(\bu)|
\le C_M/n.\qquad\qquad\qquad\qquad\qquad
\end{aligned}
\end{equation}
The remainder $\mathcal R_{n,\bbeta}(\bu)$ is deterministic. For fixed $\bbeta$ and $\bu$, as $n\to\infty$,
\begin{equation}\label{eq:likelihood-limit}
\log\frac{p_{\bbeta+\bu/n}(Y_{\bbeta})}{p_{\bbeta}(Y_{\bbeta})}
\Longrightarrow
\sqrt{2v_{\bbeta}}(\ba_{\bbeta}^{\mathsf T}\bu)Z
-v_{\bbeta}(\ba_{\bbeta}^{\mathsf T}\bu)^2.
\end{equation}
Convergence holds jointly for any fixed finite collection of directions, with the same $Z$ in all limits. The bounds \eqref{eq:information-expansion} and \eqref{eq:likelihood-expansion} also hold uniformly for the inner signed family $\B_{\mathcal J}(\eta)$ of Theorem~\ref{thm:signed}, with the actual enlarged parameter and vectors. For each fixed parameter in that family, the same limiting statements hold. The constants may additionally depend on $\mathcal J$, and the shifted parameters are controlled by the outer perturbation bound.
\end{theorem}

\begin{remark}\label{rem:local-information}
By \eqref{eq:edges} and \eqref{eq:information-expansion}, the information divided by $n^2$ converges to $2v_{\bbeta}\ba_{\bbeta}\ba_{\bbeta}^{\mathsf T}$, a matrix of rank one. Differentiating the fixed-point equation in \eqref{eq:meanfield} gives
\begin{equation}\label{eq:meanfield-response}
\nabla_{\bbeta}p_{\bbeta}=2v_{\bbeta}\ba_{\bbeta}.
\end{equation}
Thus a direction enters the limit in \eqref{eq:likelihood-limit} only through its first-order effect on the mean-field edge density. For a direction with $\ba_{\bbeta}^{\mathsf T}\bu=0$, we can also compare the two probability measures directly and obtain
\begin{equation}\label{eq:local-total-variation}
\frac12\sum_{y\in\{0,1\}^{\cI}}
|p_{\bbeta+\bu/n}(y)-p_{\bbeta}(y)|
\le C_Mn^{-1/2},\qquad |\bu|\le M,
\end{equation}
for the families and sufficiently large $n$ in Theorem~\ref{thm:likelihood}. We prove \eqref{eq:local-total-variation} at the end of the likelihood argument. This describes the information at the $n^{-1}$ parameter scale without precluding estimation at other scales, as in the two-star model of Mukherjee and Xu \cite{MX}. Related statistical questions arising from equal limiting edge densities are discussed by Chatterjee and Diaconis \cite{CD}, while Xu and Mukherjee \cite{XM2023} study limiting experiments for one-parameter Ising models on dense regular graphs. Here we use the quantitative covariance estimate to obtain the uniform likelihood remainder in \eqref{eq:likelihood-expansion} for multiparameter ERGMs.
\end{remark}

The same parameter-uniform fluctuation estimates also allow us to study tail probabilities through exponential tilting. At thresholds that grow with $n$, Cram\'er-type moderate deviation estimates control relative rather than absolute error. The Kolmogorov bound in \eqref{eq:main-normal} yields only $C/\{n\Phibar(x)\}$ for the relative error, so its usefulness decreases as the normal tail becomes small. We instead obtain a direct relative estimate with an explicit cubic correction in Theorem~\ref{thm:md}. Both tails must be treated, since a subgraph count need not have a distribution symmetric about its mean. In the present model, the two tilt directions can also have different effects on the energy coefficients. In particular, a lower-tail tilt may introduce a negative interaction, and we use Theorem~\ref{thm:signed} to keep the normal approximation and variance estimates uniform along this path.

For every parameter with a unique stable fixed point, define the cubic coefficients
\begin{equation}
\gamma_{\bbeta}=v_{\bbeta}^{3/2}\left\{\frac{1-2p_{\bbeta}}{p_{\bbeta}^2(1-p_{\bbeta})^2}
+2\Phi_{\bbeta}''(p_{\bbeta})\right\},\qquad
\gamma_{H,\bbeta}=\gamma_{\bbeta}+3(e_H-1)\frac{\sqrt{v_{\bbeta}}}{p_{\bbeta}}.
\label{eq:gamma}
\end{equation}
For the perturbed models \eqref{eq:signed-family}, we use \eqref{eq:gamma} with subscript $\btheta$. With this coefficient, we obtain the following expansion for both tails.

\begin{theorem}\label{thm:md}
Let $H$ be any fixed graph with at least one edge, and recall $N=\binom n2$ from \eqref{eq:edges}. For every fixed $L>0$, uniformly over $n\ge2$, $\epsilon\in\{-1,1\}$, $0\le x\le L\sqrt n$, and $\bbeta\in\B$ as in Theorem~\ref{thm:normal},
\begin{equation}
\frac{\Pp(\epsilon W_{H,\bbeta}>x)}{\Phibar(x)}
=\exp\left\{\frac{\epsilon\gamma_{H,\bbeta}x^3}{6\sqrt N}\right\}
\left[1+O\left(\frac{1+x^2}{n}+\frac{x^4}{N}\right)\right].
\label{eq:main-md}
\end{equation}
For the models of Theorem~\ref{thm:signed}, \eqref{eq:main-md} holds with $\bbeta$ replaced by $\btheta$, uniformly over $\btheta\in\B_{\mathcal J}(\eta)$. The constants depend only on the fixed parameter set, graph families, $H$ and $L$. For $n\ge v_H$, both conclusions remain valid under true-variance normalization, with the same cubic coefficient.
\end{theorem}

\begin{remark}
For Erd\H{o}s--R\'enyi subgraph counts, D\"oring and Eichelsbacher \cite{DE2013} established moderate deviation principles using cumulant bounds. Goldschmidt, Griffiths and Scott \cite[Theorems 1.1, 1.9 and 1.11]{GGS2020} obtained moderate deviation rates in $G(n,m)$ and upper-tail asymptotics in $G(n,p)$. Fang, Luo and Shao \cite[Proposition 2.3]{FLS2020} gave cubic-corrected relative bounds for both tails in the independent-edge model.

For the edge--two-star model, Mukherjee and Xu \cite[Section 1.3]{MX} express the edge count as an affine function of the total magnetization in an Ising model on the line graph of the complete graph. We are not aware of moderate deviation results for this count, either directly or through that Ising representation. To our knowledge, Theorem~\ref{thm:md} is the first general Cram\'er-type moderate deviation theorem for arbitrary fixed subgraph counts in dependent dense ERGMs of the form \eqref{eq:model}. The relative estimates with a cubic correction appear to be new even for edges in the original edge--triangle model with nonzero triangle interaction, within the stated parameter ranges.
\end{remark}

\begin{remark}\label{rem:curvature}
The explicit variance and cubic coefficient can also be read from the scalar mean-field potential
\[
J_{\bbeta}(u)=u\log u+(1-u)\log(1-u)
-2\sum_{j=1}^k\beta_j u^{e_j},\qquad 0<u<1.
\]
By the fixed-point equation in \eqref{eq:meanfield}, $J_{\bbeta}'(p_{\bbeta})=0$. Differentiating and using \eqref{eq:edge-scale} and \eqref{eq:gamma}, we obtain
\begin{equation}\label{eq:potential-curvature}
J_{\bbeta}''(p_{\bbeta})=v_{\bbeta}^{-1},\qquad
-v_{\bbeta}^{3/2}J_{\bbeta}^{(3)}(p_{\bbeta})=\gamma_{\bbeta}.
\end{equation}
The homogeneous mean-field description associates edge density $u$ with subgraph density $u^{e_H}$. We therefore set $s=u^{e_H}$ and define $\widetilde J_{H,\bbeta}(s)=J_{\bbeta}(s^{1/e_H})$. Applying the chain rule at $s=p_{\bbeta}^{e_H}$ and using \eqref{eq:potential-curvature}, we have
\begin{equation}\label{eq:subgraph-curvature}
-\frac{\widetilde J_{H,\bbeta}^{(3)}(p_{\bbeta}^{e_H})}
{\{\widetilde J_{H,\bbeta}''(p_{\bbeta}^{e_H})\}^{3/2}}
=\gamma_{\bbeta}+3(e_H-1)\frac{\sqrt{v_{\bbeta}}}{p_{\bbeta}}
=\gamma_{H,\bbeta}.
\end{equation}
By \eqref{eq:subgraph-curvature}, we can interpret the extra term in \eqref{eq:gamma} as the change in the standardized third derivative under the coordinate transformation $s=u^{e_H}$. This calculation concerns the scalar mean-field potential. The probabilistic role of the coefficient is established in Lemma~\ref{md:cgf}, where the expansion of the tilted variance leads to the correction appearing in the tail estimate.
\end{remark}

By expanding the cubic factor in \eqref{eq:main-md}, we obtain the following estimate without the correction.

\begin{corollary}\label{cor:ordinary}
Under the assumptions and in either parameter family of Theorem~\ref{thm:md}, for every fixed $L>0$ and both signs $\epsilon$,
\begin{equation}
\frac{\Pp(\epsilon W_{H,\bbeta}>x)}{\Phibar(x)}
=1+O\left(\frac{1+x^3}{n}\right),\qquad 0\le x\le Ln^{1/3},
\label{eq:ordinary}
\end{equation}
with $\bbeta$ replaced by $\btheta$ in the perturbed family. The same assertions hold under true-variance normalization for $n\ge v_H$.
\end{corollary}
\begin{proof}[Proof of Corollary~\ref{cor:ordinary}]
By \eqref{eq:edges}, we have $n^2/4\le N\le n^2/2$ for $n\ge2$. Hence, for $x\le Ln^{1/3}$, the cubic exponent in \eqref{eq:main-md} is bounded and has absolute value $O(x^3/n)$, since the coefficients in \eqref{eq:gamma} are uniformly bounded. Using $|e^u-1|\le e^{|u|}|u|$ and $(1+x^2)/n+x^4/N\le C_L(1+x^3)/n$ in \eqref{eq:main-md}, we obtain \eqref{eq:ordinary}.
\end{proof}

\par\smallskip\noindent\textit{Proof overview.}\quad Building on the ERGM analysis of Fang, Liu, Shao and Zhao \cite{FLSZ}, we combine the normal Stein equation with the conditional identity \eqref{stein:kol-decomposition} to express the normal approximation error. To control the centered remainder, we apply a second normal Stein equation as in Fang and Liu \cite{FL2025}, keeping the remainder inside the expectation to obtain \eqref{stein:second-identity}. The Poincar\'e inequality and covariance bound \eqref{eq:poincare}--\eqref{eq:covariance-bound} control the resulting moments and weighted sum. Combining these estimates with the remainder relation \eqref{struct:remainder-transfer}, we obtain Theorem~\ref{thm:residual}. For normal approximation, these estimates control the remainder term and the deviation of $K_H$ from one, where $K_H$ is the coefficient multiplying the Stein solution's derivative. The remaining discrete Taylor error is bounded using increment estimates and, for indicator tests, the local concentration bound \eqref{stein:local-concentration}. Comparing the perturbed Glauber dynamics with two ferromagnetic chains makes these bounds uniform under small signed perturbations. We can then integrate sufficient-statistic covariances for local likelihood ratios and tilted variances for moderate deviations.

Two ideas from this argument may be useful in other approximation problems with dependence. First, a remainder's absolute moment need not reflect the size of its contribution to a distributional approximation. Retaining its expectation against the Stein solution allows us to use centering and another Stein equation before taking absolute values. In the present setting, the resulting identity \eqref{stein:second-identity} involves quantities controlled by \eqref{stein:rh-moments}--\eqref{stein:rh-weighted}. Combining these bounds with \eqref{struct:remainder-transfer} and \eqref{stein:linear}, we obtain \eqref{eq:residual-bounds}. Second, exponential tilting converts normalized exponentially weighted moments into ordinary moments at a shifted parameter, as in \eqref{md:weighted-moment-transfer}. This is particularly useful for dependent exponential families when moment and normal approximation bounds hold uniformly along the tilt path. Accordingly, we use the classical change of measure of Petrov \cite[Chapter VIII]{Petrov} and a tilted Berry--Esseen bound, rather than derive a relative tail estimate directly from a Stein identity as in Chen, Fang and Shao \cite[Theorem 3.1]{CFS2013}. Integrating the tilted variance then determines the correction in \eqref{md:cgfexpansion}, while \eqref{md:weightedbound} controls the tail integral. This separation also suits higher-order corrections when the variance expansion and tilted normal approximation have the required precision.

\section{Proofs of Theorems~\ref{thm:residual}--\ref{thm:signed}}\label{sec:stein}

The comparison with the edge count and the normal approximation use the same estimates for the remainder $R_H$ and weight $K_H$ defined in \eqref{stein:rh}--\eqref{stein:kh} below. We state these estimates first and prove Theorem~\ref{thm:residual} in Section~\ref{sec:residual-proof}. Section~\ref{sec:normal-proof} adds the threshold estimate needed for indicators and proves Theorems~\ref{thm:normal} and \ref{thm:signed}. Section~\ref{sec:joint-consequences} contains the joint consequences. Apart from the conditional identity, the auxiliary lemmas are proved in Appendix~\ref{sec:normal-lemmas} using Appendix~\ref{sec:ferro}.

To treat the basic parameter regions and the signed perturbations in one argument, let $\B$ denote either a compact family from Theorem~\ref{thm:normal} or an enlarged family $\B_{\mathcal J}(2\eta)$ from \eqref{eq:signed-family}. In the latter case, we use $\bbeta$ for the full perturbed parameter and choose $\eta$ as in Lemma~\ref{lem:uniform-input}. That lemma establishes the required fixed-point and contraction bounds independently of the theorems proved in this section.

\subsection{Auxiliary lemmas for Theorem~\ref{thm:residual}}\label{sec:aux-estimates}

Fix $\bbeta\in\B$ and, for convenience, let $Y=Y_{\bbeta}$, $p=p_{\bbeta}$, $\rho=\rho_{\bbeta}$, $\sigma=\sigma_{\bbeta}$ and $\widetilde p=\widetilde p_{\bbeta}$. We first take an observed graph $H$ without isolated vertices and sufficiently large $n$. Section~\ref{sec:residual-proof} treats the other graphs, sample sizes and normalizations. Let $v=v_H$ and $e=e_H$, and define
\begin{equation}
f_H(y)=\frac{|\Hom(H,G_y)|-\mu_{H,\bbeta}}
{2ep^{e-1}n^{v-2}\sigma},\qquad V=f_H(Y)=W_{H,\bbeta}.
\label{stein:fh}
\end{equation}
By the definition of the true mean in \eqref{eq:subgraph-scale}, $\E V=0$.

For $y\in\{0,1\}^{\cI}$, $s\in\cI$ and $b\in\{0,1\}$, let $y^{(s,b)}$ be obtained by setting coordinate $s$ to $b$ and leaving all other coordinates unchanged. Define
\begin{equation}
\Delta_sF(y)=F(y^{(s,1)})-F(y^{(s,0)}),\qquad
\Delta_{sr}F=\Delta_s\Delta_rF.
\label{eq:force}
\end{equation}
For distinct $r,s$, the coordinate replacements commute, so $\Delta_{sr}F=\Delta_{rs}F$. Since $\Delta_sF$ does not depend on $y_s$, we also have $\Delta_{ss}F=0$. Higher-order differences are defined by composition.

Let $y_{-s}=(y_r)_{r\ne s}$ and use the corresponding notation for random configurations. Conditioning on $Y_{-s}$ means conditioning on the $\sigma$-field generated by the other edge coordinates. The map $y\mapsto F(y^{(s,b)})$ remains a function on the original configuration space, and its expectation at $Y$ is always under the original measure.

The local field and single-edge conditional probability are
\begin{equation}
L_{s,\bbeta}(y)=\Delta_sT_{\bbeta}(y),\qquad
q_{s,\bbeta}(y)=\Psi(L_{s,\bbeta}(y))
=\Pp(Y_{s,\bbeta}=1\mid Y_{-s,\bbeta}=y_{-s}).
\label{eq:conditional}
\end{equation}
For fixed $\bbeta$, write $L_i=L_{i,\bbeta}$ and $q_i=q_{i,\bbeta}$, using the edge ordering $1,\ldots,N$. Define
\begin{align}
R_H(y)&=\sum_{i=1}^N\{y_i-q_i(y)\}-\rho\sigma f_H(y),\label{stein:rh}\\
K_H(y)&=\frac1{\rho\sigma}\sum_{i=1}^Nq_i(y)(1-q_i(y))\Delta_if_H(y).
\label{stein:kh}
\end{align}
The remainder $R_H$ enters the conditional identity below, whereas $D_{H,\bbeta}$ in \eqref{eq:residual-definition} is the difference between normalized counts. Subtracting \eqref{stein:rh} for $H$ and $K_2$ connects these quantities, as expressed in \eqref{struct:remainder-transfer}. We first record the conditional identity, together with the centering of $R_H$ and the nonnegativity of $K_H$.
\begin{lemma}\label{stein:conditional}
For every locally absolutely continuous $F:\mathbb R\to\mathbb R$, with finite values assigned to $F'$ at points where the derivative does not exist,
\begin{equation}\label{stein:kol-decomposition}
\begin{aligned}
\E[VF(V)]-\E F'(V)
&=\frac1{\rho\sigma}\sum_{i=1}^N\E\bigl[q_i(Y)(1-q_i(Y))\{\Delta_i(F\circ f_H)(Y)-\Delta_if_H(Y)F'(V)\}\bigr]\\
&\quad+\E[(K_H(Y)-1)F'(V)]-\frac{\E[R_H(Y)F(V)]}{\rho\sigma}\\
&\coloneqq I_1(F)+I_2(F)+I_3(F).
\end{aligned}
\end{equation}
Here $I_1(F)$, $I_2(F)$ and $I_3(F)$ denote the three terms in their order of appearance, with the minus sign included in $I_3(F)$. Moreover, the functions in \eqref{stein:rh}--\eqref{stein:kh} satisfy $\E R_H(Y)=0$ and $K_H\ge0$.
\end{lemma}
\begin{proof}[Proof of Lemma~\ref{stein:conditional}]
By \eqref{eq:model}, for configurations with the same $y_{-i}$, we have
\[
\frac{p_{\bbeta}(y^{(i,1)})}{p_{\bbeta}(y^{(i,0)})}
=\exp\{T_{\bbeta}(y^{(i,1)})-T_{\bbeta}(y^{(i,0)})\}=e^{L_i(y)}.
\]
This proves the conditional probability formula in \eqref{eq:conditional}. Now fix a configuration function $A$ and condition on $Y_{-i}$. Let $q=q_i(Y)$ and $A_b=A(Y^{(i,b)})$. These quantities are fixed under the conditioning. Since $Y_i$ has only two possible values, we have
\[
\E[(Y_i-q)A(Y)\mid Y_{-i}]
=q(1-q)A_1-(1-q)qA_0=q(1-q)(A_1-A_0).
\]
Taking expectations, we obtain
\begin{equation}
\E[(Y_i-q_i(Y))A(Y)]
=\E[q_i(Y)(1-q_i(Y))\Delta_i A(Y)].
\label{stein:bernoulli}
\end{equation}
Summing \eqref{stein:bernoulli} over $i$ and using \eqref{stein:rh}, we obtain
\begin{equation}
\rho\sigma\mathbb E[V A(Y)]+\mathbb E[R_H(Y)A(Y)]
=\sum_{i=1}^N \mathbb E[q_i(Y)(1-q_i(Y))\Delta_iA(Y)].
\label{stein:target-identity}
\end{equation}
Taking $A=1$ and using $\E V=0$, we obtain $\E R_H(Y)=0$. The nonnegative increments in \eqref{eq:count} also imply $K_H\ge0$ by \eqref{stein:kh}. Finally, set $A=F\circ f_H$ in \eqref{stein:target-identity}, divide by $\rho\sigma$, and add and subtract $\E[K_H(Y)F'(V)]$. By \eqref{stein:kh}, the resulting identity is \eqref{stein:kol-decomposition}.
\end{proof}

We now state the three estimates needed for the comparison in Theorem~\ref{thm:residual}. The first controls changes in $f_H$ when one or two edges are replaced. It also gives the $L^2$ approximation by the normalized edge count, which we will use to change the argument of a test function.

\begin{lemma}\label{stein:increments}
For the fixed graph $H$, we have the following estimates uniformly in the parameters under consideration.
\begin{gather}
0\le\Delta_if_H(y)\le Cn^{-1},\label{stein:increment-size}\\
\left|\mathbb E\Delta_if_H(Y)-\sigma^{-1}\right|\le Cn^{-2},
\qquad\operatorname{Var}(\Delta_if_H(Y))\le Cn^{-3},\label{stein:increment-moments}\\
\|\Delta_{ir}f_H\|_\infty\le
\begin{cases}
Cn^{-2},&i\ne r,\ |i\cap r|=1,\\
Cn^{-3},&i\cap r=\varnothing,\\
0,&i=r,
\end{cases}\label{stein:second-increment}\\
\max_i\sum_{r=1}^N \|\Delta_{ir}f_H\|_\infty\le Cn^{-1},
\qquad\max_i\sum_{r=1}^N \|\Delta_{ir}f_H\|_\infty^2\le Cn^{-3}.\label{stein:increment-sums}
\end{gather}
Furthermore, we have
\begin{equation}
\|W_{H,\bbeta}-W_{\bbeta}\|_2\le Cn^{-1/2},
\qquad \mathbb E V^2\le C.
\label{stein:linear}
\end{equation}
\end{lemma}

The $L^2$ estimate in \eqref{stein:linear} controls the difference between the two counts, but its expected product with a Lipschitz function requires more information. We therefore turn to $R_H$ and $K_H$. The following bounds control $I_2(F)$ in \eqref{stein:kol-decomposition} and provide the moment estimates used to handle $I_3(F)$ with a second Stein equation.

\begin{lemma}\label{stein:remainder}
Uniformly over $\bbeta\in\B$, we have
\begin{align}
\mathbb ER_H(Y)^2\le Cn,\qquad
\sum_{i=1}^N \mathbb E|\Delta_iR_H(Y)|&\le Cn^{3/2},\label{stein:rh-moments}\\
0\le K_H(y)\le C,\qquad
\|K_H(Y)-1\|_2&\le Cn^{-1},\label{stein:kh-moments}\\
\mathbb E\left|\sum_{i=1}^N q_i(Y)(1-q_i(Y))\Delta_iR_H(Y)\right|&\le Cn.
\label{stein:rh-weighted}
\end{align}
\end{lemma}

A direct application of Cauchy--Schwarz to $R_H(Y)f(V)$ would use only the $O(\sqrt n)$ bound on $\|R_H(Y)\|_2$, which is too large for \eqref{eq:residual-bounds}. We instead keep the two factors inside the expectation. By applying a second Stein equation and using the centering of $R_H$, we can exploit \eqref{stein:rh-weighted}, where the absolute value is taken after summation. The resulting estimate applies separately to $H$ and $K_2$.

\begin{lemma}\label{stein:pairing}
For every $1$-Lipschitz function $f:\mathbb R\to\mathbb R$, we have
\begin{equation}
|\mathbb E[R_H(Y)f(V)]|\le C.
\label{stein:pairing-bound}
\end{equation}
The constant is independent of $f$, $n$ and $\bbeta\in\B$.
\end{lemma}

\subsection{Proof of Theorem~\ref{thm:residual}}\label{sec:residual-proof}

The key step is to subtract the remainders $R_H$ and $R_{K_2}$. Lemma~\ref{stein:pairing} controls their expectations against test functions, and \eqref{stein:linear} allows us to replace the subgraph argument by the edge-count argument. After proving the explicit-scale bounds, we use the conditional identity to obtain the variance precision needed for true-variance normalization.

\begin{proof}[Proof of Theorem~\ref{thm:residual}]
Fix one of the parameter families specified at the start of this section. We first consider $H$ without isolated vertices and sufficiently large $n$.

\proofstep{1} We begin by making the dependence on the parameter uniform. Since the scales involve both $p_{\bbeta}$ and $\rho_{\bbeta}$, we need common positive lower bounds for these quantities on $\B$. With $\varphi_{\bbeta}$ as in \eqref{eq:meanfield}, put $g(\bbeta,u)=u-\varphi_{\bbeta}(u)$. For each $\bbeta_0\in\B$, we have $g(\bbeta_0,p_{\bbeta_0})=0$ and $\partial_u g(\bbeta_0,p_{\bbeta_0})=\rho_{\bbeta_0}>0$. The implicit function theorem therefore provides a local solution $p_{\bbeta}$ that depends smoothly on $\bbeta$.

Choose an open interval $J$ containing $p_{\bbeta_0}$, with $\overline J\subset(0,1)$. On a sufficiently small parameter neighborhood, the local solution lies in $J$ and $\partial_u g>0$ on $\overline J$. Uniqueness at $\bbeta_0$ gives a positive minimum of $|g(\bbeta_0,\cdot)|$ on $[0,1]\setminus J$, which remains positive for nearby parameters by uniform continuity. Thus the local solution is the only root in $[0,1]$. These solutions agree on overlaps and define a smooth $p_{\bbeta}$ on a neighborhood of $\B$, with $\rho_{\bbeta}>0$. The three functions $p_{\bbeta}$, $1-p_{\bbeta}$ and $\rho_{\bbeta}$ are continuous and positive on the compact set $\B$, so
\begin{equation}
\inf_{\bbeta\in\B}\min\{p_{\bbeta},1-p_{\bbeta},\rho_{\bbeta}\}>0.
\label{eq:compact}
\end{equation}
Since $\rho_{\bbeta}$ is also bounded above on $\B$, \eqref{eq:edges}, \eqref{eq:edge-scale} and \eqref{eq:compact} yield constants $0<c<C<\infty$ such that
\begin{equation}
c\le v_{\bbeta}\le C,\qquad
cn\le\sigma_{\bbeta}\le Cn
\quad(n\ge2,\ \bbeta\in\B).
\label{eq:compact-scales}
\end{equation}

\proofstep{2} We now prove the expected-product estimate in \eqref{eq:residual-bounds}. Let $R=R_{K_2}$, $W=W_{\bbeta}$ and $D_H=V-W$. Since $f_{K_2}(Y)=W$, subtracting the two definitions in \eqref{stein:rh} gives $R_H=R-\rho\sigma D_H$. This relation lets us express the desired expectation in terms of the quantities controlled by Lemma~\ref{stein:pairing}, namely
\begin{equation}\label{struct:remainder-transfer}
\begin{aligned}
\rho\sigma\E[D_Hf(W)]
&=\E[Rf(W)]-\E[R_Hf(V)]
+\E\bigl[R_H\{f(V)-f(W)\}\bigr].
\end{aligned}
\end{equation}
Here $R$ and $R_H$ are evaluated at $Y$. All three variables $R(Y)$, $R_H(Y)$ and $D_H$ are centered, so subtracting $f(0)$ does not change the expectations. For $\Lip(f)\le1$, we apply Lemma~\ref{stein:pairing} to $K_2$ and $H$ to bound the first two terms in \eqref{struct:remainder-transfer} by $C$. The last term measures only the change in the test function's argument. By Cauchy--Schwarz, \eqref{stein:rh-moments} and \eqref{stein:linear}, it is bounded by
\[
\left|\E\bigl[R_H\{f(V)-f(W)\}\bigr]\right|
\le\|R_H(Y)\|_2\|V-W\|_2\le C\sqrt n\,n^{-1/2}\le C.
\]
Divide \eqref{struct:remainder-transfer} by $\rho\sigma\asymp n$, using \eqref{eq:compact}--\eqref{eq:compact-scales}. Together with \eqref{stein:linear}, this proves both bounds in \eqref{eq:residual-bounds}. Taking $f(w)=w$ proves \eqref{eq:residual-covariance}.

To obtain \eqref{eq:residual-smooth}, use the integral form of the first-order Taylor remainder. If $\Lip(\psi')\le1$, then
\[
|\psi(W+D_H)-\psi(W)-D_H\psi'(W)|\le\tfrac12D_H^2.
\]
Combining this inequality with \eqref{eq:residual-bounds}, applied to $f=\psi'$, proves \eqref{eq:residual-smooth}. The expectations are finite because $\psi$ has at most quadratic growth.

\proofstep{3} The estimates just proved use the explicit scales. To replace them by true standard deviations, we first obtain a relative variance bound from Lemma~\ref{stein:conditional}. Take $F(w)=w$ in \eqref{stein:kol-decomposition}. Then $I_1(F)=0$, and the remaining two terms give
\begin{equation}
\mathbb EV^2=\mathbb EK_H(Y)-\frac{\mathbb E[R_H(Y)V]}{\rho\sigma}.
\label{stein:variance-identity}
\end{equation}
Combining \eqref{stein:variance-identity} with \eqref{stein:kh-moments}, \eqref{stein:pairing-bound} for $f(w)=w$, and \eqref{eq:compact}--\eqref{eq:compact-scales}, we obtain $\E V^2=1+O(n^{-1})$. By \eqref{stein:fh} and \eqref{eq:subgraph-scale},
\begin{equation}
\mathbb EV^2=\operatorname{Var}(V)
=\frac{\widehat\sigma_{H,\bbeta}^2}{\sigma_{H,\bbeta}^2}
=1+O(n^{-1}).
\label{stein:variance-calibration}
\end{equation}
Let $b=\widehat\sigma_{H,\bbeta}/\sigma_{H,\bbeta}$. For large $n$, \eqref{stein:variance-calibration} implies $b^2\in[1/2,3/2]$, and hence
\begin{equation}
b\ge2^{-1/2},\qquad
|b-1|=\frac{|b^2-1|}{b+1}\le Cn^{-1},\qquad
|b^{-1}-1|\le Cn^{-1}.
\label{stein:scale-ratio}
\end{equation}
We now transfer \eqref{eq:residual-bounds} to true standard deviations. Let $b_H=\widehat\sigma_{H,\bbeta}/\sigma_{H,\bbeta}$ and $b_E=\widehat\sigma_{K_2,\bbeta}/\sigma_{K_2,\bbeta}$. By \eqref{stein:scale-ratio} for $H$ and $K_2$, both ratios and their inverses are $1+O(n^{-1})$. The normalized difference is
\begin{equation}\label{struct:true-residual}
\widehat D_H=b_H^{-1}D_H+(b_H^{-1}-b_E^{-1})W,
\qquad \widehat W_{K_2,\bbeta}=b_E^{-1}W.
\end{equation}
The first bound in \eqref{eq:residual-bounds} and \eqref{stein:linear} imply $\E\widehat D_H^2\le C/n$. To bound $\E[\widehat D_Hf(b_E^{-1}W)]$, subtract $f(0)$. In the first term of \eqref{struct:true-residual}, apply the second bound in \eqref{eq:residual-bounds} after dividing $w\mapsto f(b_E^{-1}w)-f(0)$ by its Lipschitz bound $b_E^{-1}$, which is uniformly bounded by \eqref{stein:scale-ratio}. For the second term, use $|f(b_E^{-1}W)-f(0)|\le b_E^{-1}|W|$ and $\E W^2\le C$. Both contributions are $O(n^{-1})$. Taking $f(w)=w$ and repeating the Taylor estimate proves the covariance and smooth comparisons.

We next include isolated vertices. If $H$ has $r$ of them, let $H_0$ be the graph obtained by deletion and put $v_0=|\mathcal V(H_0)|$. For $n\ge v_H=v_0+r$, every noninduced embedding of $H_0$ has $\prod_{j=0}^{r-1}(n-v_0-j)$ extensions to $H$, so
\begin{equation}\label{stein:isolates}
\begin{gathered}
|\Hom(H,G_y)|=\left\{\prod_{j=0}^{r-1}(n-v_0-j)\right\}|\Hom(H_0,G_y)|,\\
W_{H,\bbeta}=b_nW_{H_0,\bbeta},\qquad
\widehat W_{H,\bbeta}=\widehat W_{H_0,\bbeta},\\
b_n=\frac{\prod_{j=0}^{r-1}(n-v_0-j)}{n^r}
=\prod_{j=0}^{r-1}\left(1-\frac{v_0+j}{n}\right)
=1+O_H(n^{-1}).
\end{gathered}
\end{equation}
Here the mean and standard deviation acquire the same extension factor, $e_H=e_{H_0}$, and $0<b_n\le1$. The identities in \eqref{stein:isolates} follow from \eqref{eq:subgraph-scale}.

By \eqref{stein:isolates},
\[
D_{H,\bbeta}=b_nD_{H_0,\bbeta}+(b_n-1)W_{\bbeta}.
\]
Since $b_n=1+O_H(n^{-1})$ and $\E W_{\bbeta}^2\le C$ by \eqref{stein:linear}, this identity transfers both bounds in \eqref{eq:residual-bounds} from $H_0$ to $H$. For the second bound, subtract $f(0)$ and use $|f(W)-f(0)|\le|W|$. Moreover, \eqref{stein:variance-calibration} holds for $H$ because $\Var(W_{H,\bbeta})=b_n^2\Var(W_{H_0,\bbeta})$. The true-standardized variables for $H$ and $H_0$ coincide, and the covariance and smooth comparisons follow as in Step~2.

Finally, the finitely many excluded sample sizes can be included with a common constant. By \eqref{eq:count}, \eqref{eq:subgraph-scale} and \eqref{eq:compact}--\eqref{eq:compact-scales}, we have
\begin{equation}
0\le|\Hom(H,G_y)|\le n^v,\qquad
\sigma_{H,\bbeta}\ge c_Hn^{v-1},\qquad |W_{H,\bbeta}|\le C_Hn.
\label{stein:finite-explicit-bounds}
\end{equation}
On the finite range of excluded $n$, \eqref{stein:finite-explicit-bounds} controls the second moments of $D_H$ and $W$, hence their expected product against a centered Lipschitz function by Cauchy--Schwarz. For $n\ge v_H$, the count takes distinct values on the empty and complete graphs, both of positive probability under \eqref{eq:model}, so its true variance is positive. The corresponding true-standardized second moments are one. Increasing $C$ therefore extends \eqref{eq:residual-bounds}, its covariance and smooth consequences, and \eqref{stein:variance-calibration} to all the stated sample sizes. This proves Theorem~\ref{thm:residual} on both parameter families fixed at the start of the section.
\end{proof}

\subsection{Proofs of Theorems~\ref{thm:normal} and \ref{thm:signed}}\label{sec:normal-proof}

Having established the comparison with the edge count, we now prove the two normal approximation bounds. Lemma~\ref{stein:increments} already controls $I_1(F)$ in \eqref{stein:kol-decomposition} when $F'$ is Lipschitz. For an indicator test, however, the corresponding derivative has a jump at the threshold. Since its contribution to $I_1(F)$ carries the weight $K_H$, we need a weighted estimate for configurations whose single-edge changes can cross the threshold.

\begin{lemma}\label{stein:concentration}
Choose a fixed $C_1$ such that $0\le\Delta_if_H\le h:=C_1/n$ in \eqref{stein:increment-size}. Then, for every $z\in\mathbb R$, we have
\begin{equation}\label{stein:local-concentration}
\E[K_H(Y)\mathbf1_{\{|V-z|\le h\}}]\le Cn^{-1}.
\end{equation}
The constant is independent of $z$, $n$ and $\bbeta\in\B$.
\end{lemma}

\begin{proof}[Proof of Theorems~\ref{thm:normal} and \ref{thm:signed}]
We work with either of the parameter families fixed at the start of the section. In the perturbed case, Lemma~\ref{lem:uniform-input} has already established the unique stable fixed point, and the preceding lemmas apply with common constants on the enlarged family. It remains to prove the distance estimates there, as well as on the two basic regions. We first take $H$ without isolated vertices and large $n$, and let $h=C_1/n$ as in Lemma~\ref{stein:concentration}.

\proofstep{1} We start with the Wasserstein distance and record the normal Stein solutions for both Lipschitz and indicator tests, so that the same notation can be used in Step~2. Following Stein \cite[Section 2]{Stein1972}, consider
\begin{equation}
xF_\psi(x)-F_\psi'(x)=\psi(x)-\mathbb E\psi(Z),\qquad Z\sim N(0,1).
\label{stein:equation}
\end{equation}
Let $\mathbf1_A$ denote the indicator, equal to one on $A$ and zero otherwise. For Lipschitz $\psi$ and for $\psi_z(x)=\mathbf1_{\{x\le z\}}$, its bounded solution is
\begin{equation}
F_\psi(x)=-e^{x^2/2}\int_{-\infty}^x
\{\psi(u)-\mathbb E\psi(Z)\}e^{-u^2/2}\,du.
\label{stein:solution}
\end{equation}
A Lipschitz function grows at most linearly, so the integral is well defined. We use the standard Stein-factor bounds
\begin{align}
\operatorname{Lip}(\psi)\le1
&\quad\Longrightarrow\quad
\|F_\psi\|_\infty,\ \|F_\psi'\|_\infty,
\operatorname{Lip}(F_\psi')\le2,\label{stein:smooth-bounds}\\
\psi=\psi_z
&\quad\Longrightarrow\quad
\|F_z\|_\infty\le1,\quad \operatorname{Lip}(F_z)\le1,
\quad\|F_z'\|_\infty\le1.\label{stein:indicator-bounds}
\end{align}
The bounds \eqref{stein:smooth-bounds}--\eqref{stein:indicator-bounds} with this sign convention appear in Fang, Liu, Shao and Zhao \cite[equations (4.1)--(4.2), (4.26)--(4.27)]{FLSZ}. Related normal approximation arguments under dependence are developed by Chen and Shao \cite[Sections 2.1--2.2]{CS2004} and Chatterjee \cite[Theorem 2.2]{Chatterjee2008}.

For a general $1$-Lipschitz $\psi$, let $\psi_\varepsilon(x)=\int\psi(x-\varepsilon u)\chi(u)\,du$, where $\chi$ is a smooth probability density on $[-1,1]$. Then $\|\psi_\varepsilon-\psi\|_\infty\le\varepsilon$ and $\|\psi_\varepsilon'\|_\infty\le1$. By \eqref{stein:solution} and \eqref{stein:equation}, the solutions and their first derivatives converge uniformly on compact intervals. The bounds in \eqref{stein:smooth-bounds} therefore pass to the limit.

The derivative in \eqref{stein:indicator-bounds} is used almost everywhere in integrals. For pointwise expectations such as $\E F_z'(V)$, we set
\begin{equation}
F_z'(z):=zF_z(z)-1+\Phi(z).
\label{stein:indicator-representative}
\end{equation}
With the representative \eqref{stein:indicator-representative}, \eqref{stein:equation} holds for every real $x$, including $x=z$. In particular, the identity allows an atom of $V$ at the threshold.

Combining \eqref{stein:equation} with \eqref{eq:distances}, we have
\begin{equation}\label{stein:distance-reduction}
\begin{gathered}
\dWass(V,Z)=\sup_{\Lip(\psi)\le1}\left|\E[VF_\psi(V)]-\E F_\psi'(V)\right|,
\\
\dKol(V,Z)=\sup_{z\in\mathbb R}\left|\E[VF_z(V)]-\E F_z'(V)\right|.
\end{gathered}
\end{equation}
For the coordinate replacements in \eqref{eq:force},
\begin{equation}
f_H(y^{(i,0)})=f_H(y)-y_i\Delta_if_H(y),\quad
f_H(y^{(i,1)})=f_H(y)+(1-y_i)\Delta_if_H(y).
\label{stein:forced-endpoints}
\end{equation}
Take a $1$-Lipschitz function $\psi$. To estimate $I_1(F_\psi)$ in \eqref{stein:kol-decomposition}, use the $2$-Lipschitz bound for $F_\psi'$ in \eqref{stein:smooth-bounds}. Integrating between the endpoints in \eqref{stein:forced-endpoints} and using \eqref{stein:increment-size}, we obtain
\begin{equation}
|\Delta_i(F_\psi\circ f_H)-\Delta_if_H\,F_\psi'(f_H)|
\le (\Delta_if_H)^2\le h\Delta_if_H.
\label{stein:smooth-remainder}
\end{equation}
By \eqref{stein:smooth-remainder} and \eqref{stein:kh}, we have $|I_1(F_\psi)|\le h\E K_H(Y)$. Moreover, \eqref{stein:smooth-bounds} implies $|I_2(F_\psi)|\le2\E|K_H(Y)-1|$. For $I_3(F_\psi)$, apply \eqref{stein:pairing-bound} to $F_\psi/2$. Combining these estimates with \eqref{stein:kol-decomposition} and \eqref{stein:equation}, we obtain
\begin{equation}\label{stein:wass-proof}
\begin{aligned}
|\E\psi(V)-\E\psi(Z)|&\le\sum_{\ell=1}^3|I_\ell(F_\psi)|
\\
&\le h\E K_H(Y)+2\E|K_H(Y)-1|
+\frac{|\E[R_H(Y)F_\psi(V)]|}{\rho\sigma}\le Cn^{-1}.
\end{aligned}
\end{equation}
The last inequality in \eqref{stein:wass-proof} follows from \eqref{stein:kh-moments}, \eqref{stein:pairing-bound}, $h=C_1/n$ and \eqref{eq:compact}--\eqref{eq:compact-scales}. Taking the supremum in \eqref{stein:distance-reduction} proves the Wasserstein bound. The first moment of $V$ is finite by \eqref{stein:linear}.

\proofstep{2} We next take the indicator test $\psi_z$ to obtain the Kolmogorov bound. The estimates for $I_2(F_z)$ and $I_3(F_z)$ are the same as before, using \eqref{stein:indicator-bounds}, \eqref{stein:kh-moments} and \eqref{stein:pairing-bound}. Only $I_1(F_z)$ needs a different argument, because $F_z'$ has a jump at $z$. We separate this jump from the continuous part before applying Lemma~\ref{stein:concentration}. By \eqref{stein:equation} and \eqref{stein:indicator-bounds}, for $|u|\le h$ we have
\begin{equation}\label{stein:indicator-derivative-difference}
\begin{aligned}
|F_z'(w+u)-F_z'(w)|
&\le |(w+u)F_z(w+u)-wF_z(w)|+|\mathbf1_{\{w+u\le z\}}-\mathbf1_{\{w\le z\}}|\\
&\le |u|(1+|w|)+\mathbf1_{\{|w-z|\le h\}}.
\end{aligned}
\end{equation}
Integrating \eqref{stein:indicator-derivative-difference} over the interval in \eqref{stein:forced-endpoints} and using \eqref{stein:increment-size} for its length, we obtain
\begin{align}
|\Delta_i(F_z\circ f_H)(Y)-\Delta_if_H(Y)F_z'(V)|
&\le\frac{\{\Delta_if_H(Y)\}^2}{2}(1+|V|)
+\Delta_if_H(Y)\mathbf1_{\{|V-z|\le h\}}\notag\\
&\le\frac h2\Delta_if_H(Y)(1+|V|)
+\Delta_if_H(Y)\mathbf1_{\{|V-z|\le h\}}.
\label{stein:indicator-remainder}
\end{align}
The integral of $|u|$ over the increment interval contributes the factor $1/2$. At $V=z$, we use \eqref{stein:indicator-representative} for the subtracted value. The integral is unchanged by a single point, and the closed event in \eqref{stein:indicator-remainder} still bounds the jump.

Substituting \eqref{stein:indicator-remainder} into $I_1(F_z)$ in \eqref{stein:kol-decomposition} and using \eqref{stein:kh}, we obtain the first two expectations below. Combining this with the bounds on $I_2(F_z)$ and $I_3(F_z)$, we have
\begin{align}
|\Pp(V\le z)-\Phi(z)|
&\le\sum_{\ell=1}^3|I_\ell(F_z)|\le\frac h2\E[K_H(Y)(1+|V|)]
+\E[K_H(Y)\mathbf1_{\{|V-z|\le h\}}]\notag\\
&\quad+\E|K_H(Y)-1|
+\frac{|\E[R_H(Y)F_z(V)]|}{\rho\sigma}
\le Cn^{-1}.\label{stein:kol-proof}
\end{align}
The two contributions from $I_1(F_z)$ in \eqref{stein:kol-proof} are $O(n^{-1})$ by \eqref{stein:kh-moments}, \eqref{stein:linear} and \eqref{stein:local-concentration}. Combining \eqref{stein:indicator-bounds}, \eqref{stein:kh-moments}, \eqref{stein:pairing-bound} and \eqref{eq:compact}--\eqref{eq:compact-scales} gives the same order for $I_2(F_z)$ and $I_3(F_z)$. Taking the supremum in \eqref{stein:distance-reduction} proves the Kolmogorov bound.

\proofstep{3} We have proved both distance bounds for the explicit scales. Combining \eqref{stein:wass-proof} and \eqref{stein:kol-proof} with \eqref{stein:variance-calibration}, we obtain \eqref{eq:main-normal} and \eqref{eq:normal-signed} for large $n$. To complete the proof, we transfer the bounds to true standard deviations and then include isolated vertices and the remaining sample sizes. Let $b=\widehat\sigma_{H,\bbeta}/\sigma_{H,\bbeta}$ as in \eqref{stein:scale-ratio}. For every $a>0$, we use the elementary comparisons
\begin{equation}
\begin{aligned}
d_{\mathrm{Kol}}(aU,Z)&\le d_{\mathrm{Kol}}(U,Z)+\sup_z|\Phi(z/a)-\Phi(z)|,
\\
d_{\mathrm{Wass}}(aU,Z)&\le a\,d_{\mathrm{Wass}}(U,Z)+|a-1|\mathbb E|Z|.
\end{aligned}
\label{stein:scale-comparison}
\end{equation}
The Kolmogorov bound in \eqref{stein:scale-comparison} follows from $\Pp(aU\le z)=\Pp(U\le z/a)$. For the Wasserstein bound, compare $aU$ first with $aZ$ and then with $Z$. For $a$ in a fixed neighborhood of one, the mean value theorem and $\sup_x|x|\phi(x)<\infty$ give
\begin{equation}
\sup_z|\Phi(z/a)-\Phi(z)|\le C|a-1|.
\label{stein:normal-scale-comparison}
\end{equation}
Since $\widehat W_{H,\bbeta}=b^{-1}V$, applying \eqref{stein:scale-comparison}--\eqref{stein:normal-scale-comparison} with $a=b^{-1}$ and using \eqref{stein:scale-ratio}, \eqref{stein:kol-proof} and \eqref{stein:wass-proof} proves the two distance bounds under true-variance normalization.

For a graph with isolated vertices, use \eqref{stein:isolates} and apply \eqref{stein:scale-comparison}--\eqref{stein:normal-scale-comparison} with $a=b_n$. The two distance bounds follow from those for $H_0$, and the true-standardized variables coincide. The variance estimate was already extended to these graphs in Section~\ref{sec:residual-proof}.

It remains to cover finitely many sample sizes $2\le n<n_0$. By \eqref{eq:distances} and \eqref{stein:finite-explicit-bounds}, we have
\begin{equation}
d_{\mathrm{Kol}}(W_{H,\bbeta},Z)\le1,\quad
d_{\mathrm{Wass}}(W_{H,\bbeta},Z)\le C(1+n),\quad
\left|\widehat\sigma_{H,\bbeta}^2/\sigma_{H,\bbeta}^2-1\right|\le C_Hn^2+1.
\label{stein:finite-distance-bounds}
\end{equation}
For $n\ge v_H$, positivity of the true variance was proved in Section~\ref{sec:residual-proof}. Hence
\begin{equation}\label{stein:finite-true-bounds}
\widehat\sigma_{H,\bbeta}>0,\quad
\mathbb E\widehat W_{H,\bbeta}^{\,2}=1,\quad
d_{\mathrm{Kol}}(\widehat W_{H,\bbeta},Z)\le1,\quad
d_{\mathrm{Wass}}(\widehat W_{H,\bbeta},Z)\le1+\mathbb E|Z|.
\end{equation}
The right-hand sides of \eqref{stein:finite-distance-bounds}--\eqref{stein:finite-true-bounds} are uniformly bounded on this finite range. Increasing $C$ completes both normal approximation bounds. Together with \eqref{stein:variance-calibration} and the signed-family argument in the proof of Theorem~\ref{thm:residual}, this proves Theorems~\ref{thm:normal} and \ref{thm:signed}.
\end{proof}

For the edge count, both distance bounds have matching lower bounds under either normalization. Let $\widehat W_{\bbeta}=\widehat W_{K_2,\bbeta}$. Uniformly on each parameter family considered in this section, for $d\in\{\dKol,\dWass\}$ we have
\begin{equation}
d(W_{\bbeta},Z)\ge c n^{-1},\qquad
d(\widehat W_{\bbeta},Z)\ge c n^{-1},\qquad n\ge2.
\label{stein:lattice-lower}
\end{equation}
Indeed, the Poincar\'e inequality \eqref{eq:poincare} gives $\Var(E_{\bbeta})\le CN$. By Chebyshev's inequality, an interval containing at most $Cn$ integers carries probability at least $3/4$, so one atom has mass at least $c/n$. Continuity of $\Phi$ implies that its Kolmogorov distance from either normalized count is at least half this mass. For Wasserstein distance, let $s$ be either $\sigma_{\bbeta}$ or $\sqrt{\Var(E_{\bbeta})}$ and put $U=(E_{\bbeta}-\mu_{\bbeta})/s$ and $\Lambda=-\mu_{\bbeta}/s+s^{-1}\mathbb Z$. By \eqref{eq:compact-scales}, \eqref{eq:poincare} and \eqref{stein:finite-true-bounds}, $0<s\le Cn$. The function $f(x)=\operatorname{dist}(x,\Lambda)$ is $1$-Lipschitz and vanishes at $U$. At least unit length of $[-1,1]$ lies at distance at least $\min\{1,s^{-1}\}/8$ from $\Lambda$. Integrating over this set, we obtain
\[
\dWass(U,Z)\ge\E f(Z)
\ge\frac{\phi(1)}8\min\{1,s^{-1}\}\ge\frac cn.
\]
Combining the two lattice estimates proves \eqref{stein:lattice-lower} for both normalizations.

\subsection{Joint fluctuations}\label{sec:joint-consequences}

Since all subgraph counts are defined on the same probability space as the edge count, we can now compare them jointly with a single normal variable. The expected-product estimate also controls their covariances at the finer order needed for the information matrix in Section~\ref{sec:likelihood-proof}.

\begin{corollary}\label{stein:joint}
Let $H^{(1)},\ldots,H^{(m)}$ be arbitrary fixed simple graphs, each with at least one edge. For test functions that are Lipschitz with respect to Euclidean distance, we have, uniformly over $\bbeta\in\B$,
\begin{equation}
\sup_{\operatorname{Lip}(\psi)\le1}
\left|\mathbb E\psi(W_{H^{(1)},\bbeta},\ldots,W_{H^{(m)},\bbeta})
-\mathbb E\psi(Z,\ldots,Z)\right|
\le Cn^{-1/2}.
\label{stein:joint-distance}
\end{equation}
Furthermore, for fixed $a,b\in\{1,\ldots,m\}$, we have
\begin{equation}
\operatorname{Cov}(W_{H^{(a)},\bbeta},W_{H^{(b)},\bbeta})=1+O(n^{-1}).
\label{stein:joint-covariance}
\end{equation}
For $n\ge\max_a v_{H^{(a)}}$, both statements also hold when each component is normalized by its true standard deviation.
\end{corollary}
\begin{proof}[Proof of Corollary~\ref{stein:joint}]
We compare the subgraph vector with $(W_{\bbeta},\ldots,W_{\bbeta})$ on the original probability space. The diagonal map is $\sqrt m$-Lipschitz. Applying \eqref{eq:residual-bounds} and \eqref{eq:main-normal}, with their extensions in Theorem~\ref{thm:signed} for the perturbed family, we obtain, for every Euclidean $1$-Lipschitz $\psi$,
\[
\begin{aligned}
\left|\E\psi(W_{H^{(1)},\bbeta},\ldots,W_{H^{(m)},\bbeta})
-\E\psi(Z,\ldots,Z)\right|
&\le\E\left(\sum_{a=1}^m|W_{H^{(a)},\bbeta}-W_{\bbeta}|^2\right)^{1/2}
+\sqrt m\,\dWass(W_{\bbeta},Z)\\
&\le\left(\sum_{a=1}^m\E|W_{H^{(a)},\bbeta}-W_{\bbeta}|^2\right)^{1/2}
+C/n\le Cn^{-1/2}.
\end{aligned}
\]
Taking the supremum over $\psi$ proves \eqref{stein:joint-distance}.

For the covariance, set $W=W_{\bbeta}$, $D_a=D_{H^{(a)},\bbeta}$ and $D_b=D_{H^{(b)},\bbeta}$. All variables are centered. Therefore,
\[
\Cov(W+D_a,W+D_b)
=\E W^2+\E[D_aW]+\E[D_bW]+\E[D_aD_b].
\]
The first term equals $1+O(n^{-1})$ by \eqref{stein:variance-calibration}. The next two are $O(n^{-1})$ by \eqref{eq:residual-covariance}, and the last is $O(n^{-1})$ by Cauchy--Schwarz and \eqref{eq:residual-bounds}. This proves \eqref{stein:joint-covariance}.

For true-variance normalization, use the hatted versions of \eqref{eq:residual-bounds} and \eqref{eq:residual-covariance}, together with the normal approximation in Theorem~\ref{thm:normal} or \ref{thm:signed}. The edge-count second moment is then exactly one, and the same calculations prove \eqref{stein:joint-distance} and \eqref{stein:joint-covariance} throughout the stated range.
\end{proof}

The joint approximation just proved identifies the leading fluctuations. For Corollary~\ref{cor:conditional-centering}, we instead retain the differences after multiplication by $\sqrt n$. The second bound in \eqref{eq:residual-bounds} then allows us to determine their conditional means in every joint subsequential limit.

\begin{proof}[Proof of Corollary~\ref{cor:conditional-centering}]
By \eqref{eq:residual-bounds} and \eqref{stein:variance-calibration}, all coordinate second moments in \eqref{eq:residual-joint-limit} are uniformly bounded, which proves tightness. The edge-count bound in Theorem~\ref{thm:normal} or \ref{thm:signed} identifies the first limit as standard normal. Applying weak convergence to $u^2\wedge K$ and then letting $K\to\infty$ shows that each $U_j$ has a finite second moment.

Let $f$ be bounded and Lipschitz. By the second estimate in \eqref{eq:residual-bounds},
\begin{equation}\label{struct:limit-test}
\left|\E\bigl[\sqrt nD_{H^{(j)},\bbeta_n}f(W_{\bbeta_n})\bigr]\right|
\le C\Lip(f)n^{-1/2}\longrightarrow0.
\end{equation}
By \eqref{eq:residual-bounds}, the products in \eqref{struct:limit-test} have second moments at most $C\|f\|_\infty^2$ and are therefore uniformly integrable. Combining this with the joint convergence in \eqref{eq:residual-joint-limit}, we may pass to expectations in \eqref{struct:limit-test} and obtain $\E[U_jf(Z)]=0$. The finite signed measure $B\mapsto\E[U_j\ind_{\{Z\in B\}}]$ therefore integrates every bounded Lipschitz function to zero. Such functions determine finite signed measures on $\mathbb R$, so this measure vanishes. This is exactly \eqref{eq:residual-conditional-mean}. The true-variance version follows from the corresponding bounds in Theorem~\ref{thm:residual} and the same argument.
\end{proof}

\section{Proofs of Theorems~\ref{thm:likelihood} and \ref{thm:md}}\label{sec:md}

We now apply the parameter-uniform fluctuation estimates to local likelihood ratios and relative tail probabilities. The likelihood argument integrates covariances along a short parameter segment. For moderate deviations, we choose an exponential tilt that brings the tail threshold close to the new mean, then estimate the weighted tail integral left by the change of measure. The required lemmas are stated in Section~\ref{sec:tilting-statements}, used in Section~\ref{sec:md-proof}, and proved in Appendix~\ref{sec:tilting-lemmas}.

\subsection{Proof of Theorem~\ref{thm:likelihood}}\label{sec:likelihood-proof}

By \eqref{eq:score-information}, the covariance matrix of the sufficient statistics is the Hessian of the log partition function. We first identify its leading term using \eqref{stein:joint-covariance} and integrate along the local parameter segment. The edge-count approximation then identifies the random linear term in the likelihood ratio.

\begin{proof}[Proof of Theorem~\ref{thm:likelihood} and \eqref{eq:local-total-variation}]
Throughout this proof, $\bbeta$ is the actual parameter, and the fixed energy family includes every added coordinate in the signed case. By \eqref{eq:subgraph-scale} and \eqref{eq:sufficient-statistics}, for $1\le j\le k$ we have
\begin{equation}\label{info:statistic-identity}
A_{j,n}(Y_{\bbeta})-\E A_{j,n}(Y_{\bbeta})
=2a_{j,\bbeta}\sigma_{\bbeta}W_{H_j,\bbeta}.
\end{equation}
Combining \eqref{info:statistic-identity}, \eqref{stein:joint-covariance} and $\sigma_{\bbeta}^2=Nv_{\bbeta}$ from \eqref{eq:edge-scale}, we obtain
\[
\mathcal I_n(\bbeta)_{ij}
=4Nv_{\bbeta}a_{i,\bbeta}a_{j,\bbeta}\{1+O(n^{-1})\}.
\]
The dimension is fixed, and $v_{\bbeta}$ and $\ba_{\bbeta}$ are uniformly bounded by \eqref{eq:compact}--\eqref{eq:compact-scales}. Hence the entrywise $O(n)$ error also has operator norm $O(n)$, proving \eqref{eq:information-expansion}. Since $4N=2n^2-2n$, it follows that
\begin{equation}\label{info:normalized-information}
n^{-2}\mathcal I_n(\bbeta)
=2v_{\bbeta}\ba_{\bbeta}\ba_{\bbeta}^{\mathsf T}+O(n^{-1})
\end{equation}
in operator norm.

To use \eqref{info:normalized-information} in the likelihood expansion, we need it uniformly along $\bbeta+t\bu/n$, $0\le t\le1$. For a compact subset of $\mathcal S_+$ or $\mathcal D$, the fixed-point argument leading to \eqref{eq:compact} and the strict coefficient conditions provide a compact neighborhood in the same region. For sufficiently large $n$, it contains every segment with $|\bu|\le M$. For the inner signed family, return to the notation $\btheta=(\bbeta,\bh)\in\B_{\mathcal J}(\eta)$ from \eqref{eq:signed-family}, with $k$ original and $m$ added coordinates. Here the margin from $\eta$ to $2\eta$ applies only to the added coefficients, not to the ferromagnetic base set $\B$. To keep the base parameter in $\B$, we therefore leave $\bbeta$ fixed and represent its displacement by extra copies of $H_1,\ldots,H_k$. The total size of the added coefficients is then bounded by
\begin{equation}\label{info:signed-path}
\sum_{\ell=1}^m e_{J_\ell}^{\,2}
\left|h_\ell+\frac{tu_{k+\ell}}n\right|
+\sum_{j=1}^k e_j^2\frac{t|u_j|}n
\le\eta+\frac{C_{\cH,\mathcal J}M}n\le2\eta.
\end{equation}
Theorem~\ref{thm:signed} uses the same radius for the fixed added family $(\mathcal J,\cH)$, so \eqref{info:signed-path} places all these segments in one outer family. The repeated-coordinate representation leaves the Gibbs measure and mean-field map unchanged. Applying \eqref{stein:joint-covariance} to the original sufficient statistics under this measure proves \eqref{info:normalized-information} in the coordinates of \eqref{eq:sufficient-statistics}, including at boundary points of the inner signed set. From now on, $\bbeta$ again denotes the full actual parameter.

We also need the leading information matrix to vary by $O(n^{-1})$ along these segments. Implicit differentiation of \eqref{eq:meanfield} gives
\[
\rho_{\bbeta}\,\partial_{\beta_j}p_{\bbeta}
=2p_{\bbeta}(1-p_{\bbeta})e_jp_{\bbeta}^{e_j-1},
\qquad 1\le j\le k,
\]
which proves \eqref{eq:meanfield-response}. The positive lower bounds in \eqref{eq:compact} also bound the derivatives of $v_{\bbeta}$ and $\ba_{\bbeta}$ on the comparison families. Combining these bounds with \eqref{info:normalized-information}, we obtain, uniformly over $0\le t\le1$ and $|\bu|\le M$,
\begin{equation}\label{info:path-information}
n^{-2}\mathcal I_n(\bbeta+t\bu/n)
=2v_{\bbeta}\ba_{\bbeta}\ba_{\bbeta}^{\mathsf T}+O_M(n^{-1}).
\end{equation}
In the signed case, the same derivative bounds follow in the repeated-coordinate representation of \eqref{info:signed-path}.

We can now integrate \eqref{info:path-information}. Taylor's formula and the exact Hessian identity \eqref{eq:score-information} yield
\begin{equation}\label{info:hessian-integral}
\begin{aligned}
&\log Z(\bbeta+\bu/n)-\log Z(\bbeta)
-\frac{\bu^{\mathsf T}\E\bA_n(Y_{\bbeta})}{n}\\
&\qquad=\frac1{n^2}\int_0^1(1-t)
\bu^{\mathsf T}\mathcal I_n(\bbeta+t\bu/n)\bu\,dt=v_{\bbeta}(\ba_{\bbeta}^{\mathsf T}\bu)^2+O_M(n^{-1}),
\end{aligned}
\end{equation}
where the last equality follows from \eqref{info:path-information}. Substitution into the exact likelihood identity obtained from \eqref{eq:model},
\[
\log\frac{p_{\bbeta+\bu/n}(Y_{\bbeta})}{p_{\bbeta}(Y_{\bbeta})}
=\frac{\bu^{\mathsf T}\bA_n(Y_{\bbeta})}{n}
-\log Z(\bbeta+\bu/n)+\log Z(\bbeta),
\]
proves \eqref{eq:likelihood-expansion}. The remainder is deterministic because it is the error in \eqref{info:hessian-integral}.

The deterministic part of the likelihood expansion is now established. To identify the random term, we compare the centered sufficient statistics with the normalized edge count. Applying \eqref{eq:residual-bounds} to the energy graphs in \eqref{info:statistic-identity}, and using $2\sigma_{\bbeta}/n=\sqrt{2v_{\bbeta}(1-n^{-1})}$, we obtain
\begin{equation}\label{info:score-edge}
\E\left|
\frac{\bA_n(Y_{\bbeta})-\E\bA_n(Y_{\bbeta})}{n}
-\sqrt{2v_{\bbeta}}\ba_{\bbeta}W_{\bbeta}
\right|^2\le\frac Cn.
\end{equation}
Combining \eqref{info:score-edge} with the edge-count normal approximation in Theorem~\ref{thm:normal} or \ref{thm:signed}, and then with \eqref{eq:likelihood-expansion}, proves \eqref{eq:likelihood-limit}. The same approximation proves joint convergence for finitely many directions, since each uses the same $W_{\bbeta}$.

It remains to prove the total variation estimate \eqref{eq:local-total-variation}. Suppose that $\ba_{\bbeta}^{\mathsf T}\bu=0$. Smoothness along the segment implies $|\ba_{\bbeta+t\bu/n}^{\mathsf T}\bu|\le C_M/n$. By \eqref{info:normalized-information}, the information in this direction satisfies $n^{-2}\bu^{\mathsf T}\mathcal I_n(\bbeta+t\bu/n)\bu\le C_M/n$. We can therefore bound the relative entropy by taking the expectation of the negative log-likelihood ratio and using the integral in \eqref{info:hessian-integral}. This yields
\begin{equation}\label{info:relative-entropy}
\sum_{y\in\{0,1\}^{\cI}}p_{\bbeta}(y)
\log\frac{p_{\bbeta}(y)}{p_{\bbeta+\bu/n}(y)}
\le C_M/n.
\end{equation}
For positive probability masses $p,q$ on a finite set $\Omega$, Cauchy--Schwarz and $-\log x\ge2(1-\sqrt x)$ imply
\[
\begin{aligned}
\left(\frac12\sum_{y\in\Omega}|p(y)-q(y)|\right)^2
\le1-\left(\sum_{y\in\Omega}\sqrt{p(y)q(y)}\right)^2
\le\sum_{y\in\Omega}p(y)\log\frac{p(y)}{q(y)}.
\end{aligned}
\]
Applying this comparison with \eqref{info:relative-entropy} proves \eqref{eq:local-total-variation}.
\end{proof}

\subsection{Auxiliary lemmas for Theorem~\ref{thm:md}}\label{sec:tilting-statements}

For moderate deviations, we need estimates along the parameter path generated by tilting with the observed subgraph count. Let $\bbeta_0$ denote a base parameter in a compact family of Theorem~\ref{thm:normal} or in $\B_{\mathcal J}(\eta)$ from Theorem~\ref{thm:signed}. In the latter case, $\bbeta_0$ includes all added coefficients. Fix $L>0$ and first assume that $H$ has no isolated vertices and $n$ is sufficiently large. We write $e=e_H\ge1$ and allow both signs $\epsilon\in\{-1,1\}$.

To define the tilt, write $(H_1,\ldots,H_k)$ for the full energy family of the base model, including any added graphs. If $H$ is isomorphic to some $H_j$, choose one such $j$ and vary its coefficient. Otherwise, enlarge the family to $(H_1,\ldots,H_k,H)$, set $j=k+1$, and identify $\bbeta_0$ with $(\bbeta_0,0)$. This leaves the base distribution unchanged. We continue to use $T_{\bbeta},p_{\bbeta},Y_{\bbeta}$ for the enlarged model in \eqref{eq:model}. Let $\mathbf e_j$ denote the $j$th coordinate unit vector in the relevant parameter space. Let $\sigma_{H,0}=\sigma_{H,\bbeta_0}$ and $\mu_{H,0}=\mu_{H,\bbeta_0}$. Define the parameter path by
\begin{equation}\label{md:path}
\bbeta(t)=\bbeta_0+\frac{t n^{v_H-2}}{\sigma_{H,0}}\mathbf e_j
=\bbeta_0+\frac{t}{2e p_{\bbeta_0}^{e-1}\sqrt{Nv_{\bbeta_0}}}\mathbf e_j.
\end{equation}
Write $M_H(t)=\log\mathbb E e^{tW_{H,\bbeta_0}}$ for the cumulant generating function and $a_H(t)=\sigma_{H,\bbeta(t)}/\sigma_{H,0}$ for the scale ratio. Here $v_{\bbeta_0}$ is the variance density of the base model. We consider the tilt range
\begin{equation}\label{md:allowed-path}
t=\epsilon u,\qquad 0\le u\le2L\sqrt n,
\quad\epsilon\in\{-1,1\}.
\end{equation}
By \eqref{md:path} and \eqref{eq:compact}--\eqref{eq:compact-scales}, the range \eqref{md:allowed-path} permits a coefficient displacement of order $n^{-1/2}$, compared with the $n^{-1}$ displacements in the local likelihood argument. Define the configuration function with the base mean and scale by
\[
f_{H,\bbeta_0}(y)
=\frac{|\Hom(H,G_y)|-\mu_{H,0}}{\sigma_{H,0}},
\qquad W_{H,\bbeta_0}=f_{H,\bbeta_0}(Y_{\bbeta_0}).
\]
Write the base coefficient in \eqref{eq:gamma} as
\begin{equation}\label{md:gamma}
\gamma_{H,0}=\gamma_{H,\bbeta_0}
=\gamma_{\bbeta_0}+3(e-1)\frac{\sqrt{v_{\bbeta_0}}}{p_{\bbeta_0}}.
\end{equation}
By \eqref{eq:subgraph-scale} and \eqref{md:path}, we have
\begin{equation}\label{md:scale-ratio}
a_H(t)^2=\frac{p_{\bbeta(t)}^{2e-2}v_{\bbeta(t)}}
{p_{\bbeta_0}^{2e-2}v_{\bbeta_0}}.
\end{equation}
To keep the tilt explicit, we use the quadratic approximation to the tilted mean. The following variables record the resulting centering error and the change in scale:
\begin{equation}\label{md:Yt}
\begin{aligned}
Y_t&=f_{H,\bbeta_0}(Y_{\bbeta(t)})
-t-\frac{\gamma_{H,0}t^2}{2\sqrt N},\qquad
\delta_t=M_H'(t)-t-\frac{\gamma_{H,0}t^2}{2\sqrt N},\\
U_t&=\frac{|\Hom(H,G_{Y_{\bbeta(t)}})|
-\mu_{H,\bbeta(t)}}{\sigma_{H,\bbeta(t)}}
=W_{H,\bbeta(t)}.
\end{aligned}
\end{equation}
These variables satisfy $Y_t=a_H(t)U_t+\delta_t$, because $M_H'(t)$ is the tilted mean of $f_{H,\bbeta_0}$. Thus $U_t$ uses the tilted model's own mean and scale, $a_H(t)$ records the scale change, and $\delta_t$ measures the error in the quadratic centering. Unlike the configuration $Y_{\bbeta(t)}$, $Y_t$ is scalar.

After changing measure, we compare the resulting exponentially weighted tail expectation with its normal counterpart,
\begin{equation}\label{md:Q}
Q(z)=e^{z^2/2}\overline\Phi(z)
=\int_0^\infty e^{-zu}\phi(u)\,du,\qquad z\ge0.
\end{equation}

We choose the tilt so that the deterministic centering in \eqref{md:Yt} equals the target threshold. For $\epsilon\in\{-1,1\}$ and $1\le x\le L\sqrt n$, we solve
\begin{equation}\label{md:center-equation}
t_\epsilon+\frac{\gamma_{H,0}t_\epsilon^2}{2\sqrt N}=\epsilon x.
\end{equation}
We choose the solution that tends to zero as $x\to0$. Its explicit expression and expansion are
\begin{equation}\label{md:center-root}
t_\epsilon=\frac{2\epsilon x}
{1+\sqrt{1+2\epsilon\gamma_{H,0}x/\sqrt N}},\qquad
t_\epsilon=\epsilon x-\frac{\gamma_{H,0}x^2}{2\sqrt N}
+O(x^3/N).
\end{equation}
Before using this tilt, we must check both that the root in \eqref{md:center-root} lies in the allowed interval and that the shifted model remains in a family with uniform estimates. The following lemma establishes these facts for both signs.

\begin{lemma}\label{md:path-input}
There are a fixed compact parameter family $\B_H$ and $n_0$ such that, for every base parameter, $n\ge n_0$ and $t$ in \eqref{md:allowed-path}, the tilted parameter belongs to $\B_H$ and
\[
\dKol(U_t,Z)\le C/n,\qquad
\left|\frac{\Var(|\Hom(H,G_{Y_{\bbeta(t)}})|)}{\sigma_{H,\bbeta(t)}^2}-1\right|
\le C/n,\qquad c\le a_H(t)\le C.
\]
For $1\le x\le L\sqrt n$, \eqref{md:center-equation} has a unique solution $t_\epsilon$ in $|t|\le2L\sqrt n$. It satisfies $\sgn(t_\epsilon)=\epsilon$, $|t_\epsilon|\asymp x$, and \eqref{md:center-root}.
\end{lemma}

With the tilted parameters controlled, we next expand the cumulant generating function. Its second derivative is the tilted variance divided by the squared base scale, so Lemma~\ref{md:path-input} allows us to approximate $M_H''(t)$ by $a_H(t)^2$. We expand this squared scale ratio using \eqref{md:scale-ratio} and then integrate to obtain the following estimates.

\begin{lemma}\label{md:cgf}
Uniformly for $t$ in \eqref{md:allowed-path}, we have
\begin{equation}\label{md:scaleexpansion}
a_H(t)^2=1+\frac{\gamma_{H,0}t}{\sqrt N}+O(t^2/N),\qquad
a_H(t)=1+\frac{\gamma_{H,0}t}{2\sqrt N}+O(t^2/N).
\end{equation}
The cumulant generating function and its first two derivatives satisfy
\begin{equation}\label{md:cgfexpansion}
\begin{aligned}
M_H''(t)&=1+\frac{\gamma_{H,0}t}{\sqrt N}
+O\!\left(\frac1n+\frac{t^2}{N}\right),\qquad
M_H'(t)=t+\frac{\gamma_{H,0}t^2}{2\sqrt N}
+O\!\left(\frac{|t|}{n}+\frac{|t|^3}{N}\right),\\
M_H(t)&=\frac{t^2}{2}+\frac{\gamma_{H,0}t^3}{6\sqrt N}
+O\!\left(\frac{t^2}{n}+\frac{t^4}{N}\right).
\end{aligned}
\end{equation}
\end{lemma}

The cumulant generating function controls the exponential factor in the change of measure. To estimate the remaining tail integral, we must also compare $Y_t$ with a normal variable. By \eqref{md:Yt}, $Y_t$ is obtained from $U_t$ by rescaling by $a_H(t)$ and shifting by $\delta_t$. The next lemma controls the effect of these changes.

\begin{lemma}\label{md:tilted-normal}
For $t$ in \eqref{md:allowed-path}, we have the following bounds for the variables in \eqref{md:Yt}.
\begin{equation}\label{md:tiltBEexact}
\sup_{y\in\mathbb R}
\left|\mathbb P(Y_t\le y)
-\Phi\!\left(\frac{y-\delta_t}{a_H(t)}\right)\right|
\le C/n,
\qquad
|\delta_t|\le C\left(\frac{|t|}{n}+\frac{|t|^3}{N}\right),
\end{equation}
and
\begin{equation}\label{md:tiltBE}
d_{\mathrm{Kol}}(Y_t,Z)
\le C\left\{\frac1n+\frac{|t|}{\sqrt N}+\frac{t^2}{N}
+\frac{|t|}{n}+\frac{|t|^3}{N}\right\}
\le C\left\{\frac{1+|t|}{n}+\frac{|t|^3}{N}\right\}.
\end{equation}
\end{lemma}

To turn the distributional bound into a tail-integral estimate, we compare exponentially weighted expectations over the event $U>0$. The bounds on $Q$ also control the effect of changing its argument from $x$ to $|t_\epsilon|$.

\begin{lemma}\label{md:weighted}
For the function $Q$ in \eqref{md:Q} and every $z\ge0$, we have
\begin{equation}\label{md:Qbounds}
\frac{c}{1+z}\le Q(z)\le\frac{C}{1+z},\qquad
|Q'(z)|\le\frac{C}{(1+z)^2}.
\end{equation}
If $d_{\mathrm{Kol}}(U,Z)\le d$, then, for every $z\ge0$, we also have
\begin{equation}\label{md:weightedbound}
\left|\mathbb E[e^{-zU}\mathbf1_{\{U>0\}}]-Q(z)\right|\le2d.
\end{equation}
We also have $d_{\mathrm{Kol}}(-U,Z)\le d$.
\end{lemma}

\subsection{Proof of Theorem~\ref{thm:md}}\label{sec:md-proof}

We are now ready to combine the auxiliary estimates in the proof of Theorem~\ref{thm:md}. We first derive the change-of-measure identity \eqref{md:tailidentity}, which separates the exponential factor from the weighted tail integral. Lemmas~\ref{md:cgf} and \ref{md:tilted-normal} estimate these two parts, while Lemma~\ref{md:weighted} converts the distributional bound into the required integral comparison.

\begin{proof}[Proof of Theorem~\ref{thm:md}]
Fix a base parameter $\bbeta_0$ and a sign $\epsilon\in\{-1,1\}$. We first take $H$ without isolated vertices and $n$ sufficiently large. We treat the remaining sample sizes and graphs with isolated vertices at the end of the proof.

\proofstep{1} We start with thresholds $1\le x\le L\sqrt n$ and express the tail under the tilted measure. Choose $t_\epsilon$ as in \eqref{md:center-equation}, whose existence in the allowed range follows from Lemma~\ref{md:path-input}. This choice puts the tail event at threshold zero, because \eqref{md:Yt} and \eqref{md:center-equation} imply
\[
\epsilon f_{H,\bbeta_0}(Y_{\bbeta(t_\epsilon)})-x
=\epsilon Y_{t_\epsilon}.
\]
Let $t=t_\epsilon$. Under the tilted measure, $f_{H,\bbeta_0}(Y_{\bbeta(t)})=\epsilon x+Y_t$ and $tf_{H,\bbeta_0}(Y_{\bbeta(t)})=|t|x+|t|\epsilon Y_t$. By \eqref{eq:model} and \eqref{md:path}, the inverse density ratio is $p_{\bbeta_0}(y)/p_{\bbeta(t)}(y)=\exp\{M_H(t)-tf_{H,\bbeta_0}(y)\}$. Applying this ratio to each configuration, we obtain
\begin{equation}\label{md:tailidentity}
\mathbb P(\epsilon W_{H,\bbeta_0}>x)
=e^{M_H(t_\epsilon)-|t_\epsilon|x}
\mathbb E\left[e^{-|t_\epsilon|\epsilon Y_{t_\epsilon}}
\mathbf1_{\{\epsilon Y_{t_\epsilon}>0\}}\right].
\end{equation}
The expectation on the right is under the tilted measure, whereas the probability on the left is under the base measure. We compare this weighted expectation first with $Q(|t_\epsilon|)$, and then compare $Q(|t_\epsilon|)$ with $Q(x)$. Dividing \eqref{md:tailidentity} by $\Phibar(x)=e^{-x^2/2}Q(x)$ from \eqref{md:Q}, we obtain
\begin{equation}\label{md:three-factors}
\frac{\mathbb P(\epsilon W_{H,\bbeta_0}>x)}{\overline\Phi(x)}
=\exp\{M_H(t_\epsilon)-|t_\epsilon|x+x^2/2\}
\times
\frac{\mathbb E[e^{-|t_\epsilon|\epsilon Y_{t_\epsilon}}
\mathbf1_{\{\epsilon Y_{t_\epsilon}>0\}}]}
{Q(|t_\epsilon|)}
\frac{Q(|t_\epsilon|)}{Q(x)}.
\end{equation}
The denominators in \eqref{md:three-factors} are positive. Both \eqref{md:tailidentity} and \eqref{md:three-factors} use the original strict tail event.

\proofstep{2} We now estimate the three factors in \eqref{md:three-factors}, keeping $1\le x\le L\sqrt n$. For the exponential factor, we substitute the cumulant expansion \eqref{md:cgfexpansion} and use $|t_\epsilon|x=\epsilon t_\epsilon x$ to obtain
\[
M_H(t_\epsilon)-|t_\epsilon|x+\frac{x^2}{2}
=\frac{(t_\epsilon-\epsilon x)^2}{2}
+\frac{\gamma_{H,0}t_\epsilon^3}{6\sqrt N}
+O\!\left(\frac{x^2}{n}+\frac{x^4}{N}\right).
\]
By \eqref{md:center-root}, $|t_\epsilon-\epsilon x|\le Cx^2/\sqrt N$ and $|t_\epsilon|\le Cx$. Thus the square is $O(x^4/N)$, and factoring the difference of cubes gives $|t_\epsilon^3-\epsilon x^3|/\sqrt N\le Cx^4/N$. Hence
\begin{equation}\label{md:rateexpansion}
M_H(t_\epsilon)-|t_\epsilon|x+\frac{x^2}{2}
=\frac{\epsilon\gamma_{H,0}x^3}{6\sqrt N}
+O\!\left(\frac{x^2}{n}+\frac{x^4}{N}\right).
\end{equation}

For the second factor in \eqref{md:three-factors}, combine \eqref{md:tiltBE} and \eqref{md:weightedbound}, using the reflection bound in Lemma~\ref{md:weighted} when $\epsilon=-1$. After division by $Q(|t_\epsilon|)\ge c/(1+|t_\epsilon|)$ from \eqref{md:Qbounds}, the relative error is at most
\[
C(1+|t_\epsilon|)
\left\{\frac{1+|t_\epsilon|}{n}
+\frac{|t_\epsilon|^3}{N}\right\}.
\]
Since $x\ge1$ and $|t_\epsilon|\asymp x$ by \eqref{md:center-root}, this proves
\begin{equation}\label{md:prefactor1}
\frac{\mathbb E[e^{-|t_\epsilon|\epsilon Y_{t_\epsilon}}
\mathbf1_{\{\epsilon Y_{t_\epsilon}>0\}}]}{Q(|t_\epsilon|)}
=1+O\!\left(\frac{1+x^2}{n}+\frac{x^4}{N}\right).
\end{equation}

For the last factor in \eqref{md:three-factors}, namely $Q(|t_\epsilon|)/Q(x)$, \eqref{md:center-root} implies $\bigl||t_\epsilon|-x\bigr|\le Cx^2/\sqrt N$. Using \eqref{md:Qbounds} and $|t_\epsilon|\asymp x$, we obtain
\begin{equation}\label{md:prefactor2}
\left|\frac{Q(|t_\epsilon|)}{Q(x)}-1\right|
\le C\frac{\bigl||t_\epsilon|-x\bigr|}{1+x}
\le C\frac{x^2}{(1+x)\sqrt N}\le C\frac xn.
\end{equation}

\proofstep{3} We first combine the three estimates for $1\le x\le L\sqrt n$, then include $0\le x\le1$ to prove
\begin{equation}\label{md:corrected}
\frac{\mathbb P(\epsilon W_{H,\bbeta_0}>x)}{\overline\Phi(x)}
=\exp\!\left\{\frac{\epsilon\gamma_{H,0}x^3}{6\sqrt N}\right\}
\left[1+O\!\left(\frac{1+x^2}{n}+\frac{x^4}{N}\right)\right],
\quad 0\le x\le L\sqrt n.
\end{equation}
For the first range, we use the common error bound
\begin{equation}\label{md:error-function}
R_n(x)=\frac{1+x^2}{n}+\frac{x^4}{N}
\le1+L^2+4L^4,\qquad 0\le x\le L\sqrt n.
\end{equation}
Combining \eqref{md:rateexpansion}, \eqref{md:prefactor1} and \eqref{md:prefactor2}, the product in \eqref{md:three-factors}, after removal of the cubic exponential, has the form $e^u(1+r_1)(1+r_2)$ with $|u|+|r_1|+|r_2|\le CR_n(x)$. The inequalities $|e^u-1|\le e^{|u|}|u|$ and $|(1+r_1)(1+r_2)-1|\le|r_1|+|r_2|+|r_1r_2|$, together with $R_n(x)^2\le(1+L^2+4L^4)R_n(x)$, give
\begin{equation}\label{md:bounded-product}
|e^u(1+r_1)(1+r_2)-1|\le C_LR_n(x).
\end{equation}
Substituting \eqref{md:bounded-product} into \eqref{md:three-factors} proves \eqref{md:corrected} for $1\le x\le L\sqrt n$.

For $0\le x\le1$, the normal tail is bounded away from zero, so Theorem~\ref{thm:normal} or \ref{thm:signed} gives relative error $O(n^{-1})$. For the lower tail, use the reflection bound in Lemma~\ref{md:weighted}. The cubic exponent is also $O(n^{-1})$ on this interval, so \eqref{md:corrected} holds throughout $0\le x\le L\sqrt n$ for all sufficiently large $n$.

To include $2\le n<n_0$, note that $\Phibar(x)\ge\Phibar(L\sqrt{n_0})>0$ and that the absolute cubic exponent is bounded by a fixed $C_{L,n_0}$. Since probabilities are at most one,
\begin{equation}\label{md:finite-n-tail}
\left|
\frac{\mathbb P(\epsilon W_{H,\bbeta_0}>x)}
{\overline\Phi(x)\exp\{\epsilon\gamma_{H,0}x^3/(6\sqrt N)\}}-1
\right|
\le 1+\frac{e^{C_{L,n_0}}}{\overline\Phi(L\sqrt{n_0})}.
\end{equation}
Since \eqref{md:error-function} gives $R_n(x)\ge1/n\ge1/n_0$, \eqref{md:finite-n-tail} is absorbed by a common increase in the error constant. This proves the explicit-scale versions for graphs without isolated vertices.

\proofstep{4} We complete the proof by passing to the true standard deviation and then removing the restriction on isolated vertices. Both steps amount to changing the threshold by $O(x/n)$. We first estimate how such a change affects the normal tail and the cubic factor. Suppose that $H$ has no isolated vertices and let
\[
r_{H,n}=\widehat\sigma_{H,\bbeta_0}/\sigma_{H,\bbeta_0}=1+O(n^{-1}).
\]
This follows from \eqref{stein:variance-calibration} and \eqref{stein:scale-ratio}, and positivity for $n\ge v_H$ is part of \eqref{stein:finite-true-bounds}. Therefore,
\[
\mathbb P(\epsilon\widehat W_{H,\bbeta_0}>x)
=\mathbb P(\epsilon W_{H,\bbeta_0}>r_{H,n}x).
\]
Put $y=r_{H,n}x$, so $|y-x|\le Cx/n$. By \eqref{md:Q} and \eqref{md:Qbounds}, we obtain
\begin{equation}\label{md:variancechange}
\frac{\overline\Phi(y)}{\overline\Phi(x)}
=1+O\!\left(\frac{x+x^2}{n}\right),\qquad
\frac{|y^3-x^3|}{\sqrt N}\le C\frac{x^3}{n\sqrt N}.
\end{equation}
To prove the first estimate in \eqref{md:variancechange}, use $|(\log Q)'(u)|\le C/(1+u)$ from \eqref{md:Qbounds}. For large $n$, $r_{H,n}\in[1/2,2]$, so $1+u\asymp1+x$ between $x$ and $y$. Hence
\begin{equation}\label{md:normal-tail-rescaling}
\left|\log\frac{\overline\Phi(y)}{\overline\Phi(x)}\right|
\le \frac12|y^2-x^2|
+\left|\int_x^y\frac{Q'(u)}{Q(u)}\,du\right|
\le C\left\{\frac{x^2}{n}+\frac{x}{n(1+x)}\right\}
\le C\frac{x+x^2}{n}.
\end{equation}
The right-hand side of \eqref{md:normal-tail-rescaling} is bounded for $x\le L\sqrt n$, so exponentiation gives the first part of \eqref{md:variancechange}. The second follows from $|y^3-x^3|=x^3|r_{H,n}^3-1|\le Cx^3/n$.

We apply \eqref{md:corrected} at $y=r_{H,n}x$, with $2L$ in place of $L$. By \eqref{md:error-function} and \eqref{md:variancechange},
\[
R_n(y)\le CR_n(x),\qquad
\frac{x+x^2}{n}+\frac{x^3}{n\sqrt N}\le C_LR_n(x),
\quad 0\le x\le L\sqrt n.
\]
Here $x/\sqrt N\le C_Ln^{-1/2}$ by \eqref{eq:edges}. Thus the changes in the normal tail, cubic exponent and error term are all $O(R_n(x))$. Combining \eqref{md:corrected} and \eqref{md:variancechange} using \eqref{md:bounded-product} proves the true-scale expansion with the same coefficient for both signs. The remaining finitely many sample sizes with $n\ge v_H$ are covered by \eqref{md:finite-n-tail}.

For isolated vertices, let $H_0,b_n$ be as in \eqref{stein:isolates}. We use the same tail-rescaling argument with
\[
y=x/b_n,\qquad |y-x|\le C_Hx/n,\qquad
R_n(y)\le C_HR_n(x)
\]
for large $n$, by \eqref{stein:isolates} and \eqref{md:error-function}. Since $e_H=e_{H_0}$, the coefficients in \eqref{eq:gamma} agree. Applying \eqref{md:corrected} to $H_0$ and \eqref{md:variancechange}--\eqref{md:normal-tail-rescaling} with scale $b_n^{-1}$, then combining the errors by \eqref{md:bounded-product}, proves the claim for $H$. The true-standardized variables coincide by \eqref{stein:isolates}. The bound \eqref{md:finite-n-tail} covers all remaining sample sizes, including $n<v_H$ for the explicit-scale statement.
\end{proof}

\appendix
\section{Variance, covariance and moment estimates}\label{sec:ferro}

To prove the remaining auxiliary lemmas, we need variance and covariance inequalities under the Gibbs measure. We first derive deterministic coordinate bounds, then use Glauber dynamics to establish a Poincar\'e inequality and a covariance estimate. Their applications to local fields and the edge-count remainder will provide the moment bounds for $f_H$, $R_H$ and $K_H$ used in Appendix~\ref{sec:normal-lemmas}.

Outside the proof of the perturbation lemma, $\bbeta$ denotes the actual parameter, including the added coefficients. On each fixed compact family, \eqref{eq:compact}--\eqref{eq:compact-scales} follow from the deterministic fixed-point argument in Step~1 of the proof of Theorem~\ref{thm:residual}; Lemma~\ref{lem:uniform-input} supplies the required fixed-point properties for signed perturbations. We write $Y=Y_{\bbeta}$, $T=T_{\bbeta}$, $p=p_{\bbeta}$, $\widetilde p=\widetilde p_{\bbeta}$, $\mu=\mu_{\bbeta}$, $\rho=\rho_{\bbeta}$, $\sigma=\sigma_{\bbeta}$ and $\varphi=\varphi_{\bbeta}$.

\subsection{Coordinate estimates}\label{sec:coordinate-bounds}

We begin with differences of the energy, which determine the conditional probabilities in \eqref{eq:conditional}. For a fixed host edge $s$, we choose one of the $e_j$ source edges and one of its two endpoint orientations. In \eqref{eq:count}, the remaining vertices then have at most $n^{v_j-2}$ images. Hence
\begin{equation}
\sup_{n\ge2}\sup_{s,y}|L_{s,\bbeta}(y)|
\le2\sum_{j=1}^k|\beta_j|e_j\le C.
\label{pi:field-bound}
\end{equation}
Furthermore, for some fixed $c>0$,
\begin{equation}
c\le q_{s,\bbeta}(y)\le1-c
\qquad(n\ge2,\ s\in\cI,\ y\in\{0,1\}^{\cI}).
\label{pi:conditional-bounds}
\end{equation}
For a fixed simple graph $H$ and $s\in\cI$, let $\Hom(H,G_y,s)$ consist of the following injective maps $\iota:\mathcal V(H)\hookrightarrow[n]$. One edge of $H$ must map to $s$, and every other edge must map to an edge of $G_y$. The edge $s$ itself need not be present in $G_y$. We call $|\Hom(H,G_y,s)|$ the rooted count (see Fang, Liu, Shao and Zhao \cite[equation (5.7)]{FLSZ}). Since an injection maps at most one source edge to $s$, \eqref{eq:count} and \eqref{eq:force} give
\begin{equation}\label{loc:rootcount}
|\Hom(H,G_y,s)|=\Delta_s|\Hom(H,G_y)|
=\sum_{a\in\mathcal E(H)}
\sum_{\substack{\iota:\mathcal V(H)\hookrightarrow[n]\\\iota(a)=s}}
\prod_{b\in\mathcal E(H)\setminus\{a\}}y_{\iota(b)}.
\end{equation}
The root coordinate is removed from the product, so this function does not depend on $y_s$. For $n\ge v_H$,
\begin{equation}\label{loc:rootcountbound}
0\le |\Hom(H,G_y,s)|
\le2e_H\prod_{r=2}^{v_H-1}(n-r)\le2e_H n^{v_H-2}.
\end{equation}
It vanishes when $n<v_H$. We repeatedly use
\begin{equation}\label{loc:fallingcomparison}
0\le1-\frac{\prod_{r=2}^{v_H-1}(n-r)}{n^{v_H-2}}
=1-\prod_{r=2}^{v_H-1}(1-r/n)
\le\frac1n\sum_{r=2}^{v_H-1}r\le\frac{C_H}{n}.
\end{equation}
For distinct host edges $s,r$, another difference gives
\begin{equation}\label{loc:double-root-count}
\Delta_{sr}|\Hom(H,G_y)|
=\sum_{\substack{a,b\in\mathcal E(H)\\a\ne b}}
\sum_{\substack{\iota:\mathcal V(H)\hookrightarrow[n]\\
\iota(a)=s,\ \iota(b)=r}}
\prod_{d\in\mathcal E(H)\setminus\{a,b\}}y_{\iota(d)}.
\end{equation}
The two host edges fix three vertices if they meet and four otherwise. There are only a fixed number of choices of source edges and endpoint orientations. Multiplying these counts by $|\beta_j|n^{2-v_j}$ and summing gives
\begin{equation}\label{loc:doublecount}
\|\Delta_rL_s\|_\infty\le
\begin{cases}C/n,&|r\cap s|=1,\\ C/n^2,&r\cap s=\varnothing,
\end{cases}
\quad(r\ne s),\qquad \Delta_sL_s=0.
\end{equation}

We write $\mathbf0$ and $\mathbf1$ for the empty and complete graph configurations in $\{0,1\}^{\cI}$, reserving $\mathbf1_{\{\cdot\}}$ for indicators. For deterministic bounds on coordinate differences, let $|\bbeta|$ be the coordinatewise absolute value and define
\begin{equation}\label{resp:Adefinition}
A_{ir}=\Delta_{ir}T_{|\bbeta|}(\mathbf1)\quad(i\ne r),\qquad A_{ii}=0.
\end{equation}
The deterministic array $A$ is symmetric and nonnegative, with $\|\Delta_{ir}T_{\bbeta}\|_\infty\le A_{ir}$. Absolute coefficients are used for these pointwise bounds, while expectations retain the actual parameter $\bbeta$. Marking a second source edge in \eqref{loc:rootcount}, we obtain, for large $n$,
\begin{equation}\label{resp:rowsum}
\sum_{r=1}^N A_{ir}
=2\sum_{j=2}^k|\beta_j|e_j(e_j-1)
\frac{\prod_{a=2}^{v_j-1}(n-a)}{n^{v_j-2}}
\le2\Phi_{|\bbeta|}'(1)=4D(\bbeta)\le C.
\end{equation}
The inequality holds for all $n$, because terms with no injection vanish. Combining \eqref{resp:rowsum} with \eqref{loc:doublecount} and $\Psi'\le1/4$, we obtain
\begin{equation}\label{resp:matrixbounds}
\max_{i,r}A_{ir}\le C/n,\qquad
\max_i\sum_{r=1}^N A_{ir}^2\le C/n,\qquad
\|\Delta_rq_i\|_\infty\le A_{ir}/4.
\end{equation}
Symmetry gives the same column bounds. Marking a third source edge also gives
\begin{equation}\label{resp:thirdsum}
\sum_{j=1}^N \|\Delta_{irj}T_{\bbeta}\|_\infty\le C_{\cH}A_{ir}.
\end{equation}
Indeed, after two source edges have been marked, there are at most $\max_j e_j$ remaining choices for a third edge in each monomial. Apply this observation to $T_{|\bbeta|}$, whose mixed differences have nonnegative coefficients and attain their maxima at $\mathbf1$.

\subsection{Variance and covariance bounds}\label{sec:heat-bath}\label{sec:signed-input}

We now pass from coordinate bounds to inequalities under the Gibbs measure. After defining single-edge Glauber dynamics, we show in Lemma~\ref{lem:poincare} how contraction over a time block yields the variance and covariance bounds. Lemmas~\ref{lem:dobrushin-input} and \ref{lem:uniform-input} then verify the required contraction.

We use the discrete-time single-edge Glauber dynamics of Fang, Liu, Shao and Zhao \cite[equation (5.34)]{FLSZ}, with the actual model parameter $\bbeta$. Let $(S_r)_{r\ge1}$ be independent random edges, each uniform on $\cI$. Independently, let $(U_r)_{r\ge1}$ be independent random variables, each uniform on $[0,1]$. Starting from $X_0^{y,\bbeta}=y$, we update the chain at each step by
\[
(X_{r+1}^{y,\bbeta})_s=
\begin{cases}
\ind_{\{U_{r+1}\le q_{s,\bbeta}(X_r^{y,\bbeta})\}},&s=S_{r+1},\\
(X_r^{y,\bbeta})_s,&s\ne S_{r+1}.
\end{cases}
\]
Thus only the selected edge is resampled, using its conditional probability from \eqref{eq:conditional}. The same update edges and uniforms couple chains from different initial states or with different parameters. Each chain evaluates its own conditional probability at its current state.

For a configuration function $F$, the expected value after one update is given by the Markov operator $P=P_{\bbeta,n}$, with the convention of Fang, Liu, Shao and Zhao \cite[definition following (5.37)]{FLSZ},
\[
P_{\bbeta,n}F(y)=\frac1N\sum_{s\in\cI}
\bigl\{q_{s,\bbeta}(y)F(y^{(s,1)})+(1-q_{s,\bbeta}(y))F(y^{(s,0)})\bigr\}.
\]
Thus $P_{\bbeta,n}$ acts on functions of the configuration. Resampling from a conditional distribution makes it reversible and invariant under $p_{\bbeta}(\cdot)$. Let $I$ denote the identity operator. Then $P^0=I$ and $P_{\bbeta,n}^rF(y)=\E F(X_r^{y,\bbeta})$ for integers $r\ge0$. At fixed $\bbeta$, we write $X_r^y=X_r^{y,\bbeta}$.

For configurations $y,z$, their Hamming distance counts the edge coordinates on which they differ and is given by
\[
d_{\rm H}(y,z)=\sum_{i=1}^N|y_i-z_i|.
\]
The corresponding Lipschitz seminorm of a configuration function $F$ is
\[
\Lip_{\rm H}(F)=\sup_{y\ne z}\frac{|F(y)-F(z)|}{d_{\rm H}(y,z)}
=\max_{1\le i\le N}\|\Delta_iF\|_\infty.
\]
The upper bound follows by changing the differing coordinates one at a time, and the reverse inequality follows by considering adjacent configurations. To pass from contraction to a spectral gap and then to a Poincar\'e inequality, we use the standard finite-state arguments in Levin and Peres \cite[Theorem 13.1 and Lemma 13.7]{LP2017}. The relation between Glauber contraction and the Poincar\'e inequality for ferromagnetic ERGMs is also used by Ganguly and Nam \cite[Theorem 2.1]{GN}. Summing the coordinate differences over time will additionally yield the following covariance bound.

\begin{lemma}\label{lem:poincare}
For every fixed parameter family in Theorem~\ref{thm:normal} or \ref{thm:signed}, there is $\kappa<\infty$ such that, for all $n\ge2$ and all configuration functions $F,G$, we have the Poincar\'e inequality
\begin{equation}\label{eq:poincare}
\Var(F(Y))\le\kappa\sum_{i=1}^N \E|\Delta_iF(Y)|^2
\end{equation}
and the covariance bound
\begin{equation}\label{eq:covariance-bound}
|\Cov(F(Y),G(Y))|
\le\kappa\left(\sum_{i=1}^N \|\Delta_iF\|_\infty\right)
\max_j\|\Delta_jG\|_\infty.
\end{equation}
If $\B\subset\mathcal D$ and $\delta=1-\sup_{\bbeta\in\B}D(\bbeta)>0$, one may take $\kappa=(4\delta)^{-1}$.
\end{lemma}
\begin{proof}[Proof of Lemma~\ref{lem:poincare}]
The two inequalities use different consequences of the same contraction estimate. For the Poincar\'e inequality, we pass through the spectral gap of the Markov operator. For the covariance bound, we sum the coordinate differences of its iterates over time. The common starting point is the following estimate, which is proved in Lemmas~\ref{lem:dobrushin-input} and \ref{lem:uniform-input} without using the present lemma. For all sufficiently large $n$, there are integers $1\le\tau_n\le C_0N$ and a common constant $K<\infty$ such that
\begin{equation}\label{pi:block-seminorm}
\Lip_{\rm H}(P^{\tau_n}G)\le\tfrac12\Lip_{\rm H}(G),\qquad
\max_{0\le r<\tau_n}\Lip_{\rm H}(P^rG)\le K\Lip_{\rm H}(G).
\end{equation}
The unperturbed ferromagnetic case is included by setting the added coefficients in Lemma~\ref{lem:uniform-input} equal to zero.

For fixed $n$, work in $L^2(p_{\bbeta}(\cdot))$ with inner product $\langle F,G\rangle=\E[F(Y)G(Y)]$. The single-coordinate conditional expectation operators satisfy
\[
P_iF(y)=\E[F(Y)\mid Y_{-i}=y_{-i}],\qquad P=\frac1N\sum_{i=1}^NP_i.
\]
Each $P_i$ is an orthogonal projection, so $P$ is self-adjoint and positive semidefinite, with spectrum in $[0,1]$. By \eqref{pi:conditional-bounds}, single-edge changes have positive transition probabilities and connect all configurations. Hence the chain is irreducible, constants form the eigenspace for eigenvalue $1$, and the mean-zero subspace is invariant under $P$.

Every function on this finite space has finite Hamming Lipschitz seminorm, so we can apply \eqref{pi:block-seminorm} to an eigenfunction. If $PF=\lambda F$ and $F$ is nonconstant, then $\Lip_{\rm H}(F)>0$ because single-edge paths connect all configurations. Hence
\[
\lambda^{\tau_n}\Lip_{\rm H}(F)=\Lip_{\rm H}(P^{\tau_n}F)\le\tfrac12\Lip_{\rm H}(F),\quad
\tfrac12\le1-\lambda^{\tau_n}=(1-\lambda)\sum_{r=0}^{\tau_n-1}\lambda^r\le\tau_n(1-\lambda).
\]
Let $\operatorname{gap}(P)$ denote one minus the largest eigenvalue on the mean-zero subspace. Since $\tau_n\le C_0N$, the eigenfunction estimate gives
\begin{equation}\label{pi:spectral-gap}
\operatorname{gap}(P)\ge\frac1{2\tau_n}\ge\frac1{2C_0N}.
\end{equation}
To relate this gap to coordinate fluctuations, condition on $Y_{-i}$ in \eqref{eq:conditional}. Only $Y_i$ remains random, with success probability $q_i$, and the differences between the two endpoint values are $\Delta_iF$ and $\Delta_iG$ by \eqref{eq:force}. Hence
\[
\langle F,(I-P_i)G\rangle
=\E\{\Cov(F(Y),G(Y)\mid Y_{-i})\}
=\E\{q_i(1-q_i)\Delta_iF\,\Delta_iG\}.
\]
Averaging over $i$ yields the bilinear Dirichlet identity
\begin{equation}\label{pi:dirichlet}
\langle F,(I-P)G\rangle
=\frac1N\sum_{i=1}^N\E\{q_i(1-q_i)\Delta_iF\,\Delta_iG\}.
\end{equation}
Let $F-\E F(Y)$ be expanded in an orthonormal eigenbasis of the mean-zero subspace. Then $\langle F,(I-P)F\rangle\ge\operatorname{gap}(P)\Var(F(Y))$. Combining this with \eqref{pi:spectral-gap}--\eqref{pi:dirichlet} and $q_i(1-q_i)\le1/4$, we obtain
\[
\Var(F(Y))
\le\frac1{4N\operatorname{gap}(P)}
\sum_{i=1}^N\E|\Delta_iF(Y)|^2
\le C\sum_{i=1}^N\E|\Delta_iF(Y)|^2.
\]
This proves \eqref{eq:poincare} for large $n$.

The Poincar\'e inequality is now established. To obtain \eqref{eq:covariance-bound} with the stated coordinate norms, we also use the bound within each time block. If $r=j\tau_n+s$ with $0\le s<\tau_n$, then \eqref{pi:block-seminorm} implies $\Lip_{\rm H}(P^rG)\le K2^{-j}\Lip_{\rm H}(G)$. Summing this estimate, we obtain
\begin{equation}\label{pi:summed-seminorm}
\sum_{r=0}^{\infty}\Lip_{\rm H}(P^rG)
\le K\tau_n\sum_{j=0}^{\infty}2^{-j}\Lip_{\rm H}(G)
=2K\tau_n\Lip_{\rm H}(G)\le CN\Lip_{\rm H}(G).
\end{equation}
For each integer $M\ge1$, the identity $G-P^MG=\sum_{r=0}^{M-1}(I-P)P^rG$, invariance of $p_{\bbeta}$, and \eqref{pi:dirichlet} yield
\begin{equation}\label{pi:telescoping-covariance}
\Cov(F(Y),G(Y))-\Cov(F(Y),(P^MG)(Y))=\frac1N\sum_{r=0}^{M-1}\sum_{i=1}^N\E\{q_i(1-q_i)\Delta_iF\,\Delta_i(P^rG)\}.
\end{equation}
By invariance, we have $\E[(P^MG)(Y)]=\E G(Y)$. This mean lies between the minimum and maximum of $P^MG$, whose difference is at most $N\Lip_{\rm H}(P^MG)$. Consequently, \eqref{pi:block-seminorm} implies, for each fixed $n$,
\[
\|P^MG-\E G(Y)\|_\infty
\le N\Lip_{\rm H}(P^MG)\longrightarrow0
\qquad(M\longrightarrow\infty).
\]
Taking absolute values in \eqref{pi:telescoping-covariance}, using \eqref{pi:summed-seminorm}, and then letting $M\to\infty$, we obtain
\[
|\Cov(F(Y),G(Y))|\le\frac1{4N}\left(\sum_{i=1}^N\|\Delta_iF\|_\infty\right)\sum_{r=0}^{\infty}\Lip_{\rm H}(P^rG)\le C\left(\sum_{i=1}^N\|\Delta_iF\|_\infty\right)\Lip_{\rm H}(G).
\]
This proves \eqref{eq:covariance-bound} for large $n$. On $\mathcal D$, applying \eqref{pi:dobrushin-contraction} to eigenfunctions and summing it over time yields a spectral gap at least $\delta/N$ and the sharper sum $\sum_{r=0}^{\infty}\Lip_{\rm H}(P^rG)\le(N/\delta)\Lip_{\rm H}(G)$. Thus both constants are $(4\delta)^{-1}$.

For each remaining sample size, we compare with the product Bernoulli$(1/2)$ measure $\nu_n$. Define
\begin{equation}\label{pi:density-bounds}
r_{n,\bbeta}(y)=p_{\bbeta}(y)/\nu_n(y),\qquad
m_n=\inf_{\bbeta,y}r_{n,\bbeta}(y),\qquad
M_n=\sup_{\bbeta,y}r_{n,\bbeta}(y).
\end{equation}
The extrema range over the fixed compact parameter family and all configurations, so $0<m_n\le M_n<\infty$. For $V\sim\nu_n$, we apply variance tensorization to obtain
\begin{equation}\label{pi:product}
\Var(F(V))\le\sum_{i=1}^N\E\Var(F(V)\mid V_{-i})
=\tfrac14\sum_{i=1}^N\E|\Delta_iF(V)|^2.
\end{equation}
Using $\E F(V)$ as a comparison constant and then \eqref{pi:density-bounds}--\eqref{pi:product}, we obtain
\begin{equation}\label{pi:density-comparison}
\Var(F(Y))\le M_n\Var(F(V))
\le\frac{M_n}{4m_n}\sum_{i=1}^N\E|\Delta_iF(Y)|^2.
\end{equation}
Cauchy--Schwarz and \eqref{pi:density-comparison} bound the covariance by
\[
\frac{M_n}{4m_n}
\left(\sum_{i=1}^N\|\Delta_iF\|_\infty^2\right)^{1/2}
\left(\sum_{i=1}^N\|\Delta_iG\|_\infty^2\right)^{1/2}
\le\frac{M_n\sqrt N}{4m_n}
\left(\sum_{i=1}^N\|\Delta_iF\|_\infty\right)\Lip_{\rm H}(G).
\]
Taking the maximum over the finitely many remaining $n$ completes both bounds.
\end{proof}

We next verify the contraction estimate used in Lemma~\ref{lem:poincare}. Under the Dobrushin condition \eqref{eq:dobrushin}, the sum of the absolute influences is below one at every configuration. We can therefore couple two chains so that their expected distance decreases at each step, as follows.

\begin{lemma}\label{lem:dobrushin-input}
If $D(\bbeta)\le1-\delta$ for some $\delta>0$, then, for every $r\ge0$, we have
\begin{equation}\label{pi:dobrushin-contraction}
\Lip_{\rm H}(P_{\bbeta,n}^rF)
\le(1-\delta/N)^r\Lip_{\rm H}(F).
\end{equation}
In particular, \eqref{pi:block-seminorm} holds uniformly on every compact subset of $\mathcal D$.
\end{lemma}
\begin{proof}[Proof of Lemma~\ref{lem:dobrushin-input}]
Couple the two chains with the same update edge and uniform random variable, as in Reinert and Ross \cite[Lemma 3.2 and Theorem 1.7]{RR}. By \eqref{resp:rowsum}--\eqref{resp:matrixbounds}, which also hold for disconnected energy graphs, we have
\begin{equation}\label{pi:dobrushin-influence}
\max_{1\le s\le N}\sum_{\substack{i=1\\i\ne s}}^N\|\Delta_sq_i\|_\infty
\le\tfrac14\max_{1\le s\le N}\sum_{i=1}^N A_{is}
\le D(\bbeta)\le1-\delta.
\end{equation}
Given configurations $y,z$, an update at edge $i$ replaces the old disagreement $|y_i-z_i|$ by a disagreement of probability $|q_i(y)-q_i(z)|$. Thus
\begin{equation}\label{pi:coupled-update}
\E d_{\rm H}(X_1^y,X_1^z)
=\left(1-\frac1N\right)d_{\rm H}(y,z)
+\frac1N\sum_{i=1}^N|q_i(y)-q_i(z)|.
\end{equation}
A shortest coordinate path from $y$ to $z$ and \eqref{pi:dobrushin-influence} give $\sum_{i=1}^N|q_i(y)-q_i(z)|\le(1-\delta)d_{\rm H}(y,z)$. Substituting this bound into \eqref{pi:coupled-update} and iterating yields
\[
\E d_{\rm H}(X_r^y,X_r^z)
\le(1-\delta/N)^r d_{\rm H}(y,z).
\]
For adjacent initial states, $|P^rF(y)-P^rF(z)|\le\Lip_{\rm H}(F)\E d_{\rm H}(X_r^y,X_r^z)$, proving \eqref{pi:dobrushin-contraction}. Take $\tau_n=\lceil N\log2/\delta\rceil$. Then \eqref{pi:block-seminorm} holds with $K=1$.

\end{proof}

For small signed perturbations of a ferromagnetic subcritical model, we instead establish contraction where the normalized rooted counts are close to their mean-field values. We first bring the perturbed chains into this region by placing them between two ferromagnetic comparison chains and applying the relaxation estimate of Ganguly and Nam \cite[Theorem 4.2]{GN} to those chains. After $O(N)$ updates, the comparison chains have rooted counts close to their fixed-point values with high probability. We then control updates along paths between the perturbed states to obtain the same time-block estimate. The next lemma also establishes stability of the fixed point.

\begin{lemma}\label{lem:uniform-input}
Let $\B\subset\mathcal S_+$ be an arbitrary fixed nonempty compact set. There exists $\eta>0$, depending only on $\B$ and $\cH$, such that, for every fixed added family $\mathcal J$, the models on $\B_{\mathcal J}(2\eta)$ have a unique stable fixed point, satisfy common positive bounds on $p_{\btheta}$, $1-p_{\btheta}$ and $\rho_{\btheta}$, and satisfy \eqref{pi:block-seminorm} for all sufficiently large $n$. The constants in the block estimate can be chosen uniformly on that family.
\end{lemma}
\begin{proof}[Proof of Lemma~\ref{lem:uniform-input}]
Throughout the proof, $\bbeta\in\B$ denotes the original ferromagnetic parameter and $\btheta=(\bbeta,\bh)$ the perturbed parameter. We begin by controlling the effect of the added coefficients on the fields and the scalar fixed point. We then identify a region in which the perturbed updates contract. To reach that region, we use the relaxation of two ferromagnetic comparison chains that bound both target chains. Finally, we combine this relaxation with a bound on the distance growth before it occurs to obtain \eqref{pi:block-seminorm}.

\proofstep{1} We first bound the effect of the perturbation. The same bounds will be used to preserve the fixed point and to compare the update probabilities later in the proof. Applying the counting estimates \eqref{loc:rootcountbound} and \eqref{resp:rowsum} to each added graph $J_\ell$, we obtain
\begin{equation}\label{pi:perturbation-size}
\begin{aligned}
\max_{1\le i\le N}\|\Delta_i(T_{\btheta}-T_{\bbeta})\|_\infty
&\le2\sum_{\ell=1}^m e_{J_\ell}|h_\ell|,\\
\max_{1\le s\le N}\sum_{\substack{i=1\\i\ne s}}^N
\|\Delta_{is}(T_{\btheta}-T_{\bbeta})\|_\infty
&\le2\sum_{\ell=1}^m e_{J_\ell}(e_{J_\ell}-1)|h_\ell|.
\end{aligned}
\end{equation}
By \eqref{eq:signed-family}, the sum of these two upper bounds is $2\sum_{\ell=1}^m e_{J_\ell}^2|h_\ell|\le4\eta$ for $\btheta\in\B_{\mathcal J}(2\eta)$. Moreover, by \eqref{eq:meanfield}, we have $\|\Phi_{\btheta}-\Phi_{\bbeta}\|_\infty+\|\Phi_{\btheta}'-\Phi_{\bbeta}'\|_\infty\le\sum_{\ell=1}^m e_{J_\ell}^2|h_\ell|$ on $[0,1]$.

We use these bounds first for the fixed-point equation. Let $g_{\bbeta}(u)=u-\varphi_{\bbeta}(u)$. Near its unique stable root the derivative is positive, while away from the root the function is bounded away from zero. By compactness, we can choose common $r,c,c_*>0$ such that
\begin{equation}\label{md:uniform-stability}
\inf_{\substack{\bbeta\in\B\\|u-p_{\bbeta}|\le r}}g_{\bbeta}'(u)\ge c,
\qquad
\inf_{\substack{\bbeta\in\B,\ u\in[0,1]\\|u-p_{\bbeta}|\ge r}}
|g_{\bbeta}(u)|\ge c_*,
\end{equation}
with $[p_{\bbeta}-r,p_{\bbeta}+r]\subset(0,1)$ and opposite signs at its endpoints. For continuously differentiable $f$ on $[0,1]$, write $\|f\|_{C^1[0,1]}=\sup_{[0,1]}|f|+\sup_{[0,1]}|f'|$. Since $\Psi$ has bounded first two derivatives, \eqref{eq:meanfield} yields
\begin{equation}\label{pi:scalar-perturbation}
\|\varphi_{\btheta}-\varphi_{\bbeta}\|_{C^1[0,1]}
\le C_{\B,\cH}\bigl(\|\Phi_{\btheta}-\Phi_{\bbeta}\|_\infty
+\|\Phi_{\btheta}'-\Phi_{\bbeta}'\|_\infty\bigr)
\le C_{\B,\cH}\sum_{\ell=1}^m e_{J_\ell}^2|h_\ell|.
\end{equation}
Let $g_{\btheta}(u)=u-\varphi_{\btheta}(u)$. Choose $\eta$ so that the upper bound in \eqref{pi:scalar-perturbation} is at most $\min\{c,c_*\}/2$. Combining \eqref{md:uniform-stability} and \eqref{pi:scalar-perturbation}, we obtain
\[
g_{\btheta}'(u)\ge c/2\ \text{if }|u-p_{\bbeta}|\le r,\qquad
|g_{\btheta}(u)|\ge c_*/2\ \text{if }u\in[0,1],\ |u-p_{\bbeta}|\ge r.
\]
These bounds preserve the opposite endpoint signs and exclude roots outside the interval. Inside it, strict monotonicity and the intermediate value theorem give exactly one zero $p_{\btheta}$. By the mean value theorem on this interval, we have
\begin{equation}\label{pi:root-displacement}
|p_{\btheta}-p_{\bbeta}|
\le\frac2c|g_{\btheta}(p_{\bbeta})|
\le\frac2c\|\varphi_{\btheta}-\varphi_{\bbeta}\|_\infty
\le C_{\B,\cH}\sum_{\ell=1}^m e_{J_\ell}^2|h_\ell|.
\end{equation}
By \eqref{pi:root-displacement}, decreasing $\eta$ if necessary preserves common positive lower bounds for $p_{\btheta}$ and $1-p_{\btheta}$. The derivative bound also gives $\rho_{\btheta}=g_{\btheta}'(p_{\btheta})\ge c/2$. For the comparison chains in Step~3, we will only shift the original edge coefficient. We choose $\eta$ small enough that the resulting set $\B_{\rm e}=\{\bbeta+s\mathbf e_1:\bbeta\in\B,\ |s|\le2\eta\}$ is a compact subset of the original $\mathcal S_+$.

If $k=1$, combine the linear terms to obtain $D(\btheta)\le\tfrac12\sum_{\ell=1}^m e_{J_\ell}^2|h_\ell|\le\eta$. We can then choose $\eta<1/2$ and apply Lemma~\ref{lem:dobrushin-input}. For the rest of the proof, assume $k\ge2$.

\proofstep{2} With the fixed point controlled, we turn to the update probabilities. We define the typical region using the original fixed point $p_{\bbeta}$ and rooted counts, since \eqref{pi:perturbation-size} controls the added interactions separately. Recall that $|\Hom(G,G_y,i)|$ in \eqref{loc:rootcount} omits the factor of the edge mapped to $i$. Let $a=\max_{1\le j\le k}v_j$ and let $\mathcal G_a$ be the finite collection, up to isomorphism, of simple graphs with at most $a$ vertices and at least two edges, allowing isolated vertices. Define
\begin{equation}\label{pi:typical-set}
\begin{aligned}
r_G(y,i)&=\left\{\frac{|\Hom(G,G_y,i)|}{2e_Gn^{v_G-2}}\right\}^{1/(e_G-1)},
\\
\mathcal T_{\bbeta}(\delta)&=\{y:|r_G(y,i)-p_{\bbeta}|<\delta\ \forall G\in\mathcal G_a,\ i\in\cI\}.
\end{aligned}
\end{equation}
The collection $\mathcal G_a$ contains every original energy graph and every graph obtained by deleting one of its edges, provided at least two edges remain. It does not depend on the added family $\mathcal J$, whose contribution is controlled by \eqref{pi:perturbation-size}. Graphs with one edge have deterministic rooted counts $2\prod_{r=2}^{v_G-1}(n-r)$ for $n\ge v_G$.

To evaluate the total change in the local fields, fix a graph $K$ and sum over the second marked host edge in \eqref{loc:double-root-count}. Choosing its source edge first gives the identity
\begin{equation}\label{pi:edge-deletion-root}
\sum_{\substack{i=1\\i\ne s}}^N\Delta_{is}|\Hom(K,G_y)|
=\sum_{b\in\mathcal E(K)}\Delta_s|\Hom(K\setminus b,G_y)|,
\end{equation}
where $K\setminus b$ retains the vertex set of $K$. An edgeless deletion contributes zero. On $\mathcal T_{\bbeta}(\delta)$, we use \eqref{pi:typical-set} for graphs with at least two edges and \eqref{loc:fallingcomparison} for graphs with one edge. Combining these estimates with \eqref{loc:rootcount} and \eqref{pi:edge-deletion-root}, we obtain
\begin{equation}\label{pi:uniform-root}
L_{i,\bbeta}(y)=2\Phi_{\bbeta}(p_{\bbeta})+O(\delta+n^{-1}),\quad
\sum_{\substack{i=1\\i\ne s}}^N\Delta_{is}T_{\bbeta}(y)
=2\Phi_{\bbeta}'(p_{\bbeta})+O(\delta+n^{-1}).
\end{equation}
To estimate $\Delta_s q_{i,\btheta}(y)$, we need the fields at $y^{(s,0)}$ and $y^{(s,1)}$, as defined in \eqref{eq:force}. These configurations need not remain in $\mathcal T_{\bbeta}(\delta)$. Instead, for $i\ne s$, we compare each field with $L_{i,\bbeta}(y)$. Combining \eqref{loc:doublecount}, \eqref{pi:perturbation-size} and \eqref{pi:uniform-root}, we have
\begin{equation}\label{pi:local-field-endpoints}
\begin{aligned}
\max_{b\in\{0,1\}}|L_{i,\btheta}(y^{(s,b)})-2\Phi_{\bbeta}(p_{\bbeta})|
&\le |L_{i,\bbeta}(y)-2\Phi_{\bbeta}(p_{\bbeta})|
+\|\Delta_sL_{i,\bbeta}\|_\infty
+\|L_{i,\btheta}-L_{i,\bbeta}\|_\infty\\
&\le C_{\B,\cH}(\delta+n^{-1})+4\eta.
\end{aligned}
\end{equation}
In \eqref{pi:local-field-endpoints}, the original energy is controlled by \eqref{loc:doublecount} and the added fields by \eqref{pi:perturbation-size}. Since $\Psi'(2\Phi_{\bbeta}(p_{\bbeta}))=p_{\bbeta}(1-p_{\bbeta})$, the mean value theorem applied to \eqref{eq:conditional} yields
\begin{equation}\label{pi:absolute-local-influence}
\begin{aligned}
|\Delta_sq_{i,\btheta}(y)|&\le\{p_{\bbeta}(1-p_{\bbeta})+C(\delta+\eta+n^{-1})\}\\
&\quad\times\left\{\Delta_{is}T_{\bbeta}(y)+\sum_{\ell=1}^m|h_\ell|n^{2-v_{J_\ell}}\Delta_{is}|\Hom(J_\ell,G_y)|\right\}.
\end{aligned}
\end{equation}
The second factor in \eqref{pi:absolute-local-influence} is a nonnegative, symmetric upper bound for $|\Delta_{is}T_{\btheta}(y)|$. Summing over $i\ne s$ and combining \eqref{pi:uniform-root}, \eqref{pi:perturbation-size} and \eqref{eq:meanfield}, we obtain
\begin{equation}\label{pi:column-contraction}
\max_{1\le s\le N}\sum_{\substack{i=1\\i\ne s}}^N|\Delta_sq_{i,\btheta}(y)|
\le\varphi_{\bbeta}'(p_{\bbeta})+C(\delta+\eta+n^{-1})\le1-\alpha,
\qquad y\in\mathcal T_{\bbeta}(\delta),
\end{equation}
where $\alpha=\tfrac12\inf_{\bbeta\in\B}\rho_{\bbeta}>0$ and the bound holds after choosing $\delta$ sufficiently small, then decreasing $\eta$ and finally taking $n$ sufficiently large, with all choices depending only on $\B$ and $\cH$. The column sum in \eqref{pi:column-contraction} is evaluated at the same configuration $y$ throughout.

\proofstep{3} The local estimate \eqref{pi:column-contraction} can be used only after the target states reach $\mathcal T_{\bbeta}(\delta)$. We now arrange this by comparison with ferromagnetic chains. Recall that $\B_{\rm e}$, introduced in Step~1, consists of the original parameters with edge coefficient shifted by at most $2\eta$. The first task is to make the relaxation estimate uniform over this family. We will prove
\begin{equation}\label{pi:uniform-ferromagnetic-relaxation}
\sup_{\bzeta\in\B_{\rm e}}
\sup_{y\in\{0,1\}^{\cI}}
\sup_{\substack{s\in\mathbb Z_{\ge0}\\s\ge\lceil C_0N\rceil}}
\Pp\{X_s^{y,\bzeta}\notin\mathcal T_{\bzeta}(\delta/2)\}
\le Ce^{-cn}.
\end{equation}
For each fixed ferromagnetic parameter, Ganguly and Nam \cite[Theorem 4.2]{GN} state the estimate of Bhamidi, Bresler and Sly \cite[Lemma 16]{BBS} in the normalization \eqref{pi:typical-set}. Their assumptions allow arbitrary real edge coefficients, as specified in Bhamidi, Bresler and Sly \cite[Definition 1]{BBS}, and their time scale agrees with $N\asymp n^2$ in \eqref{eq:edges}. To obtain common constants on $\B_{\rm e}$, we cover it by finitely many small parameter boxes and compare the chains in each box with the chains at its two corners.

For vectors $\bzeta^-<\bzeta^+$, let
\[
[\bzeta^-,\bzeta^+]=\prod_{j=1}^k[\zeta^-_j,\zeta^+_j],
\]
where vector inequalities are coordinatewise. For the original ferromagnetic family, the implicit function theorem and uniqueness of the root imply that $\mathcal S_+$ is open and $\bzeta\mapsto p_{\bzeta}$ is continuous. Compactness of $\B_{\rm e}$ therefore provides finitely many pairs $\bzeta_r^-<\bzeta_r^+$, $1\le r\le M$, such that
\begin{equation}\label{pi:finite-parameter-boxes}
\begin{gathered}
\B_{\rm e}\subseteq\bigcup_{r=1}^M\operatorname{int}[\bzeta_r^-,\bzeta_r^+],
\qquad [\bzeta_r^-,\bzeta_r^+]\subset\mathcal S_+,\\
\sup_{\bzeta\in[\bzeta_r^-,\bzeta_r^+]}
\max\{|p_{\bzeta_r^-}-p_{\bzeta}|,|p_{\bzeta_r^+}-p_{\bzeta}|\}<\delta/8,
\qquad 1\le r\le M.
\end{gathered}
\end{equation}
Fix a box in \eqref{pi:finite-parameter-boxes} and denote its corners by $\bzeta^-,\bzeta^+$. The ferromagnetic update probabilities are increasing in both the parameter and the configuration. Using the same update edge and uniform random variable for the three chains therefore preserves the order
\begin{equation}\label{pi:parameter-box-coupling}
\bzeta^-\le\bzeta\le\bzeta^+
\quad\Longrightarrow\quad
X_s^{\mathbf0,\bzeta^-}\le X_s^{y,\bzeta}\le X_s^{\mathbf1,\bzeta^+},
\qquad s\ge0.
\end{equation}
Apply Ganguly and Nam \cite[Theorem 4.2]{GN} to the two corners of the $r$th box, with tolerance $\delta/8$. We obtain constants $C_{0,r},C_r,c_r>0$ and $n_{0,r}$ such that, for $n\ge n_{0,r}$ and $s\ge\lceil C_{0,r}N\rceil$,
\begin{equation}\label{pi:corner-relaxation}
\sup_{y\in\{0,1\}^{\cI}}
\max_{\bxi\in\{\bzeta_r^-,\bzeta_r^+\}}
\Pp\left\{
\max_{\substack{G\in\mathcal G_a\text{ without isolated vertices}\\i\in\cI}}
|r_G(X_s^{y,\bxi},i)-p_{\bxi}|\ge\delta/8
\right\}\le C_re^{-c_rn}.
\end{equation}
Outside the two exceptional events in \eqref{pi:corner-relaxation}, both corner chains have rooted counts within $\delta/8$ of their own fixed points. The centers differ from $p_{\bzeta}$ by less than $\delta/8$ by \eqref{pi:finite-parameter-boxes}. Since each $r_G$ is increasing in the configuration, combining these bounds with \eqref{pi:parameter-box-coupling}, we have
\[
p_{\bzeta}-\delta/4
<r_G(X_s^{\mathbf0,\bzeta^-},i)
\le r_G(X_s^{y,\bzeta},i)
\le r_G(X_s^{\mathbf1,\bzeta^+},i)
<p_{\bzeta}+\delta/4
\]
for all $i\in\cI$ and all $G\in\mathcal G_a$ without isolated vertices. To cover the remaining graphs in the definition \eqref{pi:typical-set}, delete the isolated vertices of $G$ and call the resulting graph $G_0$. Taking a single-edge difference in \eqref{stein:isolates} and using \eqref{loc:rootcount}, we obtain
\[
r_G(y,i)=b_n^{1/(e_G-1)}r_{G_0}(y,i)
=r_{G_0}(y,i)+O_G(n^{-1}),\qquad e_G\ge2.
\]
Here $b_n=1+O_G(n^{-1})$ is the deterministic factor in \eqref{stein:isolates} for $G$ and $G_0$. By \eqref{loc:rootcountbound}, the error is uniform in $y,i$. Since $\mathcal G_a$ is finite, this error is smaller than $\delta/4$ for large $n$, leaving total error below $\delta/2$ as required in \eqref{pi:uniform-ferromagnetic-relaxation}.

Only finitely many corner estimates are involved. We therefore take
\[
C_0=\max_{1\le r\le M}C_{0,r},\qquad
c=\min_{1\le r\le M}c_r,\qquad
C=2\max_{1\le r\le M}C_r,
\]
and a common sample-size threshold above $\max_{1\le r\le M}n_{0,r}$, also large enough for the isolated-vertex error. Combining \eqref{pi:corner-relaxation} with \eqref{pi:parameter-box-coupling} and the isolated-vertex bound proves \eqref{pi:uniform-ferromagnetic-relaxation}. These constants remain valid if $\eta$ is decreased.

The corner comparisons establish uniform relaxation on $\B_{\rm e}$. We now bound the dynamics at $\btheta=(\bbeta,\bh)$ by the two ferromagnetic models with $\bbeta^\pm=\bbeta\pm2\eta\mathbf e_1$ in this family. Since $|\Hom(K_2,G_y)|=2E(y)$, these shifts change every local field by $\pm4\eta$, which bounds the added field by \eqref{pi:perturbation-size}. Consequently,
\begin{equation}\label{pi:edge-field-sandwich}
L_{i,\bbeta^-}(y)\le L_{i,\btheta}(y)\le L_{i,\bbeta^+}(y).
\end{equation}
For configurations $b^-\le y\le b^+$, combine \eqref{pi:edge-field-sandwich} with \eqref{eq:conditional} and monotonicity of the two ferromagnetic comparison probabilities. We obtain
\begin{equation}\label{pi:ordered-update-sandwich}
q_{i,\bbeta^-}(b^-)
\le q_{i,\bbeta^-}(y)
\le q_{i,\btheta}(y)
\le q_{i,\bbeta^+}(y)
\le q_{i,\bbeta^+}(b^+).
\end{equation}
The order in \eqref{pi:ordered-update-sandwich} uses monotonicity only for the comparison probabilities, not for the perturbed model. We can therefore use the same update edge and uniform random variable for two target chains, started from any $y,z$, and the two comparison chains. At each update, the displayed inequalities preserve the order, so
\begin{equation}\label{pi:common-envelopes}
X_s^{\mathbf0,\bbeta^-}\le X_s^{u,\btheta}\le X_s^{\mathbf1,\bbeta^+},
\qquad u\in\{y,z\},\ s\ge0.
\end{equation}
The bound \eqref{pi:uniform-ferromagnetic-relaxation} is centered at $p_{\bbeta^\pm}$, whereas contraction in \eqref{pi:column-contraction} uses $p_{\bbeta}$. By \eqref{pi:root-displacement}, applied to the edge shifts $\pm2\eta$, we can decrease $\eta$ so that $|p_{\bbeta^\pm}-p_{\bbeta}|<\delta/4$. For $w\in\mathcal T_{\bbeta^\pm}(\delta/2)$, the definition \eqref{pi:typical-set} then gives
\[
|r_G(w,i)-p_{\bbeta}|
\le |r_G(w,i)-p_{\bbeta^\pm}|+|p_{\bbeta^\pm}-p_{\bbeta}|
<\delta/2+\delta/4<\delta
\]
for every $G\in\mathcal G_a$ and $i\in\cI$. Hence $\mathcal T_{\bbeta^\pm}(\delta/2)\subset\mathcal T_{\bbeta}(\delta)$. Applying \eqref{pi:uniform-ferromagnetic-relaxation} to the two comparison chains in \eqref{pi:common-envelopes}, we obtain
\begin{equation}\label{pi:uniform-relaxation}
A_s=\{X_s^{\mathbf0,\bbeta^-},X_s^{\mathbf1,\bbeta^+}
\in\mathcal T_{\bbeta}(\delta)\},\quad
\Pp(A_s^c)\le Ce^{-cn},\quad s\ge s_0:=\lceil C_0N\rceil.
\end{equation}
\proofstep{4} We can now convert the relaxation estimate into contraction of the target chains. Recall from \eqref{pi:uniform-relaxation} that $A_s$ is the event on which both comparison states lie in $\mathcal T_{\bbeta}(\delta)$. By \eqref{pi:common-envelopes}, both target states lie between them. To use \eqref{pi:column-contraction}, we must also keep a path joining the targets inside the same set. This follows from the monotonicity of the rooted counts, which implies
\begin{equation}\label{pi:typical-order-interval}
b^-\le b^+,\quad b^-,b^+\in\mathcal T_{\bbeta}(\delta)
\quad\Longrightarrow\quad
[b^-,b^+]:=\{w:b^-\le w\le b^+\}\subseteq\mathcal T_{\bbeta}(\delta).
\end{equation}
Indeed, $r_G(b^-,i)\le r_G(w,i)\le r_G(b^+,i)$ for every $w$ between the two endpoints. On $A_s$, choose a shortest coordinate path $w^0,\ldots,w^d$ between the target states. Each coordinate along the path takes its value from one of the two target states. Thus \eqref{pi:common-envelopes} and \eqref{pi:typical-order-interval} keep the whole path in $\mathcal T_{\bbeta}(\delta)$, even if its endpoints are not ordered. If $s_\ell$ is the coordinate changed at step $\ell$, combining \eqref{eq:force}, \eqref{eq:conditional} and \eqref{pi:column-contraction}, we obtain
\begin{equation}\label{pi:path-contraction}
\sum_{i=1}^N|q_{i,\btheta}(w^d)-q_{i,\btheta}(w^0)|
\le\sum_{\ell=1}^d\sum_{\substack{i=1\\i\ne s_\ell}}^N
|\Delta_{s_\ell}q_{i,\btheta}(w^{\ell-1})|
\le(1-\alpha)d.
\end{equation}
Here $d=d_{\rm H}(w^0,w^d)$ and $\Delta_{s_\ell}q_{s_\ell,\btheta}=0$. At each step of \eqref{pi:path-contraction}, the whole column is evaluated at the same configuration $w^{\ell-1}$, as required by \eqref{pi:column-contraction}.

The path estimate handles the updates on $A_s$. For the time before relaxation, and for the exceptional event $A_s^c$, we need a bound valid on every configuration. Combining \eqref{resp:rowsum}, \eqref{resp:matrixbounds} and \eqref{pi:perturbation-size} with the one-update identity \eqref{pi:coupled-update}, we obtain
\begin{equation}\label{pi:global-coupling-growth}
\E d_{\rm H}(X_1^{y,\btheta},X_1^{z,\btheta})
\le(1+C/N)d_{\rm H}(y,z).
\end{equation}
Now take adjacent initial states and let $d_s=d_{\rm H}(X_s^{y,\btheta},X_s^{z,\btheta})$ and $D_s=\E d_s$. Let $\mathcal F_s$ be the history of the four coupled chains through time $s$. The event $A_s$ is $\mathcal F_s$-measurable. Combining \eqref{pi:coupled-update} and \eqref{pi:path-contraction} on $A_s$, and using \eqref{pi:global-coupling-growth} on its complement, we obtain
\begin{equation}\label{pi:conditional-distance-recursion}
\E[d_{s+1}\mid\mathcal F_s]
\le(1-\alpha/N)d_s+\frac{C+\alpha}{N}d_s\ind_{A_s^c}.
\end{equation}
The fixed-time bound in \eqref{pi:uniform-relaxation} suffices, because \eqref{pi:conditional-distance-recursion} retains the exceptional event at each step. Recall that $s_0=\lceil C_0N\rceil$ and $\Pp(A_s^c)\le Ce^{-cn}$ for $s\ge s_0$. Since $d_s\le N$, the last term in \eqref{pi:conditional-distance-recursion} has expectation at most $Ce^{-cn}$. Before time $s_0$, iterating \eqref{pi:global-coupling-growth} from $D_0=1$ gives $D_{s_0}\le(1+C/N)^{s_0}\le e^{C(C_0+1)}=:K_0$. Taking expectations and iterating the resulting recursion, we obtain
\begin{equation}\label{pi:adjacent-block}
D_{s+1}\le(1-\alpha/N)D_s+Ce^{-cn},\qquad s\ge s_0,\quad
D_{s_0+r}\le K_0e^{-\alpha r/N}+CNe^{-cn},\qquad r\ge0.
\end{equation}
The accumulated error in \eqref{pi:adjacent-block} is at most $CNe^{-cn}$. Since $\alpha=\frac12\inf_{\bbeta\in\B}\rho_{\bbeta}>0$ by \eqref{pi:column-contraction}, choose $A$ with $K_0e^{-\alpha A}<1/4$ and put $\tau_n=s_0+\lceil AN\rceil\le C_*N$. For large $n$, \eqref{pi:adjacent-block} gives $D_{\tau_n}\le1/2$, so for adjacent initial states,
\[
\sup_{d_{\rm H}(y,z)=1}
|P_{\btheta,n}^{\tau_n}F(y)-P_{\btheta,n}^{\tau_n}F(z)|
\le\Lip_{\rm H}(F)\sup_{d_{\rm H}(y,z)=1}D_{\tau_n}
\le\tfrac12\Lip_{\rm H}(F).
\]
This proves the first inequality in \eqref{pi:block-seminorm}. For the second, we need an upper bound at every earlier time, for which \eqref{pi:global-coupling-growth} suffices. Iterating that estimate, we have
\[
\max_{0\le r<\tau_n}\Lip_{\rm H}(P_{\btheta,n}^rF)
\le(1+C/N)^{\tau_n}\Lip_{\rm H}(F)
\le e^{CC_*}\Lip_{\rm H}(F),
\]
which proves the second inequality in \eqref{pi:block-seminorm}. All restrictions on $\eta$ in Steps~1--3 depend only on $\B$ and $\cH$, because the added interactions are bounded by \eqref{pi:perturbation-size} and \eqref{pi:scalar-perturbation}. Thus the same radius applies to every fixed added family $\mathcal J$. Together with the fixed-point bounds from Step~1, this completes the proof.
\end{proof}

\subsection{Moment estimates}\label{sec:local-responses}

The two contraction arguments complete the proof of the Poincar\'e inequality \eqref{eq:poincare} and the covariance bound \eqref{eq:covariance-bound}. We now use these inequalities to estimate the local fields and the edge-count remainder. In this subsection, $\bbeta$ again denotes the actual model parameter, including the added coefficients in the signed case. As a first application, take $F=E$ in \eqref{eq:poincare}. By \eqref{eq:edge-scale} and \eqref{eq:compact-scales}, we obtain
\begin{equation}\label{loc:edgevariance}
\Var(E_{\bbeta})\le\kappa N,\qquad \E W_{\bbeta}^2\le C.
\end{equation}
To evaluate a local field through its rooted-count representation, we will need expectations of products of finitely many edge indicators. The covariance bound allows us to separate one factor at a time. This yields the following approximation in terms of the true edge mean.

\begin{lemma}\label{loc:product}
For every fixed integer $m\ge1$ and distinct host edges $s_1,\ldots,s_m$, we have
\begin{equation}\label{loc:producteq}
\left|\E\prod_{a=1}^mY_{s_a}-\widetilde p^{\,m}\right|
\le C_m/n.
\end{equation}
The constant is uniform over the fixed parameter family.
\end{lemma}
\begin{proof}[Proof of Lemma~\ref{loc:product}]
For $A\subset\cI\setminus\{s\}$, put $G_A(y)=\prod_{r\in A}y_r$. To use the small coordinate increments of $q_s$, we first condition on $Y_{-s}$. Since $G_A$ does not depend on $y_s$, \eqref{eq:conditional} implies $\Cov(Y_s,G_A(Y))=\Cov(q_s(Y),G_A(Y))$. Here $\sum_{r=1}^N \|\Delta_rG_A\|_\infty\le|A|$, while $\max_r\|\Delta_rq_s\|_\infty\le C/n$ by \eqref{resp:matrixbounds}. Using \eqref{eq:covariance-bound} with $F=G_A$ and $G=q_s$, we obtain
\begin{equation}\label{loc:productcov}
\left|\Cov\left(Y_s,\prod_{r\in A}Y_r\right)\right|
\le C|A|/n.
\end{equation}
Expanding the covariance of the first factor and iterating from $s_1$ to $s_m$ gives
\begin{equation}\label{loc:producttelescope}
\E\prod_{a=1}^mY_{s_a}-\widetilde p^{\,m}
=\sum_{a=1}^{m-1}\widetilde p^{\,a-1}
\Cov\left(Y_{s_a},\prod_{b=a+1}^mY_{s_b}\right).
\end{equation}
Applying \eqref{loc:productcov} term by term in \eqref{loc:producttelescope}, with $0\le\widetilde p\le1$, proves \eqref{loc:producteq}. For $m=1$ the difference is zero. The finitely many sample sizes excluded by the counting bounds are covered by increasing $C_m$, since both product expectations lie in $[0,1]$.
\end{proof}

We next use \eqref{loc:producteq} to evaluate the mean of a local field. Together with a variance estimate, this will show that the true edge mean $\widetilde p$ nearly satisfies the fixed-point equation and is close to $p$. By \eqref{eq:model}, \eqref{eq:conditional} and \eqref{loc:rootcount}, the local field has the representation
\begin{equation}\label{loc:rootfield}
L_s(y)=\sum_{j=1}^k\beta_jn^{2-v_j}|\Hom(H_j,G_y,s)|.
\end{equation}
The uniform bounds on local fields and conditional probabilities are \eqref{pi:field-bound}--\eqref{pi:conditional-bounds}. Taking expectations in \eqref{eq:conditional} also gives
\begin{equation}\label{loc:bernoulli}
\E q_s(Y)=\widetilde p,\qquad
c\le\widetilde p\le1-c.
\end{equation}
The following estimates relate $\widetilde p$ and the local fields to their mean-field values.

\begin{lemma}\label{loc:field}
Let $\ell_0=2\Phi_{\bbeta}(p)$. Then, uniformly over $\bbeta\in\B$ and host edges $s$, we have
\begin{align}
|\widetilde p-p|&\le C/n,\label{loc:meanprecision}\\
|\E L_s(Y)-\ell_0|+\E\{L_s(Y)-\ell_0\}^2
&\le C/n.\label{loc:fieldprecision}
\end{align}
\end{lemma}
In the ferromagnetic case, \eqref{loc:meanprecision} was obtained by Fang, Liu, Su and Wang \cite[Proposition 1.1]{FLSW}. The following proof also applies to the signed families.
\begin{proof}[Proof of Lemma~\ref{loc:field}]
We prove the claim in two stages. First, we estimate the field variance by \eqref{eq:poincare} and its mean by \eqref{loc:producteq}. A Taylor expansion of the conditional probability then shows that $\widetilde p$ nearly solves the mean-field equation. We use uniqueness to locate it near $p$ and stability to obtain the error of order $n^{-1}$. Throughout the calculation, take $n$ at least four and at least the number of vertices in every model graph.

Fix a host edge $s$. There are $2(n-2)$ adjacent and $\binom{n-2}{2}$ disjoint choices of $r$. Substituting \eqref{loc:doublecount} into \eqref{eq:poincare} yields
\begin{equation}\label{loc:fieldvariance}
\begin{aligned}
\Var(L_s(Y))
\le\kappa\sum_{r=1}^N \E|\Delta_rL_s(Y)|^2
\le C\left\{2(n-2)n^{-2}+\binom{n-2}{2}n^{-4}\right\}
\le C/n.
\end{aligned}
\end{equation}

To estimate the mean, take expectations in the rooted-count representation \eqref{loc:rootcount}. Each remaining product in \eqref{loc:rootcount} contains $e_j-1$ distinct edges. By \eqref{loc:producteq}, its expectation is $\widetilde p^{\,e_j-1}+O(n^{-1})$, uniformly over embeddings. There are $2e_j\prod_{r=2}^{v_j-1}(n-r)$ candidates, so
\begin{equation}\label{loc:rootexpectation}
\E|\Hom(H_j,G_Y,s)|
=2e_j\left\{\prod_{r=2}^{v_j-1}(n-r)\right\}
\widetilde p^{\,e_j-1}+O(n^{v_j-3}).
\end{equation}
For $j=1$, the remaining product is empty and the error is zero. Using \eqref{loc:rootexpectation} in \eqref{loc:rootfield}, together with \eqref{loc:fallingcomparison} and \eqref{eq:meanfield}, gives
\begin{equation}\label{loc:fieldmean}
\E L_s(Y)=2\Phi_{\bbeta}(\widetilde p)+O(n^{-1}).
\end{equation}
The summed error is uniform because the graph family is fixed and the coefficients are bounded.

The mean and variance estimates now allow us to compare $\widetilde p$ with the fixed point. By \eqref{eq:conditional} and \eqref{loc:bernoulli}, we have $\widetilde p=\E\Psi(L_s(Y))$. Expanding at $\E L_s(Y)$, the linear term has mean zero and the remainder is at most $\|\Psi''\|_\infty\Var(L_s(Y))/2$. Combining this expansion with $\|\Psi'\|_\infty\le1/4$, \eqref{eq:meanfield}, \eqref{loc:fieldvariance} and \eqref{loc:fieldmean}, we obtain
\begin{equation}\label{loc:fixedpointresidual}
\begin{aligned}
|\widetilde p-\varphi_{\bbeta}(\widetilde p)|
&\le|\E\Psi(L_s(Y))-\Psi(\E L_s(Y))|
+|\Psi(\E L_s(Y))-\Psi(2\Phi_{\bbeta}(\widetilde p))|\\
&\le\tfrac12\|\Psi''\|_\infty\Var(L_s(Y))
+\tfrac14|\E L_s(Y)-2\Phi_{\bbeta}(\widetilde p)|
\le C/n.
\end{aligned}
\end{equation}
Let $\rho_*:=\inf_{\bbeta\in\B}\rho_{\bbeta}>0$. Continuity and compactness allow a common $\delta>0$ such that
\begin{equation}\label{loc:stabilitylocal}
1-\varphi_{\bbeta}'(u)\ge\rho_*/2
\quad\text{for }\bbeta\in\B,\ u\in[0,1],\ |u-p_{\bbeta}|\le\delta.
\end{equation}
On the other hand, $u-\varphi_{\bbeta}(u)$ has only the root $p_{\bbeta}$ in $[0,1]$. Its absolute value therefore has a positive lower bound on the compact set
\[
\{(\bbeta,u):\bbeta\in\B,\ u\in[0,1],\ |u-p_{\bbeta}|\ge\delta\}
\]
whenever that set is nonempty. Thus \eqref{loc:fixedpointresidual} forces $|\widetilde p-p|<\delta$ for all $n$ beyond a common threshold. Applying the mean value theorem between $p$ and $\widetilde p$ and using \eqref{loc:stabilitylocal}, we obtain
\begin{equation}\label{loc:inverse-stability}
|\widetilde p-\varphi_{\bbeta}(\widetilde p)|
\ge(\rho_*/2)|\widetilde p-p|.
\end{equation}
Combining \eqref{loc:inverse-stability} with \eqref{loc:fixedpointresidual} proves \eqref{loc:meanprecision}.

Finally, the boundedness of $\Phi_{\bbeta}'$ in \eqref{eq:meanfield}, together with \eqref{loc:meanprecision}, \eqref{loc:fieldmean} and \eqref{loc:fieldvariance}, gives $|\E L_s(Y)-\ell_0|\le C/n$ and
\begin{equation*}
\E\{L_s(Y)-\ell_0\}^2
=\Var(L_s(Y))+\{\E L_s(Y)-\ell_0\}^2
\le C/n+C/n^2,
\end{equation*}
which proves \eqref{loc:fieldprecision}. For the finitely many remaining sample sizes, $p,\widetilde p\in[0,1]$ and \eqref{pi:field-bound} permit a common increase in the constant.
\end{proof}

We have now controlled the local fields and their means. To estimate the edge-count remainder, we need the more precise comparison of the total change in conditional probabilities with $\varphi'(p)$. Expectations below are under $Y=Y_{\bbeta}$, and we first take $n\ge\max\{4,v_1,\ldots,v_k\}$. The uniform coordinate bounds will cover the remaining sample sizes.

For the edge count, write $R=R_{K_2}$ in \eqref{stein:rh} and denote the centered count by $S$, so that
\begin{equation}\label{resp:Rdefinition}
S(y)=E(y)-\mu,\qquad
R(y)=\sum_{j=1}^N\{y_j-q_j(y)\}-\rho S(y).
\end{equation}
Since $\E(Y_j-q_j(Y))=0$ and $\E S(Y)=0$, the remainder is centered. Also, $\Delta_i\sum_{j=1}^N y_j=\Delta_iS=1$ and $1-\rho=\varphi'(p)$. Therefore,
\begin{equation}\label{resp:Rdifference}
\E R(Y)=0,\qquad
\Delta_iR(y)=\varphi'(p)-\sum_{j=1}^N \Delta_iq_j(y).
\end{equation}
Thus the mean estimate for $\Delta_iR$ reduces to approximating $\sum_{j=1}^N\E\Delta_iq_j(Y)$ by $\varphi'(p)$.

For the variance of $\Delta_iR$, we will again apply the Poincar\'e inequality. By differentiating \eqref{resp:Rdifference} once more, its gradient is a sum of second differences of $q_j$. We therefore estimate these sums before treating the mean of $\Delta_iR$.

\begin{lemma}\label{resp:influence}
Uniformly on the fixed compact parameter set, we have
\begin{equation}\label{resp:secondtotal}
\max_i\sum_{\substack{r=1\\r\ne i}}^N\sum_{j=1}^N \|\Delta_{ir}q_j\|_\infty\le C,
\qquad
\max_{i\ne r}\sum_{j=1}^N \|\Delta_{ir}q_j\|_\infty\le C/n.
\end{equation}
\end{lemma}
\begin{proof}[Proof of Lemma~\ref{resp:influence}]
To prove \eqref{resp:secondtotal}, we differentiate $q_j=\Psi(L_j)$ in two coordinates. This produces a third difference of the energy and a product of two differences of the local field. Recall that $A_{ir}=\Delta_{ir}T_{|\bbeta|}(\mathbf1)$ in \eqref{resp:Adefinition} bounds the absolute second difference of $T$. To justify the differentiation, extend $T$ to $[0,1]^N$ by its multiaffine expression and set $q_j=\Psi(\partial_jT)$ on the cube. At its vertices, $\partial_jT=\Delta_jT=L_j$, so this agrees with \eqref{eq:conditional}.

Fix $i\ne r$. If $j=i$ or $j=r$, then $q_j$ is independent of coordinate $j$, and $\Delta_{ir}q_j=0$. Otherwise, applying the fundamental theorem of calculus in coordinates $i$ and $r$ yields
\begin{equation}\label{resp:secondintegral}
\Delta_{ir}q_j(y)=\int_0^1\!\int_0^1
\bigl\{\Psi'(\partial_jT)\partial_{irj}T
+\Psi''(\partial_jT)(\partial_{ij}T)(\partial_{rj}T)\bigr\}
(y^{(i,u;r,w)})\,du\,dw.
\end{equation}
Here $y^{(i,u;r,w)}$ is the cube point obtained by setting the two coordinates to $u,w$. Each mixed partial derivative is the multilinear interpolation of its vertex values. Its absolute value is therefore bounded by the supremum norm of the corresponding vertex difference, also when the coefficients have different signs. Applying \eqref{resp:Adefinition} and $\|\Psi'\|_\infty\le1/4$ in \eqref{resp:secondintegral}, we obtain
\begin{equation}\label{resp:chainrule}
\|\Delta_{ir}q_j\|_\infty
\le\tfrac14\|\Delta_{irj}T\|_\infty
+\|\Psi''\|_\infty A_{ji}A_{jr}.
\end{equation}
For the first inequality in \eqref{resp:secondtotal}, sum the right-hand side of \eqref{resp:chainrule} first over $j$ and then over $r\ne i$. Combining \eqref{resp:thirdsum} and \eqref{resp:rowsum}, we bound its third-difference term by
\[
\sum_{\substack{r=1\\r\ne i}}^N\sum_{j=1}^N \|\Delta_{irj}T\|_\infty
\le C\sum_{r=1}^N A_{ir}\le C.
\]
For the product $A_{ji}A_{jr}$ on the right of \eqref{resp:chainrule}, use nonnegativity, symmetry and \eqref{resp:rowsum} to obtain
\[
\sum_{\substack{r=1\\r\ne i}}^N\sum_{j=1}^N A_{ji}A_{jr}
\le\sum_{j=1}^N A_{ji}\sum_{r=1}^N A_{jr}
\le C\sum_{j=1}^N A_{ji}\le C.
\]
Substituting the two sums into \eqref{resp:chainrule} proves the first inequality in \eqref{resp:secondtotal}. For its second inequality, keep $i\ne r$ fixed. By \eqref{resp:thirdsum} and \eqref{resp:matrixbounds}, the sum over $j$ of the third-difference term in \eqref{resp:chainrule} is at most $CA_{ir}\le C/n$. For the product $A_{ji}A_{jr}$ in \eqref{resp:chainrule}, we have
\[
\sum_{j=1}^N A_{ji}A_{jr}
\le\left(\max_jA_{jr}\right)\sum_{j=1}^N A_{ji}
\le C/n,
\]
by \eqref{resp:rowsum}--\eqref{resp:matrixbounds}. Substitution into \eqref{resp:chainrule} proves the second inequality in \eqref{resp:secondtotal}.
\end{proof}

We can now estimate $\Delta_iR$ in \eqref{resp:Rdifference}. The second-difference bounds just proved, together with the Poincar\'e inequality, control its variance. The mean needs a finer argument. We retain a mixed local-field term inside the expectation and use \eqref{eq:covariance-bound} to estimate it. This will identify $\varphi'(p)$ as the leading term of $\sum_{j=1}^N\E\Delta_iq_j(Y)$ with error $O(n^{-1})$.

\begin{lemma}\label{resp:mean}
For every edge $i$, we have, uniformly over the fixed parameter set,
\begin{equation}\label{resp:meanbounds}
|\E\Delta_iR(Y)|\le C/n,\qquad
\Var(\Delta_iR(Y))\le C/n,\qquad
\|\Delta_iR\|_\infty\le C.
\end{equation}
\end{lemma}
\begin{proof}[Proof of Lemma~\ref{resp:mean}]
We first prove the supremum and variance bounds in \eqref{resp:meanbounds}, using Lemma~\ref{resp:influence} and \eqref{eq:poincare}. By \eqref{resp:Rdifference}, \eqref{resp:matrixbounds} and the column bound in \eqref{resp:rowsum}, we have
\[
\|\Delta_iR\|_\infty
\le |\varphi'(p)|+\tfrac14\sum_{j=1}^N A_{ji}\le C.
\]
If $r=i$, then $\Delta_{ri}R=0$. Otherwise, another difference in \eqref{resp:Rdifference} eliminates the constant term and gives
\[
\Delta_{ri}R=-\sum_{j=1}^N \Delta_{ir}q_j,
\qquad
\|\Delta_{ri}R\|_\infty
\le\sum_{j=1}^N \|\Delta_{ir}q_j\|_\infty.
\]
Applying \eqref{eq:poincare} to $\Delta_iR$ and bounding a sum of squares by the maximum times the sum, we have
\begin{equation*}
\Var(\Delta_iR(Y))\le C\sum_{r=1}^N \E|\Delta_{ri}R(Y)|^2\le C\left(\max_{r\ne i}\sum_{j=1}^N \|\Delta_{ir}q_j\|_\infty\right)\sum_{\substack{r=1\\r\ne i}}^N\sum_{j=1}^N \|\Delta_{ir}q_j\|_\infty\le C/n,
\end{equation*}
where the last step uses both inequalities in \eqref{resp:secondtotal}.

It remains to prove the mean estimate in \eqref{resp:meanbounds}. By \eqref{resp:Rdifference}, we need to approximate $\sum_{j=1}^N\E\Delta_iq_j(Y)$ by $\varphi'(p)$. Fix $j\ne i$ and let
\begin{equation}\label{resp:azeta}
a(y)=\Delta_{ij}T(y),\qquad
\zeta(y)=L_j(y^{(i,0)})-\ell_0,
\qquad \ell_0=2\Phi_{\bbeta}(p).
\end{equation}
By \eqref{eq:conditional}, the difference $\Delta_iq_j=\Psi(\ell_0+\zeta+a)-\Psi(\ell_0+\zeta)$ has an expansion involving $a\zeta$, whose expectation we estimate first. Both functions in \eqref{resp:azeta} are independent of $y_i,y_j$. By multiaffinity, we have
\begin{equation}\label{resp:zetadecomposition}
\begin{aligned}
L_j(y)=L_j(y^{(i,0)})+y_i a(y),\quad
\zeta(Y)=L_j(Y)-\ell_0-Y_i a(Y).
\end{aligned}
\end{equation}
By \eqref{resp:Adefinition}--\eqref{resp:matrixbounds}, $\|a\|_\infty\le A_{ij}\le C/n$. Substituting this bound and \eqref{loc:fieldprecision} into \eqref{resp:zetadecomposition} yields $|\E\zeta(Y)|+\E\zeta(Y)^2\le C/n$. Moreover, $\Delta_i\zeta=0$, while for $r\ne i$ coordinate replacement commutes with $\Delta_r$, so
\[
\Delta_r\zeta(y)=(\Delta_rL_j)(y^{(i,0)}),\qquad
\|\Delta_r\zeta\|_\infty\le A_{jr}.
\]
Consequently, the cumulant generating function satisfies
\begin{equation}\label{resp:zetabounds}
|\E\zeta(Y)|\le C/n,\qquad
\E\zeta(Y)^2\le C/n,
\qquad \max_r\|\Delta_r\zeta\|_\infty\le C/n.
\end{equation}

To estimate $\E[a(Y)\zeta(Y)]$, first recall from \eqref{resp:azeta} that $a=\Delta_{ij}T$. By \eqref{resp:thirdsum}, we have
\[
\sum_{r=1}^N \|\Delta_ra\|_\infty
=\sum_{r=1}^N \|\Delta_{ijr}T\|_\infty\le C A_{ij}.
\]
Using \eqref{eq:covariance-bound} and \eqref{resp:zetabounds}, we obtain
\begin{equation}\label{resp:mixedcov}
|\Cov(a(Y),\zeta(Y))|
\le\kappa\left(\sum_{r=1}^N \|\Delta_ra\|_\infty\right)
\max_r\|\Delta_r\zeta\|_\infty
\le C A_{ij}/n.
\end{equation}
Combining \eqref{resp:mixedcov} and \eqref{resp:zetabounds}, with $|a|\le A_{ij}$ from \eqref{resp:Adefinition}, yields
\begin{equation}\label{resp:mixedbound}
|\E[a(Y)\zeta(Y)]|
\le|\Cov(a(Y),\zeta(Y))|+|\E a(Y)|\,|\E\zeta(Y)|
\le C A_{ij}/n.
\end{equation}
This is the point at which the covariance estimate improves the precision. If we instead bounded the product by $\|a\|_\infty\E|\zeta|$, then \eqref{resp:zetabounds} would give only $O(A_{ij}/\sqrt n)$. By keeping the expectation in \eqref{resp:mixedbound}, we obtain the $O(A_{ij}/n)$ bound needed when summing over $j$.

We now apply \eqref{resp:mixedbound} to the expansion of $\Delta_iq_j$. By \eqref{eq:conditional} and \eqref{resp:azeta}--\eqref{resp:zetadecomposition}, the fundamental theorem of calculus gives
\begin{equation}\label{resp:firstintegral}
\Delta_iq_j(y)=a(y)\int_0^1
\Psi'(\ell_0+\zeta(y)+t a(y))\,dt.
\end{equation}
This identity is valid for either sign of $a(y)$.

For every real $z$, Taylor's formula gives
\begin{equation}\label{resp:logistictaylor}
\left|\Psi'(\ell_0+z)-\Psi'(\ell_0)-\Psi''(\ell_0)z\right|
\le\tfrac12\|\Psi^{(3)}\|_\infty z^2.
\end{equation}
The fixed-point equation implies $\Psi(\ell_0)=p$ and $\Psi'(\ell_0)=p(1-p)$. Substituting \eqref{resp:logistictaylor} into \eqref{resp:firstintegral}, the linear term in $ta$ integrates to $\frac12\Psi''(\ell_0)a^2$. For the Taylor remainder from \eqref{resp:logistictaylor} in \eqref{resp:firstintegral}, we use
\[
|a|(\zeta+t a)^2
\le 2A_{ij}\zeta^{2}+2(A_{ij})^3,
\qquad 0\le t\le1.
\]
Since $A_{ij}$ is uniformly bounded by \eqref{resp:matrixbounds}, its cubic power is absorbed into the quadratic term. We obtain the pointwise estimate
\begin{equation}\label{resp:taylormixed}
\left|\Delta_iq_j(y)-p(1-p)a(y)-\Psi''(\ell_0)a(y)\zeta(y)\right|
\le C\{A_{ij}\zeta(y)^2+A_{ij}^2\}.
\end{equation}
Taking expectations in \eqref{resp:taylormixed} and applying \eqref{resp:zetabounds} and \eqref{resp:mixedbound}, we have
\begin{equation}\label{resp:pairmean}
\E\Delta_iq_j(Y)
=p(1-p)\E\Delta_{ij}T(Y)
+O\{A_{ij}/n+(A_{ij})^2\}.
\end{equation}
The constant is uniform in $i,j,\bbeta$. The terms vanish when $j=i$, so this index may be included. Combining \eqref{resp:rowsum}--\eqref{resp:matrixbounds}, we bound the errors in \eqref{resp:pairmean} by
\begin{equation}\label{resp:pairerrorsum}
\sum_{j=1}^N \{A_{ij}/n+(A_{ij})^2\}
\le\frac1n\sum_{j=1}^N A_{ij}
+\left(\max_jA_{ij}\right)\sum_{j=1}^N A_{ij}
\le C/n.
\end{equation}

The error terms in \eqref{resp:pairmean} now have the required total size. To finish the mean estimate, we evaluate its leading sum $\sum_{j=1}^N\E\Delta_{ij}T(Y)$. For an energy graph $H_\ell$, choose two distinct source edges $(a,b)$ and an injection with $\iota(a)=i$. Then the second host edge is determined by $j=\iota(b)\ne i$, so
\begin{equation}\label{resp:markedtworoot}
\begin{split}
\sum_{j=1}^N \Delta_{ij}|\Hom(H_\ell,G_y)|
=\sum_{\substack{a,b\in\mathcal E(H_\ell)\\a\ne b}}
\ \sum_{\substack{\iota:\mathcal V(H_\ell)\hookrightarrow[n]\\\iota(a)=i}}
\prod_{c\in\mathcal E(H_\ell)\setminus\{a,b\}}y_{\iota(c)}.
\end{split}
\end{equation}
In \eqref{resp:markedtworoot}, there are $2e_\ell(e_\ell-1)\prod_{r=2}^{v_\ell-1}(n-r)$ marked embeddings, each leaving a product of $e_\ell-2$ distinct edges. For $e_\ell=2$, this is the empty product, equal to one. For $e_\ell\ge3$, \eqref{loc:producteq} and \eqref{loc:meanprecision} give
\[
\left|\E\prod_{c\in\mathcal E(H_\ell)\setminus\{a,b\}}Y_{\iota(c)}-p^{e_\ell-2}\right|
\le C/n+|\widetilde p^{\,e_\ell-2}-p^{e_\ell-2}|
\le C/n,
\]
where $|x^m-y^m|\le m|x-y|$ for $x,y\in[0,1]$ was used in the last step. Summing over $O(n^{v_\ell-2})$ candidates and multiplying by $\beta_\ell n^{2-v_\ell}$ produces error $O(|\beta_\ell|/n)$. Summing over the fixed family and applying \eqref{loc:fallingcomparison}, we obtain
\begin{equation}\label{resp:energymean}
\sum_{j=1}^N \E\Delta_{ij}T(Y)=2\sum_{\ell=2}^k\beta_\ell e_\ell(e_\ell-1)\frac{\prod_{r=2}^{v_\ell-1}(n-r)}{n^{v_\ell-2}}p^{e_\ell-2}+O(n^{-1})=2\Phi_{\bbeta}'(p)+O(n^{-1}).
\end{equation}
Combining \eqref{resp:energymean}, \eqref{resp:pairmean} and \eqref{resp:pairerrorsum}, and then using \eqref{eq:meanfield}, we obtain
\begin{equation}\label{resp:totalmean}
\E\sum_{j=1}^N \Delta_iq_j(Y)
=2p(1-p)\Phi_{\bbeta}'(p)+O(n^{-1})
=\varphi'(p)+O(n^{-1}).
\end{equation}
Substituting \eqref{resp:totalmean} into \eqref{resp:Rdifference} proves the remaining mean bound in \eqref{resp:meanbounds}. For the finitely many smaller values of $n$, \eqref{resp:Rdifference}, \eqref{resp:rowsum} and \eqref{resp:matrixbounds} give uniform boundedness, so the same $C/n$ form follows by increasing $C$. This completes the proof.
\end{proof}

To apply the two Stein equations, we need estimates for the edge-count remainder $R$ and the corresponding weight $K=K_{K_2}$. By \eqref{stein:kh}, this weight is
\begin{equation}\label{resp:Kdefinition}
K(y)=\frac{1}{\rho\sigma^2}\sum_{i=1}^Nq_i(y)(1-q_i(y)).
\end{equation}
The denominator equals $Np(1-p)$ by \eqref{eq:edge-scale}.

We can now pass from $\Delta_iR$ to $R$ itself by the Poincar\'e inequality. The same inequality controls the fluctuations of $K$. For the last bound below, we estimate the mean and variance of the entire weighted sum. These edge-count estimates will be transferred to $R_H$ and $K_H$ in Lemma~\ref{stein:remainder}.

\begin{lemma}\label{resp:main}
Uniformly on the fixed compact parameter set, we have
\begin{align}
\E R(Y)^2\le Cn,\qquad
\sum_{i=1}^N\E|\Delta_iR(Y)|&\le Cn^{3/2},
\label{resp:Rbounds}\\
0\le K(y)\le C,\qquad \|K(Y)-1\|_2&\le C/n,
\label{resp:Kbounds}\\
\E\left|\sum_{i=1}^Nq_i(Y)(1-q_i(Y))\Delta_iR(Y)\right|
&\le Cn.\label{resp:weightedbound}
\end{align}
\end{lemma}
\begin{proof}[Proof of Lemma~\ref{resp:main}]
We first prove \eqref{resp:Rbounds} from the estimates for $\Delta_iR$. We then treat $K-1$ and use the same variance method for the whole sum in \eqref{resp:weightedbound}. By \eqref{resp:meanbounds}, we have
\begin{equation}\label{resp:Rgradientsecondmoment}
\E|\Delta_iR(Y)|^2
=\Var(\Delta_iR(Y))
+|\E\Delta_iR(Y)|^2
\le C/n+C/n^2.
\end{equation}
Using the centering in \eqref{resp:Rdifference} and \eqref{resp:Rgradientsecondmoment} in \eqref{eq:poincare}, we obtain
\[
\E R(Y)^2
\le\kappa\sum_{i=1}^N \E|\Delta_iR(Y)|^2
\le CN(n^{-1}+n^{-2})\le Cn.
\]
By Cauchy--Schwarz and \eqref{resp:Rgradientsecondmoment}, $\E|\Delta_iR(Y)|\le Cn^{-1/2}$. Summing over $N=\binom n2$ edges gives $\sum_{i=1}^N \E|\Delta_iR(Y)|\le CNn^{-1/2}\le Cn^{3/2}$, proving \eqref{resp:Rbounds}.

For \eqref{resp:Kbounds}, we first estimate the mean of $K$. Applying \eqref{eq:poincare} to $q_i$ and using \eqref{resp:matrixbounds} gives
\begin{equation}\label{resp:qvariance}
\Var(q_i(Y))
\le\kappa\sum_{r=1}^N \E|\Delta_rq_i(Y)|^2
\le\frac{\kappa}{16}\sum_{r=1}^N (A_{ir})^2\le C/n.
\end{equation}
Since $\E q_i(Y)=\widetilde p$ by \eqref{loc:bernoulli}, we can use the identity $\E[q_i(1-q_i)]=\widetilde p(1-\widetilde p)-\Var(q_i)$ to obtain
\begin{equation}\label{resp:conditionalmean}
\begin{gathered}
\sum_{i=1}^N \E[q_i(Y)(1-q_i(Y))]
=N\widetilde p(1-\widetilde p)-\sum_{i=1}^N \Var(q_i(Y)),\\
\left|\sum_{i=1}^N \E[q_i(Y)(1-q_i(Y))]-Np(1-p)\right|\le Cn.
\end{gathered}
\end{equation}
For the second line of \eqref{resp:conditionalmean}, use $|\widetilde p(1-\widetilde p)-p(1-p)|\le|\widetilde p-p|$ and combine \eqref{loc:meanprecision} with \eqref{resp:qvariance}. Dividing \eqref{resp:conditionalmean} by $\rho\sigma^2=Np(1-p)$, and using \eqref{resp:Kdefinition} and \eqref{eq:compact}, we obtain
\begin{equation}\label{resp:Kmean}
|\E K(Y)-1|\le C/n.
\end{equation}

To complete \eqref{resp:Kbounds}, we estimate the variance of $K$ through its single-edge differences. The map $u\mapsto u(1-u)$ is $1$-Lipschitz on $[0,1]$. By \eqref{resp:matrixbounds}, \eqref{resp:rowsum} and \eqref{resp:Kdefinition},
\begin{equation}\label{resp:Kgradient}
\|\Delta_rK\|_\infty
\le\frac{1}{\rho\sigma^2}
\sum_{i=1}^N \|\Delta_r\{q_i(1-q_i)\}\|_\infty
\le\frac{1}{4\rho\sigma^2}\sum_{i=1}^N A_{ir}
\le C/N.
\end{equation}
Substituting \eqref{resp:Kgradient} into \eqref{eq:poincare} gives $\Var(K(Y))\le\kappa N(C/N)^2\le C/N$. Since $N\asymp n^2$, \eqref{resp:Kmean} then yields
\[
\|K(Y)-1\|_2^2
=\Var(K(Y))+|\E K(Y)-1|^2\le C/n^2.
\]
Furthermore, $0\le q_i(1-q_i)\le1/4$ implies $0\le K(y)\le N/(4\rho\sigma^2)=[4p(1-p)]^{-1}\le C$, by \eqref{eq:edge-scale} and \eqref{eq:compact}. This proves \eqref{resp:Kbounds}.

It remains to control the entire weighted sum in \eqref{resp:weightedbound}. Summing absolute values using \eqref{resp:Rbounds} would give only $O(n^{3/2})$, so it would not suffice. Instead, we estimate the mean of the sum and its fluctuation around that mean. Let
\begin{equation}\label{resp:Udefinition}
U(y)=\sum_{i=1}^N q_i(y)(1-q_i(y))\Delta_iR(y).
\end{equation}
The Lipschitz bound for $u(1-u)$, together with \eqref{eq:poincare} and \eqref{resp:matrixbounds}, gives
\begin{equation}\label{resp:weightvariance}
\Var\{q_i(Y)(1-q_i(Y))\}
\le\kappa\sum_{r=1}^N \|\Delta_r\{q_i(1-q_i)\}\|_\infty^2
\le\frac{\kappa}{16}\sum_{r=1}^N (A_{ir})^2\le C/n.
\end{equation}
Splitting each expected product in \eqref{resp:Udefinition} into a product of means and a covariance, and applying Cauchy--Schwarz with \eqref{resp:weightvariance} and \eqref{resp:meanbounds}, we obtain
\begin{equation}\label{resp:weightedmean}
|\E U(Y)|
\le\sum_{i=1}^N \left\{\tfrac14|\E\Delta_iR(Y)|
+\sqrt{\Var(q_i(Y)(1-q_i(Y)))
\Var(\Delta_iR(Y))}\right\}
\le CN/n\le Cn.
\end{equation}

To bound the fluctuations of $U$, use the exact product-difference identity
\begin{equation}\label{resp:productrule}
\Delta_r(FG)(y)
=(\Delta_rF)(y)G(y^{(r,1)})
+F(y^{(r,0)})(\Delta_rG)(y).
\end{equation}
Apply \eqref{resp:productrule} to each summand in \eqref{resp:Udefinition}, with $F=q_i(1-q_i)$ and $G=\Delta_iR$. Using $\|\Delta_iR\|_\infty\le C$ from \eqref{resp:meanbounds} and $q_i(1-q_i)\le1/4$, we obtain
\begin{equation}\label{resp:weighteddifference}
\|\Delta_rU\|_\infty
\le C\sum_{i=1}^N \|\Delta_rq_i\|_\infty
+\tfrac14\sum_{i=1}^N \|\Delta_{ri}R\|_\infty
\le C.
\end{equation}
For the first sum on the right of \eqref{resp:weighteddifference}, combine \eqref{resp:matrixbounds} and \eqref{resp:rowsum} to obtain $\frac C4\sum_{i=1}^N A_{ir}\le C$. In its second sum, the $i=r$ difference vanishes. By \eqref{resp:Rdifference} and the first inequality in \eqref{resp:secondtotal}, with root coordinate $r$, we have
\[
\sum_{\substack{i=1\\i\ne r}}^N\|\Delta_{ri}R\|_\infty
\le\sum_{\substack{i=1\\i\ne r}}^N\sum_{j=1}^N \|\Delta_{ri}q_j\|_\infty\le C.
\]
Thus \eqref{eq:poincare} and \eqref{resp:weighteddifference} imply
\begin{equation}\label{resp:Uvariance}
\Var(U(Y))
\le\kappa\sum_{r=1}^N \E|\Delta_rU(Y)|^2\le CN.
\end{equation}
Finally, combining \eqref{resp:weightedmean}, \eqref{resp:Uvariance} and $\sqrt N\le Cn$, we have
\[
\E|U(Y)|
\le|\E U(Y)|+\E|U(Y)-\E U(Y)|
\le|\E U(Y)|+\sqrt{\Var(U(Y))}
\le Cn.
\]
By the definition \eqref{resp:Udefinition}, this is \eqref{resp:weightedbound}, which completes the proof.
\end{proof}

\section{Proofs of the lemmas in Sections~\ref{sec:stein} and \ref{sec:md}}\label{sec:auxiliary}

We now verify the auxiliary estimates used in the proofs of the main theorems. Appendix~\ref{sec:normal-lemmas} treats Lemmas~\ref{stein:increments}--\ref{stein:concentration}, using the variance, covariance and moment bounds established in Appendix~\ref{sec:ferro}. Appendix~\ref{sec:tilting-lemmas} then proves Lemmas~\ref{md:path-input}--\ref{md:weighted} using the exact change-of-measure identities.

\subsection{Proofs of Lemmas~\ref{stein:increments}--\ref{stein:concentration}}\label{sec:normal-lemmas}

We use $Y=Y_{\bbeta}$, $p=p_{\bbeta}$, $\rho=\rho_{\bbeta}$, $\sigma=\sigma_{\bbeta}$, $v=v_H$ and $e=e_H$, and recall $f_H,R_H,K_H$ from \eqref{stein:fh}, \eqref{stein:rh} and \eqref{stein:kh}, respectively. In particular, $V=f_H(Y)$ and $h=C_1/n$. All constants are uniform over $\bbeta\in\B$. As in the auxiliary lemma statements, we first take $H$ without isolated vertices and large $n$. Sections~\ref{sec:residual-proof}--\ref{sec:normal-proof} explain how the conclusions extend to the other graphs, sample sizes and normalizations.

\begin{proof}[Proof of Lemma~\ref{stein:increments}]
The increment estimates follow by counting embeddings with one or two marked edges. We use \eqref{loc:producteq} to compute their means and the Poincar\'e inequality \eqref{eq:poincare} to bound their variances. Once these estimates are established, subtracting the mean increment and applying the same inequality will yield \eqref{stein:linear}. We begin with the exact rooted-count expression from \eqref{eq:count} and \eqref{loc:rootcount},
\begin{equation}
\Delta_if_H(y)=
\frac{|\Hom(H,G_y,i)|}{2ep^{e-1}n^{v-2}\sigma}.
\label{stein:root-expression}
\end{equation}
By \eqref{loc:rootcountbound} and \eqref{eq:compact}--\eqref{eq:compact-scales}, we obtain \eqref{stein:increment-size}. In \eqref{loc:double-root-count}, adjacent roots fix three vertices and disjoint roots fix four. Their counts are therefore $O(n^{v-3})$ and $O(n^{v-4})$, and division by the denominator in \eqref{stein:root-expression} proves \eqref{stein:second-increment}. Summing over the $2(n-2)$ adjacent and $\binom{n-2}{2}$ disjoint choices gives
\[
\begin{split}
\sum_{r=1}^N \|\Delta_{ir}f_H\|_\infty
&\le 2(n-2)Cn^{-2}+\binom{n-2}{2}Cn^{-3}\le Cn^{-1},\\
\sum_{r=1}^N \|\Delta_{ir}f_H\|_\infty^2
&\le 2(n-2)Cn^{-4}+\binom{n-2}{2}Cn^{-6}\le Cn^{-3},
\end{split}
\]
which proves \eqref{stein:increment-sums}.

For the increment mean, the number of placements after fixing an oriented source root is exactly $\prod_{\ell=2}^{v-1}(n-\ell)$. Each remaining monomial contains $e-1$ distinct edges. By \eqref{loc:producteq} and \eqref{loc:meanprecision}, its expectation is $p^{e-1}+O(n^{-1})$, uniformly over the embeddings. Summing and using \eqref{loc:fallingcomparison}, we obtain
\begin{equation}
\mathbb E|\Hom(H,G_Y,i)|
=2e\left\{\prod_{\ell=2}^{v-1}(n-\ell)\right\}p^{e-1}
+O(n^{v-3})
=2ep^{e-1}n^{v-2}+O(n^{v-3}).
\label{stein:root-mean}
\end{equation}
When $e=1$, the remaining product is empty, and for $H=K_2$ the rooted count is exactly two. Substituting \eqref{stein:root-mean} into \eqref{stein:root-expression} gives mean error $O((n\sigma)^{-1})=O(n^{-2})$.

Applying \eqref{eq:poincare} to $\Delta_if_H$ and using the squared row-sum estimate in \eqref{stein:increment-sums}, we have
\begin{equation}
\operatorname{Var}(\Delta_if_H(Y))
\le\kappa\sum_{r=1}^N \mathbb E|\Delta_{ri}f_H(Y)|^2\le Cn^{-3}.
\label{stein:root-variance}
\end{equation}
This completes the proof of \eqref{stein:increment-moments}.

To prove \eqref{stein:linear}, let $c_i=\E\Delta_if_H(Y)$ so that subtracting the corresponding linear function centers each coordinate increment. Vertex permutations preserve the Gibbs measure and the count and act transitively on the edges, so $c_i$ is independent of $i$. The configuration function
\[
y\longmapsto f_H(y)-\sum_{i=1}^N c_i(y_i-\widetilde p)
\]
is centered and has $i$th difference $\Delta_if_H-c_i$. By \eqref{eq:poincare} and \eqref{stein:root-variance},
\begin{align}
\mathbb E\left|f_H(Y)-\sum_{i=1}^N c_i(Y_i-\widetilde p)\right|^2
\le\kappa\sum_{i=1}^N \mathbb E|\Delta_if_H(Y)-c_i|^2
=\kappa\sum_{i=1}^N \operatorname{Var}(\Delta_if_H(Y))
\le Cn^{-1}.\label{stein:linear-projection}
\end{align}
Recall $S=E-\mu$ from \eqref{resp:Rdefinition}. By \eqref{loc:edgevariance}, $\|S(Y)\|_2\le Cn$. Vertex symmetry and \eqref{stein:increment-moments} make $c_i-\sigma^{-1}$ a common $O(n^{-2})$ constant. Using \eqref{eq:edge-scale} and \eqref{stein:linear-projection}, we obtain
\[
\begin{aligned}
\|f_H(Y)-W_{\bbeta}\|_2
&\le\left\|f_H(Y)-\sum_{i=1}^N c_i(Y_i-\widetilde p)\right\|_2
+\left\|\sum_{i=1}^N (c_i-\sigma^{-1})(Y_i-\widetilde p)\right\|_2\\
&\le Cn^{-1/2}+Cn^{-2}\|S(Y)\|_2\le Cn^{-1/2}.
\end{aligned}
\]
This proves the first part of \eqref{stein:linear}. For the second part, centering, \eqref{eq:poincare} and \eqref{stein:increment-size} give
\[
\mathbb EV^2\le\kappa\sum_{i=1}^N \mathbb E|\Delta_if_H(Y)|^2
\le C Nn^{-2}\le C.
\]
\end{proof}

\begin{proof}[Proof of Lemma~\ref{stein:remainder}]
We begin with $R_H$, whose relation to the edge-count remainder allows us to use the estimates from Appendix~\ref{sec:local-responses}. We then estimate $K_H$ and combine the two comparisons to control the weighted sum in \eqref{stein:rh-weighted}. Recall from \eqref{resp:Rdefinition} that $S=E-\mu$ and $R=R_{K_2}$. By \eqref{stein:rh},
\[
R_H=R-\rho\sigma(f_H-S/\sigma),
\]
and therefore
\begin{equation}
\Delta_iR_H=\Delta_iR-\rho\sigma(\Delta_if_H-\sigma^{-1}).
\label{stein:rh-difference}
\end{equation}
Applying \eqref{resp:meanbounds}, \eqref{stein:increment-moments} and \eqref{eq:compact}--\eqref{eq:compact-scales} to \eqref{stein:rh-difference}, we obtain
\[
\begin{split}
|\mathbb E\Delta_iR_H(Y)|
&\le Cn^{-1}+C\sigma n^{-2}\le Cn^{-1},\\
\sqrt{\operatorname{Var}(\Delta_iR_H(Y))}
&\le\sqrt{\operatorname{Var}(\Delta_iR(Y))}
+\rho\sigma\sqrt{\operatorname{Var}(\Delta_if_H(Y))}
\le Cn^{-1/2}.
\end{split}
\]
The variance bound follows from the $L^2$ triangle inequality. The supremum estimate follows from \eqref{stein:rh-difference}, \eqref{resp:meanbounds} and \eqref{stein:increment-size}. Thus
\begin{equation}
|\mathbb E\Delta_iR_H(Y)|\le Cn^{-1},\quad
\operatorname{Var}(\Delta_iR_H(Y))\le Cn^{-1},\quad
\|\Delta_iR_H\|_\infty\le C.
\label{stein:rh-increment}
\end{equation}
By \eqref{eq:conditional}, \eqref{stein:rh} and $\E V=0$, we have $\E R_H=0$. Applying \eqref{stein:rh-increment} in \eqref{eq:poincare} gives
\[
\begin{aligned}
\E R_H(Y)^2&\le\kappa\sum_{i=1}^N \E|\Delta_iR_H(Y)|^2
=\kappa\sum_{i=1}^N\{\Var(\Delta_iR_H(Y))+|\E\Delta_iR_H(Y)|^2\}
\\&\le CN(n^{-1}+n^{-2})\le Cn.
\end{aligned}
\]
Moreover, $\E|\Delta_iR_H(Y)|\le\|\Delta_iR_H(Y)\|_2\le Cn^{-1/2}$. Summing proves \eqref{stein:rh-moments}.

We next turn to \eqref{stein:kh-moments}. The definition \eqref{stein:kh} expresses $K_H$ as a sum of products of $q_i(1-q_i)$ and $\Delta_if_H$. We first compute its mean, keeping the covariance between these two factors. Combining \eqref{loc:bernoulli}, \eqref{loc:meanprecision} and \eqref{resp:qvariance}, we obtain
\begin{equation}\label{stein:q-mean}
\E[q_i(Y)(1-q_i(Y))]
=\widetilde p(1-\widetilde p)-\Var(q_i(Y))
=p(1-p)+O(n^{-1}).
\end{equation}
The covariance of $q_i(Y)(1-q_i(Y))$ and $\Delta_if_H(Y)$ is at most $C(n^{-1}n^{-3})^{1/2}=Cn^{-2}$ by Cauchy--Schwarz, \eqref{resp:weightvariance} and \eqref{stein:increment-moments}. Splitting the expectation into a covariance and a product of means, then using \eqref{stein:q-mean} and \eqref{stein:increment-moments}, yields
\begin{align}
\mathbb E[q_i(Y)(1-q_i(Y))\Delta_if_H(Y)]
&=\mathbb E[q_i(Y)(1-q_i(Y))]\mathbb E\Delta_if_H(Y)+O(n^{-2})\notag\\
&=\{p(1-p)+O(n^{-1})\}\{\sigma^{-1}+O(n^{-2})\}
+O(n^{-2})\notag\\
&=\frac{p(1-p)}\sigma+O(n^{-2}).\label{stein:kh-mean}
\end{align}
Summing \eqref{stein:kh-mean} over $i$, dividing by $\rho\sigma$ as in \eqref{stein:kh}, and using \eqref{eq:edge-scale} yields
\[
\mathbb EK_H(Y)=\frac{Np(1-p)}{\rho\sigma^2}
+O\left(\frac{N}{\sigma n^2}\right)=1+O(n^{-1}).
\]

To control $\Var(K_H(Y))$, apply \eqref{resp:productrule} to each product in \eqref{stein:kh}. Since $u(1-u)$ is $1$-Lipschitz on $[0,1]$ and $q_i(1-q_i)\le1/4$, combining \eqref{stein:increment-size}, \eqref{stein:increment-sums} and \eqref{resp:rowsum}--\eqref{resp:matrixbounds}, we obtain
\begin{equation}
\|\Delta_rK_H\|_\infty
\le\frac1{\rho\sigma}\left\{
\frac Cn\sum_{i=1}^N \|\Delta_rq_i\|_\infty+
\frac14\sum_{i=1}^N \|\Delta_{ri}f_H\|_\infty\right\}
\le Cn^{-2}.
\label{stein:kh-difference}
\end{equation}
The $i=r$ term vanishes because both $q_i$ and $\Delta_if_H$ are independent of coordinate $i$. Applying \eqref{eq:poincare} with \eqref{stein:kh-difference}, we have
\[
\operatorname{Var}(K_H(Y))\le\kappa\sum_{r=1}^N \|\Delta_rK_H\|_\infty^2
\le C Nn^{-4}\le Cn^{-2}.
\]
For the pointwise part of \eqref{stein:kh-moments}, combine \eqref{stein:kh}, \eqref{stein:increment-size} and \eqref{eq:compact}--\eqref{eq:compact-scales} to obtain
\[
0\le K_H(y)\le\frac{N}{4\rho\sigma}\max_i\|\Delta_if_H\|_\infty\le C.
\]
This proves \eqref{stein:kh-moments}.

We have now controlled both $R_H$ and $K_H$. To prove the remaining estimate \eqref{stein:rh-weighted}, we return to the difference identity \eqref{stein:rh-difference} and sum it with weights $q_i(1-q_i)$. The definitions of $K_H$ and $K=K_{K_2}$ in \eqref{stein:kh} and \eqref{resp:Kdefinition} allow us to keep this sum intact, giving
\begin{align}
\sum_{i=1}^N q_i(1-q_i)\Delta_iR_H
&=\sum_{i=1}^N q_i(1-q_i)\Delta_iR+\rho^2\sigma^2(K-K_H).
\label{stein:weighted-identity}
\end{align}
Combining \eqref{stein:weighted-identity} with \eqref{resp:weightedbound}, \eqref{resp:Kbounds}, \eqref{stein:kh-moments} and \eqref{eq:compact}--\eqref{eq:compact-scales}, we obtain
\[
\begin{split}
\mathbb E\left|\sum_{i=1}^N q_i(Y)(1-q_i(Y))\Delta_iR_H(Y)\right|
&\le Cn+\rho^2\sigma^2\{\|K(Y)-1\|_2+\|K_H(Y)-1\|_2\}\\
&\le Cn+C\sigma^2/n\le Cn.
\end{split}
\]
This proves \eqref{stein:rh-weighted}, which will control the sum in \eqref{stein:second-identity}.
\end{proof}

The preceding lemma controls the second moment of $R_H$ and the entire weighted sum in \eqref{stein:rh-weighted}. We now use these bounds to estimate $\E[R_H(Y)f(V)]$ more precisely than its absolute-moment bound permits. The argument follows the repeated use of Stein equations by Fang and Liu \cite[Section 2.2]{FL2025}.

\begin{proof}[Proof of Lemma~\ref{stein:pairing}]
Let $G$ solve the normal Stein equation for $f$. Multiplying the equation by $R_H(Y)$ and taking expectations removes the constant term, since $R_H(Y)$ is centered. We can then use \eqref{stein:target-identity} to express the remaining expectation through coordinate differences of $R_H(G\circ f_H)$, which are controlled by Lemmas~\ref{stein:increments} and \ref{stein:remainder}. More explicitly, the equation is
\begin{equation}
xG(x)-G'(x)=f(x)-\mathbb Ef(Z).
\label{stein:auxiliary-equation}
\end{equation}
By \eqref{stein:smooth-bounds},
\begin{equation}
\|G\|_\infty\le2,\qquad\|G'\|_\infty\le2,\qquad
\operatorname{Lip}(G')\le2.
\label{stein:auxiliary-bounds}
\end{equation}
Since $\E R_H(Y)=0$ by \eqref{eq:conditional} and \eqref{stein:rh}, we multiply \eqref{stein:auxiliary-equation} by $R_H(Y)$ and take expectations to obtain
\begin{equation}
\mathbb E[R_H(Y)f(V)]
=\mathbb E[R_H(Y)VG(V)]-\mathbb E[R_H(Y)G'(V)].
\label{stein:pairing-centering}
\end{equation}
Thus adding a constant to $f$ does not change this expectation.

Using \eqref{stein:increment-size}, the endpoints \eqref{stein:forced-endpoints} and absolute continuity of $G$, we write
\begin{align}
\Delta_i(G\circ f_H)(y)&=\Delta_if_H(y)G'(f_H(y))+e_i(y),\label{stein:taylor}\\
e_i(y)&=\int_{-y_i\Delta_if_H(y)}^{(1-y_i)\Delta_if_H(y)}
\{G'(f_H(y)+u)-G'(f_H(y))\}\,du.\notag
\end{align}
Since $y_i\in\{0,1\}$, the integral runs over $[0,\Delta_if_H(y)]$ or $[-\Delta_if_H(y),0]$. Applying \eqref{stein:auxiliary-bounds} to this integral in \eqref{stein:taylor}, we obtain
\begin{equation}
|e_i(y)|\le 2\int_0^{\Delta_if_H(y)}u\,du
=\{\Delta_if_H(y)\}^2\le Cn^{-2}.
\label{stein:taylor-error}
\end{equation}

To apply \eqref{stein:target-identity} to $A=R_H(G\circ f_H)$, we also need its coordinate difference. By the product rule \eqref{resp:productrule}, we have
\begin{align}
\Delta_i\{R_H(G\circ f_H)\}(y)
&=\Delta_iR_H(y)G(f_H(y))\notag\\
&\quad+\{R_H(y)+(1-2y_i)\Delta_iR_H(y)\}
\Delta_i(G\circ f_H)(y).
\label{stein:product-rule}
\end{align}
This is \eqref{resp:productrule}, written relative to the current value for $y_i=0$ or $1$. The quantity in braces is $R_H$ at the opposite coordinate endpoint.

Take $A=R_H(G\circ f_H)$ in \eqref{stein:target-identity}. Combining \eqref{stein:product-rule} and \eqref{stein:taylor}, and using \eqref{stein:kh} to sum the terms containing $R_H\Delta_if_HG'$, we obtain
\begin{equation}
\begin{aligned}
\rho\sigma\mathbb E[R_H(Y)VG(V)]+\mathbb E[R_H(Y)^2G(V)]
&=\sum_{i=1}^N \mathbb E[q_i(Y)(1-q_i(Y))\Delta_iR_H(Y)G(V)]\\
&\quad+\rho\sigma\mathbb E[R_H(Y)K_H(Y)G'(V)]
+\rho\sigma\mathcal E_{H,f},
\end{aligned}
\label{stein:second-intermediate}
\end{equation}
where the remaining terms are
\begin{align}
\mathcal E_{H,f}
&=\frac1{\rho\sigma}\sum_{i=1}^N \mathbb E[
q_i(Y)(1-q_i(Y))(1-2Y_i)\Delta_iR_H(Y)\Delta_if_H(Y)G'(V)]\notag\\
&\quad+\frac1{\rho\sigma}\sum_{i=1}^N \mathbb E[
q_i(Y)(1-q_i(Y))\{R_H(Y)+(1-2Y_i)\Delta_iR_H(Y)\}e_i(Y)].
\label{stein:pairing-error}
\end{align}
Divide \eqref{stein:second-intermediate} by $\rho\sigma$ and substitute into \eqref{stein:pairing-centering}. The expectation containing $R_HK_HG'(V)$ combines with $-\E[R_HG'(V)]$, leaving the factor $K_H-1$ controlled by \eqref{stein:kh-moments}. Thus
\begin{align}
\mathbb E[R_H(Y)f(V)]
&=\frac1{\rho\sigma}\mathbb E\left[
\left\{\sum_{i=1}^N q_i(Y)(1-q_i(Y))\Delta_iR_H(Y)-R_H(Y)^2\right\}G(V)\right]
\notag\\
&\quad+\mathbb E[R_H(Y)(K_H(Y)-1)G'(V)]+\mathcal E_{H,f}.
\label{stein:second-identity}
\end{align}
To estimate \eqref{stein:second-identity}, first control the error $\mathcal E_{H,f}$ from the discrete product rule and Taylor expansion. For the two sums in \eqref{stein:pairing-error}, we use $q_i(1-q_i)\le1/4$, \eqref{stein:auxiliary-bounds}, \eqref{stein:increment-size} and \eqref{stein:taylor-error}. Combining these with $\E|R_H|\le C\sqrt n$ and $\sum_{i=1}^N\E|\Delta_iR_H|\le Cn^{3/2}$ from \eqref{stein:rh-moments}, and with $\rho\sigma\asymp n$ from \eqref{eq:compact}--\eqref{eq:compact-scales}, we obtain
\begin{equation}
\begin{aligned}
|\mathcal E_{H,f}|
\le\frac C{\sigma n}\sum_{i=1}^N \mathbb E|\Delta_iR_H(Y)|
+\frac C{\sigma n^2}\left\{N\mathbb E|R_H(Y)|+
\sum_{i=1}^N \mathbb E|\Delta_iR_H(Y)|\right\}
\le Cn^{-1/2}.
\end{aligned}\label{stein:pairing-error-bound}
\end{equation}
Finally, applying \eqref{stein:auxiliary-bounds}, Cauchy--Schwarz and \eqref{stein:pairing-error-bound} to \eqref{stein:second-identity} gives
\[
\begin{aligned}
|\E[R_H(Y)f(V)]|
&\le\frac{C}{\rho\sigma}
\left\{\E\left|\sum_{i=1}^N q_i(Y)(1-q_i(Y))\Delta_iR_H(Y)\right|
+\E R_H(Y)^2\right\}\\
&\quad+2\|R_H(Y)\|_2\|K_H(Y)-1\|_2+Cn^{-1/2}\\
&\le\frac{C}{n}(n+n)+C\sqrt n\,n^{-1}+Cn^{-1/2}\le C,
\end{aligned}
\]
where the last line uses \eqref{stein:rh-moments}--\eqref{stein:rh-weighted} and \eqref{eq:compact}--\eqref{eq:compact-scales}. This proves \eqref{stein:pairing-bound}, uniformly in the $1$-Lipschitz function $f$.
\end{proof}

We have proved the expectation estimate for the remainder. To complete the auxiliary bounds used for normal approximation, it remains to control configurations near an indicator threshold. We do so directly from the conditional identity and the nonnegative increments of $f_H$.

\begin{proof}[Proof of Lemma~\ref{stein:concentration}]
To estimate the contribution near the threshold in $I_1(F)$ of \eqref{stein:kol-decomposition}, we apply \eqref{stein:bernoulli} to a truncated linear function of $V$. We first bound the second moment of the sum $\sum_{i=1}^N\{Y_i-q_i(Y)\}$ that appears on the left of the summed identity. Using $\sum_{j=1}^N\{y_j-q_j(y)\}$ as the configuration function in \eqref{stein:bernoulli} and summing over $i$, we obtain
\begin{equation}\label{stein:conditional-sum-moment}
\begin{aligned}
\E\left[\left(\sum_{j=1}^N\{Y_j-q_j(Y)\}\right)^2\right]
&=\sum_{i=1}^N\E\left[q_i(Y)(1-q_i(Y))
\left\{1-\sum_{j=1}^N\Delta_iq_j(Y)\right\}\right]\\
&\le\frac14\sum_{i=1}^N
\left(1+\sum_{j=1}^N\|\Delta_iq_j\|_\infty\right)
\le CN.
\end{aligned}
\end{equation}
For the last inequality, combine \eqref{resp:rowsum}--\eqref{resp:matrixbounds} and the symmetry of $A_{ij}$. For $z\in\mathbb R$, define
\[
\ell_{z,h}(w)=\max\{-2h,\min\{w-z,2h\}\}.
\]
This function is increasing, $1$-Lipschitz and bounded in absolute value by $2h$. If $|f_H(y)-z|\le h$, then \eqref{stein:forced-endpoints} and \eqref{stein:increment-size} place both coordinate endpoints in $[z-2h,z+2h]$, where $\ell_{z,h}(w)=w-z$. Thus $\Delta_i(\ell_{z,h}\circ f_H)(y)=\Delta_if_H(y)$. For all other configurations this difference is nonnegative. By \eqref{stein:kh}, we therefore have
\[
\sum_{i=1}^Nq_i(y)(1-q_i(y))\Delta_i(\ell_{z,h}\circ f_H)(y)
\ge\rho\sigma K_H(y)\mathbf1_{\{|f_H(y)-z|\le h\}}.
\]
Taking $A=\ell_{z,h}\circ f_H$ in \eqref{stein:bernoulli} and summing over $i$, we now obtain
\[
\begin{aligned}
\E[K_H(Y)\mathbf1_{\{|V-z|\le h\}}]
&\le\frac1{\rho\sigma}\E\left[\ell_{z,h}(V)
\sum_{i=1}^N\{Y_i-q_i(Y)\}\right]\\
&\le\frac{2h}{\rho\sigma}
\left\{\E\left[\left(\sum_{i=1}^N\{Y_i-q_i(Y)\}\right)^2\right]\right\}^{1/2}
\le\frac{Ch\sqrt N}{\rho\sigma}\le Cn^{-1}.
\end{aligned}
\]
Here we used Cauchy--Schwarz, \eqref{stein:conditional-sum-moment}, $h=C_1/n$ and \eqref{eq:compact}--\eqref{eq:compact-scales}. This proves \eqref{stein:local-concentration}.
\end{proof}

\subsection{Proofs of Lemmas~\ref{md:path-input}--\ref{md:weighted}}\label{sec:tilting-lemmas}

With Lemmas~\ref{stein:increments}--\ref{stein:concentration} proved, the normal approximation and variance estimates of Theorems~\ref{thm:normal} and \ref{thm:signed} are now available for the tilt arguments. Recall from \eqref{md:path}--\eqref{md:Yt} that $U_t$ uses the current mean and scale, whereas $Y_t$ uses the base scale and a quadratic approximation to the tilted mean.

\begin{proof}[Proof of Lemma~\ref{md:path-input}]
We first justify the root and expansion in \eqref{md:center-root}, then verify the bounds along the full tilt interval. Let $m_3(t)=t+\gamma_{H,0}t^2/(2\sqrt N)$. Since $\gamma_{H,0}$ is uniformly bounded, for large $n$ we have
\[
|m_3'(t)-1|\le C_Ln^{-1/2}\le\tfrac12
\quad (|t|\le2L\sqrt n).
\]
Moreover, $m_3(-2L\sqrt n)<-L\sqrt n$ and $m_3(2L\sqrt n)>L\sqrt n$. Thus $m_3$ is strictly increasing on this interval and its image contains $[-L\sqrt n,L\sqrt n]$. By the intermediate value theorem, there is a unique solution $t_\epsilon$ there. Solving the quadratic and rationalizing its numerator, we obtain \eqref{md:center-root}. The radicand is uniformly bounded away from zero, so $\sgn(t_\epsilon)=\epsilon$ and $|t_\epsilon|\asymp x$, with $|t_\epsilon|\le2L\sqrt n$.

For the expansion in \eqref{md:center-root}, let $h=2\epsilon\gamma_{H,0}x/\sqrt N$. Then
\[
|h|\le C_Ln^{-1/2},\qquad
\frac{2}{1+\sqrt{1+h}}=1-\frac h4+O(h^2).
\]
Multiplying by $\epsilon x$ proves the expansion uniformly in the base parameter. When $\gamma_{H,0}=0$, the formula reduces to $t_\epsilon=\epsilon x$.

We next verify the probability and scale bounds on the full interval \eqref{md:allowed-path}. Let $u(t)=tn^{v_H-2}/\sigma_{H,0}$ be the displacement in \eqref{md:path}. By \eqref{eq:subgraph-scale} and \eqref{eq:compact}--\eqref{eq:compact-scales}, $|u(t)|\le C_Ln^{-1/2}$. As in \eqref{info:signed-path}, an existing energy coordinate can be varied by adding a copy of its template.

For a compact base family $\B\subset\mathcal D$, let $\delta=1-\sup_{\bbeta\in\B}D(\bbeta)>0$. Definition~\ref{def:parameter-regions}(ii) gives
\begin{equation}\label{md:compact-existing}
D(\bbeta(t))\le D(\bbeta_0)+\tfrac12e(e-1)|u(t)|.
\end{equation}
By \eqref{md:compact-existing}, a fixed small interval of displacements lies in a compact subset of $\mathcal D$ for the enlarged model. Theorem~\ref{thm:normal} therefore applies uniformly along this interval in either direction.

For a compact base family in $\mathcal S_+$, apply Theorem~\ref{thm:signed} to the added graph $H$. Both signs lie in its fixed signed neighborhood for large $n$. More generally, if the base is $\btheta=(\bbeta,\bh)\in\B_{\mathcal J}(\eta)$, append $H$ to the added family. Then
\begin{equation}\label{md:compact-new}
\sum_{\ell=1}^m e_{J_\ell}^2|h_\ell|+e^2|u(t)|
\le\eta+e^2C_Ln^{-1/2}\le2\eta
\end{equation}
for all $n$ beyond a common threshold. The radius in Lemma~\ref{lem:uniform-input} is unchanged when $H$ is appended. Thus \eqref{md:compact-new} keeps both tilt directions in the outer family, where Theorem~\ref{thm:signed} applies uniformly.

In each case, let $\B_H$ be the compact family just constructed. By Theorem~\ref{thm:normal} or \ref{thm:signed}, as appropriate, we obtain the distance and relative variance estimates in Lemma~\ref{md:path-input}. Combining \eqref{md:scale-ratio} with \eqref{eq:compact}--\eqref{eq:compact-scales} on $\B_H$, we also have $c\le a_H(t)\le C$.
\end{proof}

\begin{proof}[Proof of Lemma~\ref{md:cgf}]
We express $M_H''(t)$ as the tilted variance divided by $\sigma_{H,0}^2$. To obtain the cubic term, we expand the explicit squared scale ratio $a_H(t)^2$ in \eqref{md:scale-ratio}, then use Lemma~\ref{md:path-input} to transfer the expansion to $M_H''(t)$. Integration yields \eqref{md:cgfexpansion}, without differentiating the variance approximation error.

All means and scales are as in \eqref{eq:subgraph-scale}. By \eqref{eq:model} and \eqref{md:path}, the energy increment is
\[
T_{\bbeta(t)}(y)-T_{\bbeta_0}(y)=t|\Hom(H,G_y)|/\sigma_{H,0}.
\]
Consequently,
\begin{equation}\label{md:partition-cgf}
M_H(t)=\log Z(\bbeta(t))-\log Z(\bbeta_0)
-t\mu_{H,0}/\sigma_{H,0}.
\end{equation}
The configuration space is finite. Differentiating termwise, for any fixed configuration function $A$, we have
\begin{equation}\label{md:covariance-derivative}
\frac{d}{dt}\mathbb E A(Y_{\bbeta(t)})
=\operatorname{Cov}\!\left(A(Y_{\bbeta(t)}),
f_{H,\bbeta_0}(Y_{\bbeta(t)})\right).
\end{equation}
Taking $A=f_{H,\bbeta_0}$ in \eqref{md:covariance-derivative} and using \eqref{md:partition-cgf}, we obtain the exact identities
\begin{equation}\label{md:exacttilt}
\begin{gathered}
\frac{p_{\bbeta(t)}(y)}{p_{\bbeta(0)}(y)}
=e^{t f_{H,\bbeta_0}(y)-M_H(t)},\qquad
M_H'(t)
=\frac{\mathbb E|\Hom(H,G_{Y_{\bbeta(t)}})|-\mu_{H,0}}
{\sigma_{H,0}},\\
M_H''(t)
=\frac{\operatorname{Var}(|\Hom(H,G_{Y_{\bbeta(t)}})|)}
{\sigma_{H,0}^2}.
\end{gathered}
\end{equation}
In particular, $M_H(0)=M_H'(0)=0$. The density ratio in \eqref{md:exacttilt} also shows that, for every configuration function $A$, we have
\begin{equation}\label{md:weighted-moment-transfer}
\frac{\E[A(Y_{\bbeta_0})e^{tW_{H,\bbeta_0}}]}
{\E e^{tW_{H,\bbeta_0}}}
=\E A(Y_{\bbeta(t)}).
\end{equation}
Thus the uniform moment bounds apply to the weighted expectations in \eqref{md:weighted-moment-transfer}. For fixed $t$, taking $A(y)=\{f_{H,\bbeta_0}(y)-M_H'(t)\}^2$ gives $M_H''(t)$.

We next compute the derivatives of $a_H(t)^2$ in \eqref{md:scale-ratio}. Recall from \eqref{md:path} that only coordinate $j$, corresponding to $H$, varies along the tilt. In the parameter derivatives below, let $p=p_{\bbeta}$, $\rho=\rho_{\bbeta}$ and $v=v_{\bbeta}$, with $v$ denoting the variance density rather than the number of vertices of $H$. Since the $j$th graph is $H$,
\begin{equation}\label{md:actualphi}
\partial_{\beta_j}\Phi_{\bbeta}(u)=eu^{e-1},\qquad
\rho=1-2p(1-p)\Phi_{\bbeta}'(p),\qquad
v=\frac{p(1-p)}{\rho},
\end{equation}
where $p=\Psi(2\Phi_{\bbeta}(p))$. The local fixed-point argument leading to \eqref{eq:compact}, applied to the compact family $\B_H$ from Lemma~\ref{md:path-input}, makes $p_{\bbeta}$ smooth with uniformly bounded derivatives of the required orders. Moreover, $p$, $1-p$ and $\rho$ are bounded away from zero. These bounds are uniform throughout the enlarged parameter family, including at zero values of the added coefficients.

Differentiate the fixed-point equation $p=\Psi(2\Phi_{\bbeta}(p))$ and the inverse-variance identity $v^{-1}=1/\{p(1-p)\}-2\Phi_{\bbeta}'(p)$ from \eqref{md:actualphi}. Since $\Psi'(z)=\Psi(z)(1-\Psi(z))$, we obtain
\begin{equation}\label{md:response-intermediate}
\begin{aligned}
\partial_{\beta_j}p
&=2p(1-p)\{ep^{e-1}+\Phi_{\bbeta}'(p)\partial_{\beta_j}p\},\\
\partial_{\beta_j}(v^{-1})
&=-\left\{\frac{1-2p}{p^2(1-p)^2}+2\Phi_{\bbeta}''(p)\right\}
\partial_{\beta_j}p-2e(e-1)p^{e-2}.
\end{aligned}
\end{equation}
Solve the first line of \eqref{md:response-intermediate} for $\partial_{\beta_j}p$ and use $\partial_{\beta_j}v=-v^2\partial_{\beta_j}(v^{-1})$ in the second. With \eqref{eq:gamma}, we obtain
\begin{equation}\label{md:response}
\begin{gathered}
\partial_{\beta_j}p=2ep^{e-1}v,\qquad
\partial_{\beta_j}v=2ep^{e-1}v^{3/2}\gamma_{\bbeta}+2e(e-1)p^{e-2}v^2,\\
\partial_{\beta_j}\log(p^{2e-2}v)
=2ep^{e-1}\sqrt v
\left\{\gamma_{\bbeta}+3(e-1)\frac{\sqrt v}{p}\right\}.
\end{gathered}
\end{equation}
For the last line of \eqref{md:response}, substitute the first two into
\[
\partial_{\beta_j}\log(p^{2e-2}v)
=(2e-2)\frac{\partial_{\beta_j}p}{p}+\frac{\partial_{\beta_j}v}{v}.
\]
The terms proportional to $e-1$ sum to $6e(e-1)p^{e-2}v$, which yields the factor $3(e-1)\sqrt v/p$ in \eqref{md:response}, also when $e=1$.

Using \eqref{md:response} and the path derivative from \eqref{md:path}, with the coefficient defined in \eqref{md:gamma}, we obtain
\begin{equation}\label{md:scalederivative}
(a_H^2)'(0)=\frac{\gamma_{H,0}}{\sqrt N},\qquad
\sup_{\substack{\epsilon\in\{-1,1\}\\0\le u\le2L\sqrt n}}
|(a_H^2)''(\epsilon u)|\le\frac C N.
\end{equation}
The first identity in \eqref{md:scalederivative} identifies the cubic coefficient through the first-order change of the squared scale ratio. For the second bound, the path \eqref{md:path} is affine and only coordinate $j$ varies. Differentiating \eqref{md:scale-ratio} twice gives
\begin{equation}\label{md:scale-second-derivative}
(a_H^2)''(t)=
\frac{\left.\partial_{\beta_j}^2(p_{\bbeta}^{2e-2}v_{\bbeta})\right|_{\bbeta=\bbeta(t)}}
{p_{\bbeta_0}^{2e-2}v_{\bbeta_0}}
\left(\frac{n^{v_H-2}}{\sigma_{H,0}}\right)^2.
\end{equation}
In \eqref{md:scale-second-derivative}, the parameter derivatives are uniformly bounded, the denominator is bounded away from zero, and the final square is $O(N^{-1})$ by \eqref{md:path}. This proves the second part of \eqref{md:scalederivative}. Differentiating \eqref{md:scale-ratio} once also gives
\begin{equation}\label{md:scale-first-derivative}
(a_H^2)'(t)=
\frac{\left.\partial_{\beta_j}(p_{\bbeta}^{2e-2}v_{\bbeta})
\right|_{\bbeta=\bbeta(t)}}
{p_{\bbeta_0}^{2e-2}v_{\bbeta_0}}
\frac{n^{v_H-2}}{\sigma_{H,0}}
=O(N^{-1/2})
\end{equation}
uniformly along the tilt paths. Using \eqref{md:scalederivative}, \eqref{md:scale-first-derivative} and the positive lower bound in Lemma~\ref{md:path-input}, differentiation of the square root yields
\begin{equation}\label{md:scale-square-root-second}
a_H''(t)
=\frac{(a_H^2)''(t)}{2a_H(t)}
-\frac{\{(a_H^2)'(t)\}^2}{4a_H(t)^3}
=O(N^{-1}).
\end{equation}
At $t=0$, we have $a_H(0)=1$ and, by \eqref{md:scalederivative}, $a_H'(0)=\gamma_{H,0}/(2\sqrt N)$. Taylor's theorem, with \eqref{md:scalederivative} and \eqref{md:scale-square-root-second}, now proves both expansions in \eqref{md:scaleexpansion}.

We now transfer the scale expansion to $M_H''$. Combining \eqref{md:exacttilt}, \eqref{md:scale-ratio} and the variance bound in Lemma~\ref{md:path-input}, we obtain
\begin{equation}\label{md:variance-input}
|M_H''(t)-a_H(t)^2|
=a_H(t)^2\left|\frac{\operatorname{Var}(|\Hom(H,G_{Y_{\bbeta(t)}})|)}
{\sigma_{H,\bbeta(t)}^2}-1\right|
\le C/n.
\end{equation}
Combining \eqref{md:variance-input} with \eqref{md:scaleexpansion} proves the first line of \eqref{md:cgfexpansion}. To integrate for either sign of $t$, put
\[
R_2(t)=M_H''(t)-1-\gamma_{H,0}t/\sqrt N,
\qquad |R_2(t)|\le C(n^{-1}+t^2/N).
\]
Since \eqref{md:exacttilt} gives $M_H'(0)=M_H(0)=0$, integration with the usual orientation for either sign of $t$ yields
\begin{equation}\label{md:cgf-integrals}
\begin{split}
M_H'(t)-t-\frac{\gamma_{H,0}t^2}{2\sqrt N}
=\int_0^t R_2(s)\,ds,\quad
M_H(t)-\frac{t^2}{2}-\frac{\gamma_{H,0}t^3}{6\sqrt N}
=\int_0^t(t-s)R_2(s)\,ds.
\end{split}
\end{equation}
Let $u=|t|$. By the first line of \eqref{md:cgfexpansion}, the two remainders in \eqref{md:cgf-integrals} are at most $C\int_0^u(n^{-1}+s^2/N)\,ds$ and $C\int_0^u(u-s)(n^{-1}+s^2/N)\,ds$, respectively. Evaluating the two integrals gives $C\{u/n+u^3/(3N)\}$ and $C\{u^2/(2n)+u^4/(12N)\}$, respectively, and proves the remaining lines of \eqref{md:cgfexpansion}.
\end{proof}

\begin{proof}[Proof of Lemma~\ref{md:tilted-normal}]
By \eqref{md:path}, \eqref{md:exacttilt} and \eqref{md:Yt}, we have
\begin{equation}\label{md:affine-tilt}
Y_t=a_H(t)U_t+\delta_t.
\end{equation}

The identity \eqref{md:affine-tilt} separates the approximation of $U_t$ from the change in its center and scale. Since $a_H(t)>0$, the event $\{Y_t\le y\}$ equals $\{U_t\le(y-\delta_t)/a_H(t)\}$. The first part of \eqref{md:tiltBEexact} therefore follows from Lemma~\ref{md:path-input}, and the second follows from the estimate of $M_H'$ in \eqref{md:cgfexpansion}.

To compare with the standard normal distribution, use the positive bounds on $a_H(t)$ in Lemma~\ref{md:path-input}. Translation is controlled by $\|\phi\|_\infty=(2\pi)^{-1/2}$. For a scale $s$ between $1$ and $a_H(t)$, differentiation gives $|\partial_s\Phi(y/s)|=|y/s|\phi(y/s)/s$, uniformly bounded because $s$ is bounded away from zero and $\sup_z|z|\phi(z)<\infty$. Thus
\begin{equation}\label{md:normalshifts}
\begin{aligned}
\sup_y\left|\Phi\!\left(\frac{y-\delta_t}{a_H(t)}\right)-\Phi(y)\right|
&\le\sup_y\left|\Phi\!\left(\frac{y-\delta_t}{a_H(t)}\right)
-\Phi\!\left(\frac y{a_H(t)}\right)\right|+\sup_y\left|\Phi\!\left(\frac y{a_H(t)}\right)-\Phi(y)\right|\\
&
\le C\{|\delta_t|+|a_H(t)-1|\}.
\end{aligned}
\end{equation}
Combining \eqref{md:tiltBEexact}, \eqref{md:normalshifts} and \eqref{md:scaleexpansion} proves the first bound in \eqref{md:tiltBE}.

For the last inequality in \eqref{md:tiltBE}, use $N\ge n^2/4$ from \eqref{eq:edges}, so $|t|/\sqrt N\le2|t|/n$. If $|t|\le1$, then $t^2/N\le C/n$, whereas if $|t|\ge1$, then $t^2\le|t|^3$. This proves the last inequality in \eqref{md:tiltBE}.
\end{proof}

\begin{proof}[Proof of Lemma~\ref{md:weighted}]
We first prove the bounds on $Q$ by using its integral representation in \eqref{md:Q}. For $0\le u\le(1+z)^{-1}$, we have $u\le1$ and $zu\le1$, so $e^{-zu}\phi(u)\ge e^{-1}\phi(1)$. Integration over this interval gives the lower bound in \eqref{md:Qbounds}. For $z\ge1$,
\[
Q(z)\le\phi(0)\int_0^\infty e^{-zu}\,du=\phi(0)/z,
\quad
|Q'(z)|\le\phi(0)\int_0^\infty ue^{-zu}\,du=\phi(0)/z^2.
\]
Differentiation under the integral is justified by the integrable bound $u\phi(u)$. For $0\le z\le1$, use $Q(z)\le1/2$ and $|Q'(z)|\le\int_0^\infty u\phi(u)\,du$. This proves \eqref{md:Qbounds} with common constants for all $z\ge0$.

To prove \eqref{md:weightedbound}, we write the weighted expectation as an integral of strict tail probabilities and then apply the Kolmogorov bound. If $z>0$, the identity $1-e^{-zu}=\int_0^u ze^{-zy}\,dy$ gives
\[
e^{-zu}\mathbf1_{\{u>0\}}
=\mathbf1_{\{u>0\}}-z\int_0^\infty e^{-zy}\mathbf1_{\{u>y\}}\,dy.
\]
Both sides vanish for $u\le0$. The integral has a nonnegative integrand, so Tonelli's theorem yields
\begin{equation}\label{md:strict-integral}
\mathbb E[e^{-zU}\mathbf1_{\{U>0\}}]
=\mathbb P(U>0)-z\int_0^\infty e^{-zy}\mathbb P(U>y)\,dy.
\end{equation}
By \eqref{eq:distances}, $d_{\rm Kol}(U,Z)\le d$ implies $|\Pp(U>y)-\Pp(Z>y)|\le d$ for every real $y$. Apply \eqref{md:strict-integral} to $U$ and $Z$. The probability terms on its right-hand side differ by at most $d$, and the weighted integrals differ by at most $dz\int_0^\infty e^{-zy}\,dy=d$. The normal weighted expectation equals $Q(z)$ by \eqref{md:Q}, proving \eqref{md:weightedbound}. For $z=0$, compare $\Pp(U>0)$ directly with $1/2$.

Finally, if $F_U$ is the distribution function of $U$, then $F_{-U}(y)=1-F_U((-y)-)$. Letting $w\uparrow-y$ in $|F_U(w)-\Phi(w)|\le d$ and using continuity and symmetry of $\Phi$ proves the reflection bound. The use of left limits retains all possible atoms.
\end{proof}

\section*{Acknowledgments}
\textbf{Use of generative AI.} During the preparation of this manuscript, the authors used ChatGPT for language editing, LaTeX assistance, and preliminary checks of bibliographic information. The mathematical results and arguments are the authors' own. The authors reviewed, edited, and independently verified all content and references and take full responsibility for the manuscript.

\section*{Funding}
S.-H. Liu was supported in part by the Fundamental Research Funds for the Central Universities (No. DUT25RC(3)133).

Dawei Lu was supported by the National Natural Science Foundation of China (No. 12371154) and the National Key R\&D Program of China (No. 2023YFA1009200).

\end{document}